\documentclass[11pt, reqno]{amsart}
\usepackage{amsmath, amsthm, amscd, amsfonts, amssymb, graphicx, color, mathtools, mathrsfs}
\usepackage[bookmarksnumbered, colorlinks, plainpages]{hyperref}
\usepackage{enumerate}
\usepackage{setspace}
\usepackage{multicol}
\usepackage[margin=1in]{geometry}
\usepackage{comment}
\usepackage{booktabs}
\usepackage{caption}
\usepackage{subcaption}
\usepackage{dsfont}
\usepackage{tikz-cd}
\usepackage{algorithmic}
\usepackage{pgfplots}

\usepackage{dsfont}
\usepackage{tikz-cd}
\usepackage{algorithmic}

\usepackage[normalem]{ulem}
\newcommand\tsout{\bgroup\markoverwith{\textcolor{red}{\rule[0.5ex]{2pt}{1.4pt}}}\ULon}
\newcommand{\stkout}[1]{\ifmmode\text{\tsout{\ensuremath{#1}}}\else\tsout{#1}\fi}
\allowdisplaybreaks

\theoremstyle{definition}
\newtheorem{theorem}{Theorem}[section]
\newtheorem{lemma}[theorem]{Lemma}
\newtheorem{proposition}[theorem]{Proposition}

\newtheorem{definition}[theorem]{Definition}

\numberwithin{equation}{section}

\newcommand{\curl}{\mathrm{curl}\,}

\newcommand{\divg}{\mathrm{div}\,}
\newcommand{\bff}{\boldsymbol}
\newcommand{\bb}{\mathbb}

\newcommand{\du}{\mathrm{d}\bff{u}}
\newcommand{\dB}{\mathrm{d}\bff{B}}
\newcommand{\dt}{\,\mathrm{d}t}

\newcommand{\dx}{\,\mathrm{d}x}
\newcommand{\ds}{\,\mathrm{d}s}

\newcommand{\dbeta}{\mathrm{d}\beta}

\newcommand{\norm}[2]{\left\|{#1}\right\|_{#2}}
\newcommand{\inpro}[2]{\left\langle#1,#2\right\rangle}

\newcommand{\abs}[1]{\left|{#1}\right|}

\newcommand{\hdiv}{\mathbb{H}(\mathrm{div})}
\newcommand{\hcurl}{\mathbb{H}(\mathrm{curl})}
\newcommand{\hzerodiv}{\mathbb{H}_0(\mathrm{div})}
\newcommand{\hzerocurl}{\mathbb{H}_0(\mathrm{curl})}
\newcommand{\curlh}{\mathrm{curl}_h\,}

\begin{document}
\setcounter{page}{1}

\title[Stochastic Hall--MHD: weak solutions via a structure-preserving FEM]{Stochastic resistive Hall--MHD with current fluctuations: martingale weak solutions via a convergent structure-preserving finite element method}

\author[Agus L. Soenjaya]{Agus L. Soenjaya}
\address{Institute of Analysis and Scientific Computing, TU Wien, Wiedner Hauptstrasse 8--10, 1040 Vienna, Austria}
\email{\textcolor[rgb]{0.00,0.00,0.84}{agus.soenjaya@asc.tuwien.ac.at}}

\thanks{\textbf{Acknowledgment.} This research was initiated while the author was supported by the Commonwealth through an Australian Government Research Training Program (RTP) Scholarship (\href{https://doi.org/10.82133/C42F-K220}{DOI: 10.82133/C42F-K220}). The author has also been supported by an Australian Mathematical Society Lift-off Fellowship. The work was completed during the author's postdoctoral appointment at the Institute of Analysis and Scientific Computing of TU Wien, supported by the Austrian Science Fund (FWF) through the international project I6802 ``Functional error estimates for PDEs on unbounded domains'' (\href{https://doi.org/10.55776/I6802}{DOI: 10.55776/I6802}), led by Prof. Dirk Praetorius.}

\begin{abstract}
We study the stochastic resistive Hall--magnetohydrodynamic (Hall--MHD) system on bounded convex polyhedral domains, subject to nonlinear perfectly conducting boundary conditions. The system is driven by multiplicative Gaussian forcing in the momentum equation together with structured $\curl$-type noise in the induction equation arising from stochastic perturbations of the resistive and Hall parts in the generalised Ohm law. We develop a fully discrete, linearly implicit, structure-preserving mixed finite element method based on a compatible discrete de Rham complex, which preserves the magnetic Gauss law exactly. We prove subsequential convergence of the discrete solutions to a finite-energy weak martingale solution, thereby obtaining a constructive existence result for the stochastic Hall--MHD system with $\curl$-type noise. Numerical experiments illustrate the stochastic dynamics, exact divergence preservation, and convergence of the method.
\end{abstract}
\maketitle


\section{Introduction}
Magnetohydrodynamics (MHD) provides a continuum description of the interaction between the motion of an electrically conducting fluid, such as a plasma or a liquid metal, and an electromagnetic field. In the incompressible resistive regime, the MHD system couples the Navier--Stokes equations for the fluid velocity with the Maxwell equations for the electromagnetic variables, closed by a constitutive Ohm law; see, for example, \cite{Dav17, SerTem83}.
Stochastic MHD models incorporate random external forcing and unresolved fluctuations arising from turbulent or imperfectly observed processes. The resulting equations provide a probabilistic description of the possible fluid and magnetic-field evolutions and have been studied extensively from the viewpoint of well-posedness and ergodicity~\cite{BarPra07, Mot22, SriSun99,  Yam19}, often with simplified boundary conditions.

Classical single-fluid MHD is appropriate when the ion and electron components of the plasma move approximately together. At length scales smaller than the ion inertial length, however, the relative drift between the ion and electron velocities becomes significant. This is highly relevant to small-scale plasma dynamics, including magnetic reconnection and solar plasmas~\cite{SimCha08}. The resulting Hall effect is incorporated through a strongly nonlinear term in the generalised Ohm law, having the same differential order as magnetic diffusion. Consequently, Hall--MHD is a strongly coupled quasilinear parabolic-dispersive system subject to geometric constraints~\cite{AchDegFroLiu11, ChaDegLiu14, Dai21}.

\subsection{Problem introduction}

In this work, we consider a stochastic resistive Hall--MHD model on a bounded convex polyhedral domain $\mathscr{D}\subset\mathbb{R}^3$ over a time interval $[0,T]$. Randomness enters the momentum equation through a multiplicative stochastic force, and the induction equation through a multiplicative curl-type stochastic electromotive contributions. More precisely, we study
\begin{subequations}\label{equ:stoch mhd}
	\begin{alignat}{2}
		&\du+\big[-\nu\Delta\bff{u}+(\bff{u}\cdot\nabla)\bff{u}+\nabla p-\bff{J}\times\bff{B}\big]\dt
		=
		\bff{f}\dt+\sum_{i=1}^{\infty}\bff{g}_i(\bff{u},\bff{B})\,\dbeta_i^u(t)
		\; && \quad\text{in }(0,T)\times\mathscr{D},
		\label{equ:stoch mhd u}
		\\
		&\dB+\curl\bff{E}\,\dt
		=
		-\sum_{i=1}^{\infty}\curl\bigl(\sigma\bff{\chi}_i+\eta\bff{\chi}_i\times\bff{B}\bigr)\,\dbeta_i^B(t)
		\; && \quad\text{in }(0,T)\times\mathscr{D},
		\label{equ:stoch mhd B}
		\\[1ex]
		&\bff{E}= \sigma\bff{J}+\eta\bff{J}\times\bff{B} -\bff{u}\times\bff{B}
		\; && \quad\text{in }(0,T)\times\mathscr{D},
		\label{equ:stoch mhd ohm}
		\\[1ex]
		&\bff{J}-\curl\bff{B}=\bff{0}
		\; && \quad\text{in }(0,T)\times\mathscr{D},
		\label{equ:stoch mhd ampere}
		\\[1ex]
		&\divg\bff{u}=\divg\bff{B}=0
		\; && \quad\text{in }(0,T)\times\mathscr{D},
		\label{equ:stoch mhd div}
		\\[1ex]
		&\bff{u}(0)=\bff{u}_0,\qquad \bff{B}(0)=\bff{B}_0
		\; && \quad\text{in }\mathscr{D},
		\label{equ:stoch mhd init}
		\\[1ex]
		&\bff{u}=\bff{0},\qquad \bff{B}\cdot\bff{n}=0,\qquad \bff{E}\times\bff{n}=\bff{0}
		\; && \quad\text{on }\partial\mathscr{D}.
		\label{equ:stoch mhd boundary}
	\end{alignat}
\end{subequations}
Here, $\bff{u}$ and $p$ denote the bulk-fluid velocity and pressure, while $\bff{B}$, $\bff{E}$, and $\bff{J}$ denote the magnetic induction, electric field, and current density, respectively. The parameters $\nu>0$, $\sigma>0$, and $\eta\geq0$ are the kinematic viscosity, electrical resistivity, and Hall coefficient, respectively. The processes $\{\beta_i^u\}_{i\in \bb{N}}$ and $\{\beta_i^B\}_{i\in \bb{N}}$ are independent families of real-valued Brownian motions on a filtered probability space. The coefficients $\bff{g}_i(\bff{u},\bff{B})$ in \eqref{equ:stoch mhd u} describe a state-dependent random force acting on the bulk plasma momentum. The conditions depicted in \eqref{equ:stoch mhd boundary} are the so-called perfectly conducting boundary conditions, which are also nonlinear (in $\bff{B}$) in our case.

The magnetic noise in \eqref{equ:stoch mhd B} is introduced at the level of the electric-field closure. For fixed $\bff{B}$, the current-dependent part of the deterministic Ohm law is the linear map $\bff{j}\mapsto \sigma\bff{j}+\eta\bff{j}\times\bff{B}$.
The prescribed fields $\{\bff{\chi}_i\}_{i\in \bb{N}}$ determine current-like spatial directions in which this constitutive response is randomly perturbed. Accordingly, the $i$th stochastic electric-field mode is chosen as
\[
\bff{\Xi}_i(\bff{B})
:=
\sigma\bff{\chi}_i
+\eta\bff{\chi}_i\times\bff{B}.
\]
The first component represents a fluctuating resistive contribution, while the second has the structure of a fluctuating Hall or electron-drift contribution. This ansatz may thus be viewed as a stochastic parameterisation of unresolved electromotive effects in the generalised Ohm law~\eqref{equ:stoch mhd ohm}. Such current fluctuation is motivated by the importance of
small-scale current and Hall effects in turbulent plasmas; see, e.g.,
\cite{KanVasCer06, MinAlePou07, MinGomMah03}.

Since both the deterministic and stochastic electric-field contributions enter the induction equation under the $\curl$ operator, the magnetic solenoidality constraint is propagated formally:
\[
\divg\bff{B}_0=0
\quad\Longrightarrow\quad
\divg\bff{B}(t)=0, \quad \forall t\in [0,T],
\]
an important physical requirement depicting the absence of magnetic monopoles. Preserving this magnetic solenoidality constraint (Gauss' law) is central to the numerical method developed in this paper.
The fluid pressure enforces an independent incompressibility constraint $\divg\bff{u}=0$.

\subsection{Literature review}

Despite its physical importance, the mathematical literature on the Hall--MHD system is relatively scarce. Mathematical analysis of \emph{deterministic} Hall--MHD with periodic boundary was initiated in \cite{AchDegFroLiu11}, where a kinetic formulation and global weak solutions were established. This is further extended by~\cite{Tan21} to include the case of perfectly conducting boundary and density-dependent viscosity and resistivity. Local well-posedness and small-data global results for strong solutions on $\bb{R}^3$ were subsequently investigated in~\cite{ChaDegLiu14, Dai21}; see also the references therein. 

In the \emph{stochastic} setting, \cite{Yam17} established global martingale solutions for three- and two-and-a-half-dimensional Hall--MHD driven by Gaussian multiplicative noise, while~\cite{Mot23} treated the problem on $\bb{R}^3$ using stochastic compactness in non-metrisable path spaces. These results demonstrate that global martingale solutions are a natural solution concept in three dimensions, where uniqueness in the finite-energy class remains unavailable. Notably, the case of curl-type multiplicative noise with perfectly conducting (nonlinear) boundary conditions has never been studied before.

On the numerical aspect, there is a rather extensive mathematical literature on stable, convergent, and structure-preserving numerical methods for \emph{deterministic} MHD \emph{without} the Hall term. Compatible mixed finite element methods have been developed to preserve magnetic Gauss' law, discrete energy identities, and other invariants; see, e.g., \cite{VeiGomPerZam26, VeiHuMas25, HuMaXu17, HipLiMaoZhe18, MaoXi25}. Rigorous numerical analysis for \emph{deterministic} Hall--MHD is considerably more limited. In~\cite{LaaHuFar23}, a structure-preserving and helicity-conserving finite element scheme for the stationary resistive Hall--MHD, together with well-posed linearisation and solver analysis, was developed. The work~\cite{GolSoeTra26} proposed and analysed a structure-preserving mixed finite element scheme for the non-stationary case.

By contrast, rigorous numerical analysis for \emph{stochastic} MHD remains extremely limited. To the best of our knowledge, \cite{DeuFotTch25} is the only existing work that develops a finite element approximation for the standard stochastic MHD system with multiplicative Gaussian noise and proves convergence to a weak martingale solution. However, the scheme considered there is nonlinear and does not preserve the magnetic Gauss law exactly at the discrete level. No previous work appears to have established convergence of a fully discrete structure-preserving finite element method for stochastic MHD. In particular, there is no corresponding convergence analysis for stochastic Hall--MHD, nor for models driven by current-type stochastic electromotive forcing. The present work therefore provides, to the best of our knowledge, the first convergence result of this kind.

\subsection{Contributions of this paper}

The purpose of this work is twofold. First, we establish, constructively, the existence of a weak martingale solution to the stochastic resistive Hall--MHD system with $\curl$-type Gaussian current noise on bounded convex polyhedra. Second, we develop a structure-preserving mixed finite element method for approximating such solutions. The scheme is linearly implicit and preserves the magnetic Gauss law exactly at the discrete level. It is based on compatible finite element spaces forming a discrete de Rham complex. The velocity--pressure pair is chosen to satisfy the discrete incompressibility condition, while the electric field, current density, and magnetic induction are approximated in suitable conforming $\hcurl$- and $\hdiv$-spaces.

Convergence of the scheme to a weak martingale solution is established in several steps. First, we derive moment bounds for the discrete variables. Using these estimates, stochastic time-translate bounds, and compactness arguments, we prove tightness of laws of the discrete solutions in suitable path spaces. The Jakubowski--Skorokhod representation theorem then yields almost-sure convergence on a new probability space. We subsequently identify the limits and show that every convergent subsequence determines a weak martingale solution of \eqref{equ:stoch mhd}. Consequently, the numerical approximation provides a constructive existence proof for a weak martingale solution satisfying the perfectly conducting boundary conditions in the appropriate trace sense. Since uniqueness of finite-energy martingale solutions is not known, the convergence result is subsequential, and no convergence rate is asserted.

The remainder of the paper is organised as follows. Section~\ref{subsec:notation} introduces the relevant notation and function spaces. In Section~\ref{subsec:assumptions-noise}--\ref{subsec:weak mart sol}, we present the assumptions on the data and noise coefficients, as well as the definition of a weak martingale solution. Section~\ref{sec:fem} introduces the compatible finite element spaces and the fully discrete scheme, together with its solvability, stability, and structure-preserving properties. Section~\ref{sec:stoch convergence} proves tightness of the discrete laws and subsequential convergence to a weak martingale solution. Numerical experiments are reported in Section~\ref{sec:numerical experiments}. Auxiliary inequalities and results on time-increment estimates are gathered in the appendix.

\section{Preliminaries}

\subsection{Notation}\label{subsec:notation}

Throughout this paper, $\mathscr{D}\subset\mathbb{R}^3$ denotes a bounded convex contractible polyhedral domain with boundary $\partial\mathscr{D}$ and outward unit normal $\bff{n}$. We fix a final time $T>0$.

For $p\in [1,\infty]$ and $k\in\mathbb{N}_0$, we denote by
\[
L^p:=L^p(\mathscr{D}),
\qquad
W^{k,p}:=W^{k,p}(\mathscr{D})
\]
the usual scalar-valued Lebesgue and Sobolev spaces, respectively. Their
$\mathbb{R}^3$-valued counterparts are denoted by
\[
\bb{L}^p:=L^p(\mathscr{D};\mathbb{R}^3),
\qquad
\bb{W}^{k,p}:=W^{k,p}(\mathscr{D};\mathbb{R}^3).
\]
We write
\[
H^k:=W^{k,2},
\qquad
\bb{H}^k:=\bb{W}^{k,2}.
\]
The spaces $H_0^1$ and $\bb{H}_0^1$ consist of scalar- and
vector-valued $H^1$-functions, respectively, with vanishing trace on
$\partial\mathscr{D}$. For $m\in \bb{N}$, we denote by $\widetilde{\bb{H}}^{-m}$ the dual of $\bb{H}^m$. We also set
\[
L_0^2
:=
\left\{
q\in L^2:
\int_{\mathscr{D}}q\,\dx=0
\right\}.
\]
The space $\mathcal D'((0,T)\times\mathscr D)$ denotes the space of distributions on $(0,T)\times\mathscr D$.

The differential operators $\nabla$, $\divg$, $\curl$, and $\Delta$ are
understood in the distributional sense whenever the functions involved are
not smooth, with $\Delta$ acting componentwise on vector fields. We abbreviate
$\partial/\partial x_j$ by $\partial_j$ and denote the time derivative by
$\partial_t$.

The standard electromagnetic energy spaces are
\begin{align*}
\hcurl
&:=
\left\{
\bff{v}\in\bb{L}^2:
\curl\bff{v}\in\bb{L}^2
\right\},
\\
\hdiv 
&:=
\left\{
\bff{v}\in\bb{L}^2:
\divg\bff{v}\in L^2
\right\}.
\end{align*}
They are respectively equipped with the graph norms
\begin{align*}
\norm{\bff{v}}{\hcurl}^2
&:=
\norm{\bff{v}}{\bb{L}^2}^2
+
\norm{\curl\bff{v}}{\bb{L}^2}^2,
\\
\norm{\bff{v}}{\hdiv}^2
&:=
\norm{\bff{v}}{\bb{L}^2}^2
+
\norm{\divg\bff{v}}{L^2}^2.
\end{align*}
The corresponding spaces with homogeneous tangential and normal traces are, respectively,
\begin{align*}
\hzerocurl
&:=
\left\{
\bff{v}\in\hcurl:
\bff{v}\times\bff{n}=\bff{0}
\text{ on }\partial\mathscr{D}
\right\},
\\
\hzerodiv
&:=
\left\{
\bff{v}\in\hdiv:
\bff{v}\cdot\bff{n}=0
\text{ on }\partial\mathscr{D}
\right\},
\end{align*}
where the boundary conditions are understood in the appropriate trace sense.

We introduce the smooth solenoidal velocity test space
\[
\mathcal D_\sigma
:=
\left\{
\bff v\in C_c^\infty(\mathscr D;\bb R^3):
\divg\bff v=0
\right\},
\]
and define
\begin{align}\label{equ:V sigma}
	\bb L_\sigma^2
	:=
	\overline{\mathcal D_\sigma}^{\,\bb L^2},
	\qquad
	\bb V_\sigma
	:=
	\overline{\mathcal D_\sigma}^{\,\bb H^1}.
\end{align}
For bounded Lipschitz domains, these spaces admit the
characterisations~\cite{Tem01}:
\[
\bb L_\sigma^2
=
\left\{
\bff v\in\bb L^2:
\inpro{\bff v}{\nabla q}=0
\;\text{ for every }q\in H^1
\right\}
=
\left\{
\bff v\in\bb L^2:
\divg\bff v=0,\ 
\bff v\cdot\bff n=0
\right\},
\]
where the divergence and normal trace are understood in the weak sense, and $\bb V_\sigma= \bb H_0^1\cap\bb L_\sigma^2$.

For the magnetic test space, we similarly set
\[
\mathcal D_B
:=
C^\infty(\overline{\mathscr D};\bb R^3)
\cap\hzerodiv,
\]
and define
\begin{align}\label{equ:V B}
	\bb V_B
	:=
	\overline{\mathcal D_B}^{\,\bb W^{1,3}}.
\end{align}

For $\bff{v},\bff{w}\in\bb{L}^2$, we write
\[
\inpro{\bff{v}}{\bff{w}}
:=
\int_{\mathscr{D}}
\bff{v}\cdot\bff{w}\,\dx.
\]
The same notation is used for the $L^2$-inner product of matrix-valued
functions, with the Euclidean scalar product replaced by the Frobenius
product. More generally, the inner product in a Hilbert space $H$ is denoted
by $\inpro{\cdot}{\cdot}_H$, and its corresponding norm by
$\norm{\cdot}{H}$. If $X$ is a Banach space, its topological dual is denoted
by $X'$, and the duality pairing between $X'$ and $X$ is written as 
$\inpro{\cdot}{\cdot}_{X',X}$.

For a Banach space $X$ and $1\leq p\leq\infty$, we use the abbreviations
\[
L_T^p(X):=L^p(0,T;X),
\qquad
W_T^{k,p}(X):=W^{k,p}(0,T;X),
\]
for the usual Bochner--Lebesgue and Bochner--Sobolev spaces, respectively.
We denote by $X_{\mathrm{w}}$ the space $X$ endowed with its weak topology.
Accordingly, $C([0,T];X_{\mathrm{w}})$
denotes the space of weakly continuous $X$-valued functions on $[0,T]$.
For $p\in [1,\infty)$, the notation $L_T^p(X)_{\mathrm{w}}$
refers to $L_T^p(X)$ endowed with its weak topology. Similarly, $L_T^\infty(X)_{\mathrm{w}^\ast}$
denotes $L_T^\infty(X)$ endowed with its weak-$\ast$ topology. When no
topological subscript is displayed, the corresponding space is understood
to carry its usual strong topology.

Let $\mathfrak{S}
=
\bigl(
\Omega,\mathscr{F},
\{\mathscr{F}_t\}_{t\in[0,T]},
\bb{P}
\bigr)$ be a stochastic basis satisfying the usual conditions. The expectation with respect to $\bb{P}$ is denoted by $\bb{E}$. For a Banach space $X$, the space $L^p(\Omega;X)$ consists of strongly measurable $X$-valued random variables $Y$ such that $\bb{E}\bigl[\norm{Y}{X}^p\bigr]<\infty$. The families $\{\beta_i^u\}_{i\in \bb{N}}$ and $\{\beta_i^B\}_{i\in \bb{N}}$ denote families of real-valued standard Brownian motions adapted to
$\{\mathscr{F}_t\}_{t\in[0,T]}$. Any independence assumptions between these
families will be stated explicitly.

For the temporal discretisation, let $N\in\mathbb{N}$, and
\begin{equation}\label{equ:tau}
\tau:=\frac{T}{N},
\qquad
t_n:=n\tau,
\qquad n=0,\ldots,N.
\end{equation}
The Brownian increments are denoted by
\begin{equation}\label{equ:Brownian increment}
\Delta_n\beta_i^u
:=
\beta_i^u(t_n)-\beta_i^u(t_{n-1}),
\qquad
\Delta_n\beta_i^B
:=
\beta_i^B(t_n)-\beta_i^B(t_{n-1}).
\end{equation}
The finite element spaces, discrete differential operators, projections, and
temporal interpolants are introduced in Section~\ref{sec:fem}.

Finally, $C>0$ denotes a generic deterministic constant whose value may
change from line to line. Unless stated otherwise, $C$ is independent of the
mesh size $h$, the time-step size $\tau$, and the number of time steps $N$.
Dependence on relevant parameters is indicated by subscripts, for example
$C_{p,T}$.

\subsection{Assumptions on the data and the noise}
\label{subsec:assumptions-noise}

Let $(\Omega,\mathscr{F},\{\mathscr{F}_t\}_{t\in[0,T]},\bb{P})$ be a stochastic basis satisfying the usual conditions. Let $\{\beta_i^u\}_{i\in\bb{N}}$ and $\{\beta_i^B\}_{i\in\bb{N}}$ be two independent families of real-valued Brownian motions. We assume that $\bff{f}\in L^2(0,T;\bb{L}^2)$ is a deterministic function.

Next, we assume that
\begin{align}\label{equ:chi assumption}
	\bff{\chi}_i\in \bb{W}^{1,\infty}_0(\mathscr{D}),\qquad
	\sum_{i=1}^\infty \norm{\bff{\chi}_i}{\bb{W}^{1,\infty}}^2<\infty.
\end{align}
The condition $\bff{\chi}_i=\bff{0}$ on $\partial\mathscr{D}$ ensures that $\bff{\Xi}_i(\bff{B})\in\hzerocurl$ whenever $\bff{B}$ is sufficiently regular.

The velocity-noise coefficients satisfy boundedness and Lipschitz conditions: there exist positive constants $C$ and $L$ such that
\begin{align}
	\label{equ:g growth}
	\sum_{i=1}^\infty \norm{\bff{g}_i(\bff{v},\bff{C})}{\bb{L}^2}^2
	&\leq
	C\bigl(1+\norm{\bff{v}}{\bb{L}^2}^2+\norm{\bff{C}}{\bb{L}^2}^2\bigr), \quad \forall \bff{v},\bff{C}\in \bb{L}^2,
	\\
	\label{equ:g lipschitz}
	\sum_{i=1}^\infty \norm{\bff{g}_i(\bff{v}_1,\bff{C}_1)-\bff{g}_i(\bff{v}_2,\bff{C}_2)}{\bb{L}^2}^2
	&\leq
	L\bigl(\norm{\bff{v}_1-\bff{v}_2}{\bb{L}^2}^2+\norm{\bff{C}_1-\bff{C}_2}{\bb{L}^2}^2\bigr),
	\quad \forall \bff{v}_1,\bff{v}_2,\bff{C}_1,\bff{C}_2\in \bb{L}^2.
\end{align}
Finally, we assume the Hall-current noise is sufficiently small relative to resistivity:
\begin{align}\label{equ:hall noise small}
	\delta_\chi:=C_{\mathscr{D}}\eta^2\sum_{i=1}^\infty \norm{\bff{\chi}_i}{\bb{W}^{1,\infty}}^2<\sigma,
\end{align}
where $C_{\mathscr{D}}$ is a constant depending on $\mathscr{D}$ arising from the Gaffney estimate~\cite{AmrBerDauGir98}.

\subsection{Weak martingale solution}
\label{subsec:weak mart sol}

We state the following definition of a weak martingale solution of the stochastic Hall--MHD problem. Recall that the spaces $\bb{V}_\sigma$ and $\bb{V}_B$ were defined in \eqref{equ:V sigma}--\eqref{equ:V B}.

\begin{definition}\label{def:weak sol}
A weak martingale solution of \eqref{equ:stoch mhd} consists of a stochastic
basis
\[
(\Omega,\mathscr F,\{\mathscr F_t\}_{t\in[0,T]},\bb P),
\]
two independent families of real-valued Brownian motions $\{\beta_i^u\}_{i\in\bb{N}}$ and $\{\beta_i^B\}_{i\in\bb{N}}$,
and progressively measurable processes $(\bff u,\bff B,\bff J)$ such that $\bff{u},\bff{B} \in C([0,T];\bb L^2_{\mathrm w})$, $\bb{P}$-a.s. and
\begin{subequations}
\begin{alignat*}{1}
	\bff u&\in
	L^2\bigl(\Omega;L^\infty_T(\bb L^2)\cap L^2_T(\bb H_0^1)\bigr),
	\\
	\bff B&\in
	L^2\bigl(\Omega;L^\infty_T(\bb L^2)\cap L^2_T(\bb H^1)\bigr),
	\\
	\bff J&\in L^2\bigl(\Omega;L^2_T(\bb L^2)\bigr).
\end{alignat*}
\end{subequations}
Moreover,
\begin{align*}
	\divg\bff u=0,
	\qquad
	\divg\bff B=0,
	\qquad
	\bff B\cdot\bff n=0,
\end{align*}
in the sense of distributions and normal traces. The current is identified by
\begin{align*}
	\inpro{\bff J(t)}{\bff\omega}
	-
	\inpro{\bff B(t)}{\curl\bff\omega}
	=0,
	\qquad
	\forall\bff\omega\in\hzerocurl,
	\quad
	\text{for a.e. }t\in(0,T),\ \bb P\text{-a.s.}
\end{align*}
For every $\bff\phi\in\bb V_\sigma$ and every $t\in[0,T]$, the velocity equation
holds $\bb P$-a.s.:
\begin{align}
	\label{equ:mart u}
	&\inpro{\bff u(t)}{\bff\phi}
	+\int_0^t\inpro{(\bff u\cdot\nabla)\bff u}{\bff\phi}\ds
	+\nu\int_0^t\inpro{\nabla\bff u}{\nabla\bff\phi}\ds
	\nonumber\\
	&\quad=
	\inpro{\bff u_0}{\bff\phi}
	+\int_0^t\inpro{\bff J\times\bff B}{\bff\phi}\ds
	+\int_0^t\inpro{\bff f}{\bff\phi}\ds
	+\sum_{i=1}^{\infty}
	\int_0^t
	\inpro{\bff g_i(\bff u,\bff B)}{\bff\phi}
	\,\dbeta_i^u(s).
\end{align}
For every $\bff\psi \in\bb{V}_B$,
and every $t\in[0,T]$, the magnetic equation holds $\bb P$-a.s.:
\begin{align}
	\label{equ:mart B}
	&\inpro{\bff B(t)}{\bff\psi}
	+\sigma\int_0^t\inpro{\bff J}{\curl\bff\psi}\ds
	+\eta\int_0^t\inpro{\bff J\times\bff B}{\curl\bff\psi}\ds
	-\int_0^t\inpro{\bff u\times\bff B}{\curl\bff\psi}\ds
	\nonumber\\
	&\quad=
	\inpro{\bff B_0}{\bff\psi}
	-\sum_{i=1}^{\infty}
	\int_0^t
	\inpro{\sigma\bff\chi_i+\eta\bff\chi_i\times\bff B}{\curl\bff\psi}
	\,\dbeta_i^B(s).
\end{align}
\end{definition}
The electric field is not included as a primary unknown in the definition.
Indeed, in the limiting weak formulation Ohm's law \eqref{equ:stoch mhd ohm} has already been
substituted into the magnetic induction equation. If desired, one may recover
\[
\bff E
:=
\sigma\bff J+\eta\bff J\times\bff B-\bff u\times\bff B
\]
as an element of $L^1(0,T;\bb L^{3/2})$, or equivalently as a distributional field.

\section{Finite element approximation}
\label{sec:fem}

\subsection{Finite element setup}

We now introduce the finite element spaces employed in the numerical approximation of problem~\eqref{equ:stoch mhd}. Let $\{\mathcal{T}_h\}_{h>0}$ be a family of quasi-uniform tetrahedral meshes of $\mathscr{D}$ with maximal mesh size $h$. Let $\bb{P}_k$ and $P_k$ denote, respectively, the space of vector-valued and scalar-valued polynomials of degree at most $k$. 

Let $(\bb{V}_h, Q_h)$ denote the MINI finite element pair, i.e. the lowest-order inf-sup stable conforming pair for incompressible flow~\cite{Vol16}. We denote by $\bb{X}_h$, $\bb{RT}_h$, and $DG_h$ the lowest-order N\'ed\'elec space, the lowest-order Raviart--Thomas space, and the space of scalar-valued piecewise constant functions, respectively.
The corresponding finite element spaces with zero (standard, tangential, or normal) traces are $\bb{V}_h^0:= \bb{V}_h\cap \bb{H}^1_0$, $\bb{X}_h^0:= \bb{X}_h\cap \hzerocurl$, and $\bb{RT}_h^0 := \bb{RT}_h \cap \hzerodiv$. Moreover, let $Q_h^0:= Q_h\cap L^2_0$ and $V_h^0:= Q_h\cap H^1_0$.

For contractible domains in $\bb{R}^3$, a key ingredient in our scheme is the following de Rham exact sequence structure for the function spaces and the associated conforming finite element spaces~\cite{ArnFalWin06}:
\begin{equation}\label{equ:comm}
	\begin{tikzcd}
		H^1_0 \arrow[r, "\nabla"] \arrow[d, "\mathcal{I}_h^Q"] & \hzerocurl \arrow[r, "\mathrm{curl}"] \arrow[d, "\mathcal{I}_h^{\bb{X}}"] & \hzerodiv \arrow[r, "\mathrm{div}"] \arrow[d, "\mathcal{I}_h^{\bb{RT}}"] & L^2_0 \arrow[d, "\mathcal{I}_h^{DG}"] \\
		Q_{h}^0 \arrow[r, "\nabla"] & \mathbb{X}_h^0 \arrow[r, "\mathrm{curl}"]                                                     & \mathbb{RT}_h^0 \arrow[r, "\mathrm{div}"]                     & DG_h              
	\end{tikzcd}
\end{equation}
The above de Rham complexes are linked by the corresponding canonical interpolators $\mathcal{I}_h^\bullet$, where $\bullet\in \{Q, \bb{X}, \bb{RT}, DG\}$. These can also be replaced by quasi-interpolators defined in~\cite{ErnGue17}.

We introduce the discrete curl operator $\curlh:\bb{RT}_h^0\to\bb{X}_h^0$ defined by
\begin{align}\label{equ:stoch curlh}
	\inpro{\curlh\bff{C}_h}{\bff{\omega}_h}=\inpro{\bff{C}_h}{\curl\bff{\omega}_h},
	\qquad
	\forall \bff{\omega}_h\in\bb{X}_h^0.
\end{align}
By the discrete Maxwell compactness theorem for finite element differential
forms~\cite{HeHuXu19}, the Raviart--Thomas spaces satisfy the following compactness
property: if $\{\bff C_h\}_h\subset\bb{RT}_h^0$ is such that
\[
\sup_{h>0}
\left(
\norm{\bff C_h}{\bb L^2}
+
\norm{\curlh\bff C_h}{\bb L^2}
+
\norm{\divg\bff C_h}{L^2}
\right)<\infty,
\]
then there exist a subsequence, not relabelled, and
$\bff C\in\hcurl \cap \hzerodiv$ such that
\[
\bff C_h\to\bff C
\quad\text{strongly in }\bb L^2(\mathscr D).
\]


In the analysis of the fluid variables, we utilise the discretely solenoidal space
\[
\bb V_{h,\sigma}
:=
\left\{\bff v_h\in\bb V_h^0: \inpro{\divg\bff v_h}{q_h}=0,
\quad\forall q_h\in Q_h^0
\right\},
\]
and the $\bb{L}^2$-orthogonal projector $\Pi_h^{\bb V}: \bb{L}^2\to \bb{V}_h^0$.
We shall also use the Stokes projection operator associated with the velocity-pressure pair. More precisely, we define the Stokes projector $\mathcal{S}_h: \bb{H}^1_0\times L^2_0\to \bb{V}_{h}^0\times Q_h^0$ by $\mathcal{S}_h(\bff{v},q):=(\mathcal{S}_h \bff{v}, \mathcal{S}_h q)$ such that 
\begin{subequations}\label{equ:stokes proj}
	\begin{alignat}{2}
		&\nu \inpro{\nabla \mathcal{S}_h \bff{v}-\nabla \bff{v}}{\nabla\bff{\phi}_h}- \inpro{\mathcal{S}_h q - q}{\divg \bff{\phi}_h}= 0,
		\; && \quad\forall \bff{\phi}_h \in \bb{V}_{h}^0,
		\label{equ:stokes proj u}
		\\[1ex]
		&\inpro{\divg \mathcal{S}_h \bff{v}- \divg \bff{v}}{q_h} = 0,
		\; && \quad\forall q_h \in Q_h^0.
		\label{equ:stokes proj p}
	\end{alignat}
\end{subequations}
The Stokes projector satisfies the following boundedness and approximation properties~\cite{GaoQiu19, Vol16}:
\begin{align}\label{equ:Stokes approx}
	\norm{\bff{v}- \mathcal{S}_h \bff{v}}{\bb{H}^1} + \norm{q-\mathcal{S}_h q}{L^2}
	&\leq
	Ch \left(\norm{\bff{v}}{\bb{H}^2}+ \norm{q}{H^1}\right),
	\\
	\label{equ:Stokes bdd}
	\norm{\mathcal{S}_h \bff{v}}{\bb{L}^\infty}+ \norm{\mathcal{S}_h \bff{v}}{\bb{W}^{1,3}} &\leq C \left(\norm{\bff{v}}{\bb{H}^2}+ \norm{q}{H^1}\right).
\end{align}

For the magnetic test functions, let $\Pi_h^{\bb{RT}}$ and $\Pi_h^{\bb{X}}$ denote the $\bb{L}^2$-orthogonal projections onto $\bb{RT}_h^0$ and $\bb{X}_h^0$, respectively. The following lemma gathers some properties necessary for our analysis. 

\begin{lemma}
	For every $\bff{\psi}\in \hcurl$, we have the identity
	\begin{align}\label{equ:conv curlh projection}
		\curlh \Pi_h^{\bb{RT}}\bff{\psi}=\Pi_h^{\bb{X}}\curl\bff{\psi}.
	\end{align}
	The operators $\Pi_h^{\bb{RT}}$ and $\Pi_h^{\bb{X}}$ satisfy the following bounds: for $q\in [1,\infty]$,
	\begin{subequations}
		\begin{alignat}{2}
			\label{equ:conv magnetic projection estimates}
			\norm{\bff{\psi}-\Pi_h^{\bb{X}} \bff{\psi}}{\bb{L}^q} 
			&\leq C_{r,q} h^r \norm{\bff{\psi}}{\bb{W}^{r,q}}, \qquad &&\forall \bff{\psi}\in \bb{W}^{r,q}, \quad r\in \{0,1\},
			\\
			\label{equ:stab magnetic projection}
			\norm{\Pi_h^{\bb{X}} \bff{\psi}}{\bb{L}^q}
			&\leq
			C_q \norm{\bff{\psi}}{\bb{L}^q}, \qquad &&\forall \bff{\psi}\in \bb{L}^q,
			\\
			\label{equ:conv magnetic projection Linfty}
			\norm{\curlh\Pi_h^{\bb{RT}}\bff{\psi}}{\bb{L}^q} &\le C_q \norm{\curl \bff{\psi}}{\bb{L}^q}, \qquad &&\forall \bff{\psi}\in\hzerodiv \cap \hcurl.
		\end{alignat}
	\end{subequations}
	Furthermore, we also have
	\begin{subequations}
		\begin{alignat}{2}
			\label{equ:PiX curl H1 stability}
			\norm{\curl\Pi_h^{\bb X}\bff \psi}{\bb L^2}
			&\le
			C\norm{\bff \psi}{\bb H^1},
			\qquad
			&&\forall\bff \psi\in\bb H^1 \cap\hzerocurl,
			\\
			\label{equ:PiX negative curl stability}
			\norm{\curl\Pi_h^{\bb X}\bff \psi}{\widetilde{\bb H}^{-m}}
			&\le
			C_m \norm{\bff \psi}{\bb L^2},
			\qquad
			&&\forall\bff \psi\in\bb L^2, \quad m\geq 1,
		\end{alignat}
	\end{subequations}
\end{lemma}

\begin{proof}
	The identity \eqref{equ:conv curlh projection} follows from \eqref{equ:stoch curlh}. Indeed, for every $\bff{\omega}_h\in\bb{X}_h^0$, we have
	\[
	\inpro{\curlh\Pi_h^{\bb{RT}}\bff{\psi}}{\bff{\omega}_h}
	=
	\inpro{\Pi_h^{\bb{RT}}\bff{\psi}}{\curl\bff{\omega}_h}
	=
	\inpro{\bff{\psi}}{\curl\bff{\omega}_h}
	=
	\inpro{\curl\bff{\psi}}{\bff{\omega}_h}
	=
	\inpro{\Pi_h^{\bb{X}}\curl\bff{\psi}}{\bff{\omega}_h}.
	\]
	Estimates~\eqref{equ:conv magnetic projection estimates} and~\eqref{equ:stab magnetic projection} were shown in~~\cite{VeiGomPerZam26,VeiHuMas25}, while inequality~\eqref{equ:conv magnetic projection Linfty} follows immediately by \eqref{equ:conv curlh projection} and \eqref{equ:stab magnetic projection}.
	
	Next, we prove \eqref{equ:PiX curl H1 stability}. Let
	$\mathcal J_h^{\bb X}$ be a first-order N\'ed\'elec quasi-interpolation operator~\cite{ErnGue18} satisfying
	\[
	\norm{\bff \psi-\mathcal J_h^{\bb X}\bff \psi}{\bb L^2}
	\le
	Ch\norm{\bff \psi}{\bb H^1},
	\qquad
	\norm{\curl\mathcal J_h^{\bb X}\bff \psi}{\bb L^2}
	\le
	C\norm{\bff \psi}{\bb H^1}.
	\]
	We write
	\[
	\curl\Pi_h^{\bb X}\bff \psi
	=
	\curl\mathcal J_h^{\bb X}\bff \psi
	+
	\curl \Pi_h^{\bb X} \left(\bff \psi-\mathcal J_h^{\bb X}\bff \psi\right).
	\]
	Using the inverse estimate and the $\bb L^2$-stability of
	$\Pi_h^{\bb X}$, we get
	\[
	\begin{aligned}
		\norm{\curl\Pi_h^{\bb X}\bff \psi}{\bb L^2}
		&\le
		\norm{\curl\mathcal J_h^{\bb X}\bff \psi}{\bb L^2}
		+
		Ch^{-1}
		\norm{\Pi_h^{\bb X}\left(\bff \psi-\mathcal J_h^{\bb X}\bff \psi\right)}
		{\bb L^2}
		\\
		&\le
		C\norm{\bff \psi}{\bb H^1}
		+
		Ch^{-1}
		\norm{\bff \psi-\mathcal J_h^{\bb X}\bff \psi}{\bb L^2}
		\\
		&\le
		C\norm{\bff \psi}{\bb H^1}.
	\end{aligned}
	\]
	This proves \eqref{equ:PiX curl H1 stability}.
	
	Finally, for $\bff z \in\bb H^m(\mathscr D)$ with
	$\norm{\bff z}{\bb H^m}\le1$, integration by parts gives
	\[
	\abs{\inpro{\curl\Pi_h^{\bb X}\bff \psi}{\bff z}}
	=
	\abs{\inpro{\Pi_h^{\bb X}\bff \psi}{\curl\bff z}}
	\le
	\norm{\Pi_h^{\bb X}\bff \psi}{\bb L^2}
	\norm{\curl \bff z}{\bb L^2}
	\le
	C\norm{\bff \psi}{\bb L^2}.
	\]
	Taking the supremum over such $\bff z$ gives~\eqref{equ:PiX negative curl stability}. This completes the proof.
\end{proof}

In our analysis, we need to introduce a reconstruction space $\bb{W}_h^0\subset\bb{H}^1(\mathscr{D})\cap\hzerodiv$ as follows. Let
\[
\bb{S}_h^{\mathcal R}:=\{\bff v_h\in \bb{H}^1(\mathscr D): \bff{v}_h|_K\in \mathbb P_1(K),\;\forall K\in\mathcal T_h\}.
\]
We set
\begin{align}\label{equ:Wh0 definition}
	\bb W_h^0
	:=
	\Bigl\{
	\bff\Phi_h\in \bb{S}_h^{\mathcal R}:
	\bff\Phi_h\cdot\bff n_F=0
	\text{ on every boundary face }F\subset\partial\mathscr D
	\Bigr\}.
\end{align}
Here $\bff n_F$ denotes the constant outward unit normal on the boundary face $F$. Then
\[
\bb W_h^0\subset \bb H^1(\mathscr D)\cap\hzerodiv.
\]
Moreover, the spaces $\bb W_h^0$ are dense in
$\bb H^1(\mathscr D)\cap\hzerodiv$ in the following sense: for every
$\bff\Phi\in\bb H^1(\mathscr D)\cap\hzerodiv$, there exists
$\bff\Phi_h\in\bb W_h^0$ such that
\begin{align}\label{equ:Wh0 approximation property}
	\norm{\bff\Phi_h-\bff\Phi}{\bb H^1}\to0
	\qquad\text{as }h\to0.
\end{align}
Indeed, this follows from the density of smooth vector fields satisfying the
homogeneous normal boundary condition in $\bb H^1(\mathscr D)\cap\hzerodiv$ on polyhedral domains~\cite{BauPau16}. For such smooth fields, the nodal Lagrange interpolant preserves the condition $\bff\Phi_h\cdot\bff n_F=0$ on each boundary face $F$, since $\bff n_F$ is constant and the interpolant is affine on $F$. The space $\bb W_h^0$ is
used only as an auxiliary conforming space for the computable Maxwell
reconstruction to be defined next.

For $\bff{C}_h\in\bb{RT}_h^0$, define its Maxwell reconstruction $\mathcal{R}_h\bff{C}_h\in\bb{W}_h^0$ by
\begin{align}\label{equ:Maxwell reconstruction scheme}
	&\inpro{\mathcal{R}_h\bff{C}_h}{\bff{\Phi}_h}
	+\inpro{\curl\mathcal{R}_h\bff{C}_h}{\curl\bff{\Phi}_h}
	+\inpro{\divg\mathcal{R}_h\bff{C}_h}{\divg\bff{\Phi}_h}
	\nonumber\\
	&\quad
	=
	\inpro{\bff{C}_h}{\bff{\Phi}_h}
	+\inpro{\curlh\bff{C}_h}{\curl\bff{\Phi}_h}
	+\inpro{\divg\bff{C}_h}{\divg\bff{\Phi}_h},
	\qquad
	\forall \bff{\Phi}_h\in\bb{W}_h^0.
\end{align}
Since the bilinear form in \eqref{equ:Maxwell reconstruction scheme} is coercive on $\bb{W}_h^0$, this operator $\mathcal{R}_h$ is well-defined.
We have the following lemma on the boundedness and the consistency of the operator $\mathcal R_h$.

\begin{lemma}[Boundedness and consistency of $\mathcal{R}_h$]
	\label{lem:conv reconstruction consistency}
	Let $\{\bff C_h\}_{h>0}$ be a sequence with $\bff C_h\in\bb{RT}_h^0$. Then we have
	\begin{align}\label{equ:Maxwell reconstruction stability}
		\norm{\mathcal{R}_h\bff{C}_h}{\bb{H}^1}
		\leq
		C\bigl(\norm{\bff{C}_h}{\bb{L}^2}+\norm{\curlh\bff{C}_h}{\bb{L}^2}+\norm{\divg\bff{C}_h}{L^2}\bigr),
		\qquad
		\forall \bff{C}_h\in\bb{RT}_h^0.
	\end{align}
	Furthermore, suppose that as $h\to 0$,
	\begin{align*}
		&\bff C_h\to\bff C\quad\text{strongly in }\bb L^2(\mathscr D),
		\\
		&\curlh\bff C_h\rightharpoonup\bff K\quad\text{weakly in }\bb L^2(\mathscr D),
		\\
		&\divg\bff C_h\rightharpoonup D\quad\text{weakly in }L^2(\mathscr D).
	\end{align*}
	Assume that $\bff K=\curl\bff C$ and $D=\divg\bff C$
	in the sense of distributions. Then
	\begin{align}
		\label{equ:recon consistency conclusion}
		\mathcal R_h\bff C_h\to\bff C
		\quad\text{strongly in }\bb L^2(\mathscr D).
	\end{align}
\end{lemma}

\begin{proof}
	Taking $\bff\Phi_h=\mathcal R_h\bff C_h$ in the definition of the reconstruction gives
	\begin{align*}
		&\norm{\mathcal R_h\bff C_h}{\bb L^2}^2
		+\norm{\curl\mathcal R_h\bff C_h}{\bb L^2}^2
		+\norm{\divg\mathcal R_h\bff C_h}{L^2}^2
		\\
		&=
		\inpro{\bff C_h}{\mathcal R_h\bff C_h}
		+\inpro{\curlh\bff C_h}{\curl\mathcal R_h\bff C_h}
		+\inpro{\divg\bff C_h}{\divg\mathcal R_h\bff C_h}
		\\
		&\le
		\bigl(
		\norm{\bff C_h}{\bb L^2}
		+\norm{\curlh\bff C_h}{\bb L^2}
		+\norm{\divg\bff C_h}{L^2}
		\bigr)
		\bigl(
		\norm{\mathcal R_h\bff C_h}{\bb L^2}
		+\norm{\curl\mathcal R_h\bff C_h}{\bb L^2}
		+\norm{\divg\mathcal R_h\bff C_h}{L^2}
		\bigr).
	\end{align*}
	Therefore,
	\begin{align}
		\label{equ:recon graph bound proof}
		\norm{\mathcal R_h\bff C_h}{\bb L^2}
		+\norm{\curl\mathcal R_h\bff C_h}{\bb L^2}
		+\norm{\divg\mathcal R_h\bff C_h}{L^2}
		\le
		C\bigl(
		\norm{\bff C_h}{\bb L^2}
		+\norm{\curlh\bff C_h}{\bb L^2}
		+\norm{\divg\bff C_h}{L^2}
		\bigr).
	\end{align}
	Since $\mathscr D$ is a convex contractible polyhedron, the Maxwell regularity estimate gives
	\[
	\norm{\bff v}{\bb H^1}
	\le
	C\bigl(
	\norm{\bff v}{\bb L^2}
	+\norm{\curl\bff v}{\bb L^2}
	+\norm{\divg\bff v}{L^2}
	\bigr),
	\qquad
	\forall \bff v\in \bb H^1(\mathscr D)\cap\hzerodiv .
	\]
	Using $\mathcal R_h\bff C_h\in\bb W_h^0\subset \bb H^1(\mathscr D)\cap\hzerodiv$, we obtain
	\begin{align}
		\label{equ:recon H1 bound proof}
		\norm{\mathcal R_h\bff C_h}{\bb H^1}
		\le
		C\bigl(
		\norm{\bff C_h}{\bb L^2}
		+\norm{\curlh\bff C_h}{\bb L^2}
		+\norm{\divg\bff C_h}{L^2}
		\bigr).
	\end{align}
	Thus the boundedness assertion \eqref{equ:Maxwell reconstruction stability} follows.
	
	We now prove consistency. By \eqref{equ:recon H1 bound proof}, there exist a subsequence, not relabelled, and $\bff R\in\bb H^1(\mathscr D)$ such that
	\[
	\mathcal R_h\bff C_h\rightharpoonup\bff R\quad\text{weakly in }\bb H^1(\mathscr D),
	\qquad
	\mathcal R_h\bff C_h\to\bff R\quad\text{strongly in }\bb L^2(\mathscr D).
	\]
	Since $\bb W_h^0\subset\hzerodiv$ and $\hzerodiv$ is weakly closed, $\bff R\in\hzerodiv$.
	
	Let $\bff\Phi\in\bb H^1(\mathscr D)\cap\hzerodiv$. By the approximation property of $\bb W_h^0$, choose $\bff\Phi_h\in\bb W_h^0$ such that
	\[
	\bff\Phi_h\to\bff\Phi
	\qquad\text{strongly in }\bb H^1(\mathscr D).
	\]
	Testing the reconstruction equation with $\bff\Phi_h$ gives
	\begin{align}
		\label{equ:recon equation test}
		&\inpro{\mathcal R_h\bff C_h}{\bff\Phi_h}
		+\inpro{\curl\mathcal R_h\bff C_h}{\curl\bff\Phi_h}
		+\inpro{\divg\mathcal R_h\bff C_h}{\divg\bff\Phi_h}
		\nonumber\\
		&=
		\inpro{\bff C_h}{\bff\Phi_h}
		+\inpro{\curlh\bff C_h}{\curl\bff\Phi_h}
		+\inpro{\divg\bff C_h}{\divg\bff\Phi_h}.
	\end{align}
	Passing to the limit gives
	\begin{align}
		\label{equ:recon limit variational}
		\inpro{\bff R}{\bff\Phi}
		+\inpro{\curl\bff R}{\curl\bff\Phi}
		+\inpro{\divg\bff R}{\divg\bff\Phi}
		=
		\inpro{\bff C}{\bff\Phi}
		+\inpro{\bff K}{\curl\bff\Phi}
		+\inpro{D}{\divg\bff\Phi}.
	\end{align}
	Using $\bff K=\curl\bff C$ and $D=\divg\bff C$,
	\begin{align}
		\label{equ:recon limit variational C}
		\inpro{\bff R}{\bff\Phi}
		+\inpro{\curl\bff R}{\curl\bff\Phi}
		+\inpro{\divg\bff R}{\divg\bff\Phi}
		=
		\inpro{\bff C}{\bff\Phi}
		+\inpro{\curl\bff C}{\curl\bff\Phi}
		+\inpro{\divg\bff C}{\divg\bff\Phi}.
	\end{align}
	Since $\bff C\in\hzerodiv$, $\curl\bff C\in\bb L^2$, and $\divg\bff C\in L^2$, the Maxwell regularity estimate gives $\bff C\in\bb H^1(\mathscr D)$. Hence $\bff Z:=\bff R-\bff C\in \bb H^1(\mathscr D)\cap\hzerodiv$. Taking $\bff\Phi=\bff Z$ in \eqref{equ:recon limit variational C} gives
	\[
	\norm{\bff Z}{\bb L^2}^2+\norm{\curl\bff Z}{\bb L^2}^2+\norm{\divg\bff Z}{L^2}^2=0.
	\]
	Thus $\bff R=\bff C$. Since every subsequence has a further subsequence converging to $\bff C$, the whole sequence converges strongly in $\bb L^2$.
\end{proof}

\subsection{Fully discrete scheme}
\label{subsec:stoch fem scheme}

Let the discretisation parameters $h,\tau>0$, and recall the notation in \eqref{equ:tau} and \eqref{equ:Brownian increment}. For $n=1,\ldots,N$, we set $\bff f^n:= \frac1{\tau} \int_{t_{n-1}}^{t_n} \bff f(s) \ds$. A fully discrete scheme for solving \eqref{equ:stoch mhd} can be described as follows.
Let
\[
\bff u_h^0\in\bb V_{h,\sigma},
\qquad
\bff B_h^0\in\bb{RT}_h^0,
\qquad
\bff J_h^0\in\bb X_h^0,
\]
satisfy
\[
\divg\bff B_h^0=0,
\qquad
\bff J_h^0=\curlh\bff B_h^0.
\]
We assume that
\[
\bff u_h^0\to\bff u_0,
\qquad
\bff B_h^0\to\bff B_0
\quad\text{strongly in }\bb L^2,
\]
and $\sup_h\norm{\bff J_h^0}{\bb L^2}<\infty$.

For $n=1,2,\ldots,N$, assume $(\bff{u}_h^{n-1},\bff{B}_h^{n-1},\bff{J}_h^{n-1})$ is known. The scheme seeks to find
\[
(\bff{u}_h^n,p_h^n,\bff{B}_h^n,\bff{J}_h^n,\bff{E}_h^n)\in \bb{V}_h^0\times Q_h^0\times\bb{RT}_h^0\times\bb{X}_h\times\bb{X}_h^0
\]
such that, for all $(\bff{\phi}_h,q_h,\bff{\psi}_h,\bff{\chi}_h,\bff{\omega}_h)\in \bb{V}_h^0\times Q_h^0\times\bb{RT}_h^0\times\bb{X}_h\times\bb{X}_h^0$,
\begin{subequations}\label{equ:stoch fem}
	\begin{align}
		&\inpro{\bff{u}_h^n-\bff{u}_h^{n-1}}{\bff{\phi}_h}
		+\nu\tau\inpro{\nabla\bff{u}_h^n}{\nabla\bff{\phi}_h}
		+\frac{\tau}{2}\Big[\inpro{(\bff{u}_h^{n-1}\cdot\nabla)\bff{u}_h^n}{\bff{\phi}_h}-\inpro{(\bff{u}_h^{n-1}\cdot\nabla)\bff{\phi}_h}{\bff{u}_h^n}\Big]
		\nonumber\\
		&\quad
		-\tau\inpro{p_h^n}{\divg\bff{\phi}_h}
		-\tau\inpro{\bff{J}_h^n\times\bff{B}_h^{n-1}}{\bff{\phi}_h}
		=
		\tau\inpro{\bff{f}^n}{\bff{\phi}_h}
		+\sum_{i=1}^\infty\inpro{\bff{g}_i(\bff{u}_h^{n-1},\bff{B}_h^{n-1})}{\bff{\phi}_h}\Delta_n\beta_i^u,
		\label{equ:stoch fem u}
		\\[1ex]
		&\inpro{\divg\bff{u}_h^n}{q_h}=0,
		\label{equ:stoch fem div u}
		\\[1ex]
		&\inpro{\bff{B}_h^n-\bff{B}_h^{n-1}}{\bff{\psi}_h}
		+\tau\inpro{\curl\bff{E}_h^n}{\bff{\psi}_h}
		=
		-\sum_{i=1}^{\infty}
		\inpro{\curl \Pi_h^{\bb{X}}\bigl(\sigma\bff{\chi}_i+\eta\bff{\chi}_i\times\mathcal{R}_h\bff{B}_h^{n-1}\bigr)}{\bff{\psi}_h}\Delta_n\beta_i^B.
		\label{equ:stoch fem B}
		\\[1ex]
		&\sigma\inpro{\bff{J}_h^n}{\bff{\chi}_h}
		+\eta\inpro{\bff{J}_h^n\times\bff{B}_h^{n-1}}{\bff{\chi}_h}
		=
		\inpro{\bff{E}_h^n}{\bff{\chi}_h}
		+\inpro{\bff{u}_h^n\times\bff{B}_h^{n-1}}{\bff{\chi}_h},
		\label{equ:stoch fem Ohm}
		\\[1ex]
		&\inpro{\bff{J}_h^n}{\bff{\omega}_h}
		-\inpro{\bff{B}_h^n}{\curl\bff{\omega}_h}=0,
		\label{equ:stoch fem J}
	\end{align}
\end{subequations}
where $\Delta_n\beta_i^u$ and $\Delta_n\beta_i^B$ are the Brownian increments given by \eqref{equ:Brownian increment}.

We remark that the scheme is well-posed by an argument similar to~\cite{GolSoeTra26}.

\subsection{Stability estimates}
\label{subsec:stoch stability estimates}

In this subsection we derive the estimates needed for compactness of the fully discrete stochastic scheme.

We first note that \eqref{equ:stoch fem J} is equivalent to $\Pi_h^{\bb{X}} \bff{J}_h^{n-1}=\curlh\bff{B}_h^{n-1}$ \eqref{equ:Maxwell reconstruction stability}. If $\divg\bff{B}_h^{n-1}=0$, then
\begin{align}\label{equ:stoch reconstruction B bound}
	\norm{\mathcal{R}_h\bff{B}_h^{n-1}}{\bb{H}^1}
	\leq
	C\bigl(\norm{\bff{B}_h^{n-1}}{\bb{L}^2}+\norm{\bff{J}_h^{n-1}}{\bb{L}^2}\bigr),
\end{align}
which implies by \eqref{equ:PiX curl H1 stability} and Young's inequality,
\begin{align}\label{equ:hall current coeff bound final}
	\sum_{i=1}^{\infty}\norm{\curl \Pi_h^{\bb{X}}\bigl(\sigma\bff{\chi}_i+\eta\bff{\chi}_i\times\mathcal{R}_h\bff{B}_h^{n-1}\bigr)}{\bb{L}^2}^2
	\leq
	C_\chi\bigl(1+\norm{\bff{B}_h^{n-1}}{\bb{L}^2}^2\bigr)+\delta_\chi\norm{\bff{J}_h^{n-1}}{\bb{L}^2}^2,
\end{align}
where $\delta_\chi$ was defined in \eqref{equ:hall noise small} and
\begin{align*}
	C_\chi:=C\sigma^2\sum_{i=1}^{\infty}\norm{\bff{\chi}_i}{\bb{H}^1}^2+C\eta^2\sum_{i=1}^{\infty}\norm{\bff{\chi}_i}{\bb{W}^{1,\infty}}^2.
\end{align*}
The constant $C$ is independent of $h$, $\tau$, and $n$.

The following proposition shows the preservation of magnetic divergence at the discrete level and higher moment bounds under the assumption that $\delta_\chi$ is sufficiently small.

\begin{proposition}
	\label{prop:stoch stability revised}
	Let $(\bff{u}_h^n,p_h^n,\bff{B}_h^n,\bff{E}_h^n,\bff{J}_h^n)$, $n=0,\ldots,N$, be a solution of the fully discrete scheme \eqref{equ:stoch fem}. Suppose that
	\[
	\divg\bff{B}_h^0=0,
	\qquad
	\Pi_h^{\bb{X}} \bff{J}_h^0=\curlh\bff{B}_h^0,
	\]
	and that the velocity-noise growth condition \eqref{equ:g growth} and the current-noise assumptions \eqref{equ:chi assumption} and \eqref{equ:hall noise small} hold.
	Then
	\begin{align}\label{equ:div B exact revised}
		\divg\bff{B}_h^n=0
		\quad\text{pointwise on each element of }\mathcal T_h,
		\qquad n=0,\ldots,N.
	\end{align}
	Moreover, for every integer $p\ge1$, there exists
	$\delta_{p}^\ast>0$ such that, if
	$\delta_\chi\le \delta_{p}^\ast$, then
	\begin{align}\label{equ:stoch stability revised main}
		&\bb E \left[\max_{0\le n\le N}
		\left(
		\norm{\bff u_h^n}{\bb L^2}^2
		+
		\norm{\bff B_h^n}{\bb L^2}^2
		\right)^p\right]
		+
		\bb E\left[
		\left(
		\sum_{n=1}^{N}
		\norm{\bff u_h^n-\bff u_h^{n-1}}{\bb L^2}^2
		+
		\sum_{n=1}^{N}
		\norm{\bff B_h^n-\bff B_h^{n-1}}{\bb L^2}^2
		\right)^p
		\right]
		\nonumber\\
		&\quad
		+
		\bb E\left[
		\left(
		\sum_{n=1}^{N}\tau
		\left(
		\nu\norm{\nabla\bff u_h^n}{\bb L^2}^2
		+
		\sigma\norm{\bff J_h^n}{\bb L^2}^2
		\right)
		\right)^p
		\right]
		\nonumber\\
		&\le
		C_{p,T}
		\left(
		1
		+
		\bb E\left[
		\left(
		\norm{\bff u_h^0}{\bb L^2}^2
		+
		\norm{\bff B_h^0}{\bb L^2}^2
		\right)^p
		\right]
		+
		\bb E\left[
		\left(
		\tau\norm{\bff J_h^0}{\bb L^2}^2
		\right)^p
		\right]
		+
		\norm{\bff f}{L^2_T(\bb L^2)}^{2p}
		\right).
	\end{align}
\end{proposition}

\begin{proof}
	We divide the proof into several steps.
	
	\emph{Step 1: Exact preservation of magnetic divergence.}
	Since
	\[
	\bff E_h^n\in\bb X_h^0,
	\qquad
	\Pi_h^{\bb X}\bigl(\sigma\bff\chi_i+\eta\bff\chi_i\times\mathcal R_h\bff B_h^{n-1}\bigr)\in\bb X_h^0,
	\]
	we have
	\[
	\curl\bff E_h^n\in\bb{RT}_h^0,
	\qquad
	\curl\Pi_h^{\bb X}\bigl(\sigma\bff\chi_i+\eta\bff\chi_i\times\mathcal R_h\bff B_h^{n-1}\bigr)\in\bb{RT}_h^0.
	\]
	The magnetic update \eqref{equ:stoch fem B} therefore gives, as an identity in $\bb{RT}_h^0$,
	\begin{align}\label{equ:B update strong revised}
		\bff B_h^n
		=
		\bff B_h^{n-1}
		-\tau\,\curl\bff E_h^n
		-\sum_{i=1}^{\infty}
		\curl\Pi_h^{\bb X}\bigl(\sigma\bff\chi_i+\eta\bff\chi_i\times\mathcal R_h\bff B_h^{n-1}\bigr)\Delta_n\beta_i^B.
	\end{align}
	Taking the elementwise divergence and using $\divg\curl\bff\omega_h=0$ for all $\bff\omega_h\in\bb X_h^0$, we get
	\[
	\divg\bff B_h^n=\divg\bff B_h^{n-1}.
	\]
	The claim \eqref{equ:div B exact revised} follows by induction over $n$.
	
	\emph{Step 2: Discrete energy identity.}
	We define $\bff g_{i,h}^{n-1}:= \Pi_h^{\bb V} \bff g_i(\bff u_h^{n-1},\bff B_h^{n-1})$ and set
	\begin{align}\label{equ:G u G B revised}
		\bff G_{u,h}^n
		:=
		\sum_{i=1}^{\infty}\bff g_{i,h}^{n-1}\Delta_n\beta_i^u,
		\qquad
		\bff G_{B,h}^n
		:=
		\sum_{i=1}^{\infty}
		\Pi_h^{\bb X}\bigl(\sigma\bff\chi_i+\eta\bff\chi_i\times\mathcal R_h\bff B_h^{n-1}\bigr)\Delta_n\beta_i^B.
	\end{align}
	For ease of presentation, we define the discrete energy
	\[
	\mathcal E_h^n
	:=
	\frac12\norm{\bff u_h^n}{\bb L^2}^2
	+
	\frac12\norm{\bff B_h^n}{\bb L^2}^2,
	\]
	and take 
	\[
	\bff\phi_h=\bff u_h^n,\qquad
	q_h=p_h^n,\qquad
	\bff\psi_h=\bff B_h^n,\qquad
	\bff\chi_h=\bff J_h^n,\qquad
	\bff\omega_h=\bff E_h^n
	\]
	in the scheme \eqref{equ:stoch fem}. We note the identities
	\begin{align*}
	\inpro{\curl\bff E_h^n}{\bff B_h^n}
	&=
	\inpro{\bff E_h^n}{\bff J_h^n}
	=
	\sigma\norm{\bff J_h^n}{\bb L^2}^2
	-
	\inpro{\bff u_h^n\times\bff B_h^{n-1}}{\bff J_h^n},
	\end{align*}
	by the definition of $\curlh$ and Ohm's law \eqref{equ:stoch fem Ohm}. 
	Using the elementary vector identity
	\[
	\inpro{\bff a-\bff b}{\bff a}
	=
	\frac12\norm{\bff a}{\bb L^2}^2
	-
	\frac12\norm{\bff b}{\bb L^2}^2
	+
	\frac12\norm{\bff a-\bff b}{\bb L^2}^2,
	\]
	we obtain the pathwise energy identity
	\begin{align}\label{equ:energy identity revised}
		&\mathcal E_h^n-\mathcal E_h^{n-1}
		+
		\frac12\norm{\bff u_h^n-\bff u_h^{n-1}}{\bb L^2}^2
		+
		\frac12\norm{\bff B_h^n-\bff B_h^{n-1}}{\bb L^2}^2
		+
		\nu\tau\norm{\nabla\bff u_h^n}{\bb L^2}^2
		+
		\sigma\tau\norm{\bff J_h^n}{\bb L^2}^2
		\nonumber\\
		&\qquad
		=
		\tau\inpro{\bff f^n}{\bff u_h^n}
		+
		\inpro{\bff G_{u,h}^n}{\bff u_h^n}
		-
		\inpro{\curl\bff G_{B,h}^n}{\bff B_h^n}.
	\end{align}
	
	\emph{Step 3: Pathwise inequality with martingale increment.}
	Split
	\[
	\inpro{\bff G_{u,h}^n}{\bff u_h^n}
	=
	\inpro{\bff G_{u,h}^n}{\bff u_h^{n-1}}
	+
	\inpro{\bff G_{u,h}^n}{\bff u_h^n-\bff u_h^{n-1}},
	\]
	and
	\[
	-\inpro{\curl\bff G_{B,h}^n}{\bff B_h^n}
	=
	-\inpro{\curl\bff G_{B,h}^n}{\bff B_h^{n-1}}
	-
	\inpro{\curl\bff G_{B,h}^n}{\bff B_h^n-\bff B_h^{n-1}}.
	\]
	Define
	\begin{align}\label{equ:mart increment revised}
		M_h^n
		:=
		\inpro{\bff G_{u,h}^n}{\bff u_h^{n-1}}
		-
		\inpro{\curl\bff G_{B,h}^n}{\bff B_h^{n-1}}.
	\end{align}
	Young's inequality then yields
	\begin{align}\label{equ:pathwise inequality revised}
		&\mathcal E_h^n-\mathcal E_h^{n-1}
		+
		c_0\Big(
		\norm{\bff u_h^n-\bff u_h^{n-1}}{\bb L^2}^2
		+
		\norm{\bff B_h^n-\bff B_h^{n-1}}{\bb L^2}^2
		+
		\tau\norm{\nabla\bff u_h^n}{\bb L^2}^2
		+
		\tau\norm{\bff J_h^n}{\bb L^2}^2
		\Big)
		\nonumber\\
		&\qquad
		\le
		C\tau\norm{\bff f^n}{\widetilde{\bb H}^{-1}}^2
		+
		M_h^n
		+
		C\norm{\bff G_{u,h}^n}{\bb L^2}^2
		+
		C\norm{\curl\bff G_{B,h}^n}{\bb L^2}^2,
	\end{align}
	where $c_0>0$ depends on $\nu$ and $\sigma$ but not on $h,\tau,n$.
	
	\emph{Step 4: First moment estimate.}
	We note that the conditional It\^o isometry gives
	\begin{align}\label{equ:cond Ito u revised}
		\bb E\left[
		\norm{\bff G_{u,h}^n}{\bb L^2}^2\mid\mathscr F_{t_{n-1}}
		\right]
		=
		\tau\sum_{i=1}^{\infty}\norm{\bff g_{i,h}^{n-1}}{\bb L^2}^2,
	\end{align}
	and
	\begin{align}\label{equ:cond Ito B revised}
		\bb E\left[
		\norm{\curl\bff G_{B,h}^n}{\bb L^2}^2\mid\mathscr F_{t_{n-1}}
		\right]
		&=
		\tau\sum_{i=1}^{\infty}
		\norm{
			\curl\Pi_h^{\bb X}\bigl(\sigma\bff\chi_i+\eta\bff\chi_i\times\mathcal R_h\bff B_h^{n-1}\bigr)
		}{\bb L^2}^2
		\nonumber\\
		&\le
		C_\chi \tau \bigl(1+\norm{\bff{B}_h^{n-1}}{\bb{L}^2}^2\bigr)+\delta_\chi \tau \norm{\bff{J}_h^{n-1}}{\bb{L}^2}^2,
	\end{align}
	where we also used \eqref{equ:hall current coeff bound final}.

	Now, taking conditional expectation in \eqref{equ:pathwise inequality revised}, noting the fact that $\bb E[M_h^n\mid\mathscr F_{t_{n-1}}]=0$, using \eqref{equ:cond Ito u revised}--\eqref{equ:cond Ito B revised}, \eqref{equ:g growth}, and the tower property, we obtain
	\begin{align}\label{equ:first moment one step revised}
		&\bb E \left[\mathcal E_h^n\right]
		+
		c_0\bb E\Big[
		\norm{\bff u_h^n-\bff u_h^{n-1}}{\bb L^2}^2
		+
		\norm{\bff B_h^n-\bff B_h^{n-1}}{\bb L^2}^2
		+
		\tau\norm{\nabla\bff u_h^n}{\bb L^2}^2
		+
		\tau\norm{\bff J_h^n}{\bb L^2}^2
		\Big]
		\nonumber\\
		&\qquad
		\le
		(1+C\tau)\bb E\mathcal E_h^{n-1}
		+
		C\tau
		+
		C\tau\bb E\norm{\bff f^n}{\widetilde{\bb H}^{-1}}^2
		+
		C\delta_\chi\tau\bb E\norm{\bff J_h^{n-1}}{\bb L^2}^2.
	\end{align}
	Summing over $n=1,\ldots,m$, and absorbing the last term to the left-hand side when $\delta_\chi$ is sufficiently small, we obtain the first moment estimate by the discrete Gronwall lemma.

	\emph{Step 5: Quadratic stochastic remainders.}
	For $m=1,\ldots,N$, set
	\[
	Q_{u,m}
	:=
	\sum_{n=1}^{m}
	\norm{\bff G_{u,h}^n}{\bb L^2}^2,
	\qquad
	Q_{B,m}
	:=
	\sum_{n=1}^{m}
	\norm{\curl\bff G_{B,h}^n}{\bb L^2}^2.
	\]
	Conditionally on $\mathscr F_{t_{n-1}}$, the random variables
	$\bff G_{u,h}^n$ and $\curl\bff G_{B,h}^n$ are centred
	$\bb L^2$-valued Gaussian random variables. Hence, for every integer
	$p\ge1$,
	\begin{align}
		\label{equ:conditional Gaussian Gu}
		\bb E\left[
		\norm{\bff G_{u,h}^n}{\bb L^2}^{2p}
		\mid\mathscr F_{t_{n-1}}
		\right]
		&\le
		C_p\tau^p
		\left(
		\sum_{i=1}^{\infty}
		\norm{\bff g_{i,h}^{n-1}}{\bb L^2}^2
		\right)^p,
		\\
		\label{equ:conditional Gaussian GB}
		\bb E\left[
		\norm{\curl\bff G_{B,h}^n}{\bb L^2}^{2p}
		\mid\mathscr F_{t_{n-1}}
		\right]
		&\le
		C_p\tau^p
		\left(
		\sum_{i=1}^{\infty}
		\norm{
			\curl\Pi_h^{\bb X}
			\bigl(
			\sigma\bff\chi_i
			+
			\eta\bff\chi_i\times\mathcal R_h\bff B_h^{n-1}
			\bigr)}
		{\bb L^2}^2
		\right)^p.
	\end{align}
	The corresponding identities for $p=1$ are
	\eqref{equ:cond Ito u revised} and \eqref{equ:cond Ito B revised}.
	
	For $p>1$, the elementary inequality
	\[
	(x+y)^p-x^p
	\le
	C_px^{p-1}y+C_py^p,
	\qquad x,y\ge0,
	\]
	gives
	\[
	Q_{u,n}^p-Q_{u,n-1}^p
	\le
	C_pQ_{u,n-1}^{p-1}
	\norm{\bff G_{u,h}^n}{\bb L^2}^2
	+
	C_p\norm{\bff G_{u,h}^n}{\bb L^2}^{2p}.
	\]
	Taking conditional expectations, summing over $n=1,\ldots,m$, and using
	$Q_{u,n-1}\le Q_{u,m}$, we obtain
	\begin{align*}
		\bb E[Q_{u,m}^p]
		&\le
		C_p\bb E\left[
		Q_{u,m}^{p-1}
		\sum_{n=1}^{m}\tau
		\sum_{i=1}^{\infty}
		\norm{\bff g_{i,h}^{n-1}}{\bb L^2}^2
		\right]
		+
		C_p\bb E\left[
		\sum_{n=1}^{m}
		\left(
		\tau\sum_{i=1}^{\infty}
		\norm{\bff g_{i,h}^{n-1}}{\bb L^2}^2
		\right)^p
		\right].
	\end{align*}
	Since $\sum_{n=1}^{m}a_n^p
	\le
	\left(\sum_{n=1}^{m}a_n\right)^p$ for $a_n\ge0$,
	Young's inequality allows the first term on the right-hand side to be absorbed, yielding
	\begin{align}
		\label{equ:Q u p bound revised}
		\bb E[Q_{u,m}^p]
		\le
		C_p\bb E\left[
		\left(
		\sum_{n=1}^{m}\tau
		\sum_{i=1}^{\infty}
		\norm{\bff g_{i,h}^{n-1}}{\bb L^2}^2
		\right)^p
		\right]
		\le
		C_{p,T}
		\left[
		1+
		\sum_{n=1}^{m}\tau
		\bb E\left[(\mathcal E_h^{n-1})^p\right]
		\right],
	\end{align}
	where we also used \eqref{equ:g growth} and
	Lemma~\ref{lem:weighted-discrete-holder}.

	Repeating the same argument with
	$\norm{\curl\bff G_{B,h}^n}{\bb L^2}$ gives
	\begin{align}\label{equ:Q B p bound revised}
		\bb E[Q_{B,m}^p]
		&\le
		C_p\bb E\left[
		\left(
		\sum_{n=1}^{m}\tau
		\sum_{i=1}^{\infty}
		\norm{
			\curl\Pi_h^{\bb X}
			\bigl(
			\sigma\bff\chi_i
			+
			\eta\bff\chi_i\times\mathcal R_h\bff B_h^{n-1}
			\bigr)}
		{\bb L^2}^2
		\right)^p
		\right]
		\nonumber\\
		&\leq
		C_{p,T}
		\left[
		1+
		\sum_{n=1}^{m}\tau
		\bb E\left[(\mathcal E_h^{n-1})^p\right]
		\right]
		+
		C_p\delta_\chi^p
		\bb E\left[
		\left(
		\sum_{n=1}^{m}\tau
		\norm{\bff J_h^{n-1}}{\bb L^2}^2
		\right)^p
		\right],
	\end{align}
	where in the last step we used \eqref{equ:hall current coeff bound final}, Lemma~\ref{lem:weighted-discrete-holder}, and the fact that
	$\norm{\bff B_h^{n-1}}{\bb L^2}^2\le2\mathcal E_h^{n-1}$.
	The final term is retained until the conclusion of the proof, where it is absorbed into the dissipation moment.

	\emph{Step 6: Martingale estimate.}
	Let
	\[
	S_m:=\max_{0\le j\le m}\mathcal E_h^j.
	\]
	Since $\{M_h^n\}_{n=1}^{N}$ defined in
	\eqref{equ:mart increment revised} is a real-valued martingale difference
	sequence, we can apply Lemma~\ref{lem:discrete-BDG}. To this end, note that by \eqref{equ:mart increment revised} and the Cauchy--Schwarz inequality,
	\begin{align*}
		\abs{M_h^n}^2
		&\le
		2\norm{\bff G_{u,h}^n}{\bb L^2}^2
		\norm{\bff u_h^{n-1}}{\bb L^2}^2
		+
		2\norm{\curl\bff G_{B,h}^n}{\bb L^2}^2
		\norm{\bff B_h^{n-1}}{\bb L^2}^2
		\\
		&\le
		C\mathcal E_h^{n-1}
		\left(
		\norm{\bff G_{u,h}^n}{\bb L^2}^2
		+
		\norm{\curl\bff G_{B,h}^n}{\bb L^2}^2
		\right).
	\end{align*}
	Consequently,
	\[
	\sum_{n=1}^{m}\abs{M_h^n}^2
	\le
	CS_m\left(Q_{u,m}+Q_{B,m}\right).
	\]
	Using Lemma~\ref{lem:discrete-BDG}, Young's inequality, and
	\eqref{equ:Q u p bound revised}--\eqref{equ:Q B p bound revised}, we obtain
	\begin{align*}
		\bb E\left[
		\max_{1\le r\le m}
		\abs{\sum_{n=1}^{r}M_h^n}^{p}
		\right]
		&\le
		C_p\bb E\left[
		S_m^{p/2}
		\left(Q_{u,m}+Q_{B,m}\right)^{p/2}
		\right]
		\nonumber\\
		&\le
		\varepsilon\bb E[S_m^p]
		+
		C_{\varepsilon,p}
		\bb E\left[
		\left(Q_{u,m}+Q_{B,m}\right)^p
		\right]
		\nonumber\\
		&\le
		\varepsilon\bb E[S_m^p]
		+
		C_{\varepsilon,p,T}
		\left[
		1+
		\sum_{n=1}^{m}\tau
		\bb E\left[(\mathcal E_h^{n-1})^p\right]
		\right]
		+
		C_{\varepsilon,p}\delta_\chi^p
		\bb E\left[
		\left(
		\sum_{n=1}^{m}\tau
		\norm{\bff J_h^{n-1}}{\bb L^2}^2
		\right)^p
		\right].
	\end{align*}

	\emph{Step 7: Conclusion.}
	The case $p=1$ was established in Step~4, so let $p>1$. For
	$m=1,\ldots,N$, let $S_m$ be as in Step 6, and define
	\begin{align*}
		D_m &:=
		\sum_{n=1}^{m}
		\left(
		\norm{\bff u_h^n-\bff u_h^{n-1}}{\bb L^2}^2
		+
		\norm{\bff B_h^n-\bff B_h^{n-1}}{\bb L^2}^2
		\right)
		+
		\sum_{n=1}^{m}\tau
		\left(
		\nu\norm{\nabla\bff u_h^n}{\bb L^2}^2
		+
		\sigma\norm{\bff J_h^n}{\bb L^2}^2
		\right).
	\end{align*}
	Summing \eqref{equ:pathwise inequality revised} from $n=1$ to $m$,
	taking the maximum over the partial sums of the martingale term, raising
	to the power $p$, and using Steps~5--6 yields, for any sufficiently small
	$\varepsilon>0$,
	\begin{align}
		\label{equ:stability conclusion before absorption}
		\bb E[S_m^p]
		+
		c_p\bb E[D_m^p]
		&\le
		C_{p,T}
		\left(
		1
		+
		\bb E[(\mathcal E_h^0)^p]
		+
		\left(
		\sum_{n=1}^{N}
		\tau\norm{\bff f^n}{\bb H^{-1}}^2
		\right)^p
		\right)
		+
		C_{p,T}
		\sum_{n=1}^{m}
		\tau\bb E[(\mathcal E_h^{n-1})^p]
		+
		\varepsilon\bb E[S_m^p]
		\nonumber\\
		&\quad
		+
		C_{\varepsilon,p}\delta_\chi^p
		\bb E\left[
		\left(
		\sum_{n=1}^{m}
		\tau\norm{\bff J_h^{n-1}}{\bb L^2}^2
		\right)^p
		\right],
	\end{align}
	where $c_p>0$ is independent of $h$, $\tau$, and $m$.
	
	The shifted current term satisfies
	\begin{align*}
		\sum_{n=1}^{m}
		\tau\norm{\bff J_h^{n-1}}{\bb L^2}^2
		=
		\tau\norm{\bff J_h^0}{\bb L^2}^2
		+
		\sum_{n=1}^{m-1}
		\tau\norm{\bff J_h^n}{\bb L^2}^2
		\le
		\tau\norm{\bff J_h^0}{\bb L^2}^2
		+
		\frac{1}{\sigma}D_m.
	\end{align*}
	Consequently,
	\[
	\bb E\left[
	\left(
	\sum_{n=1}^{m}
	\tau\norm{\bff J_h^{n-1}}{\bb L^2}^2
	\right)^p
	\right]
	\le
	C_p\bb E\left[
	\left(
	\tau\norm{\bff J_h^0}{\bb L^2}^2
	\right)^p
	\right]
	+
	C_{p,\sigma}\bb E[D_m^p].
	\]
	We first choose $\varepsilon>0$ sufficiently small and then choose
	$\delta_p^\ast>0$ sufficiently small such that, whenever
	$\delta_\chi\le\delta_p^\ast$, the terms involving
	$\bb E[S_m^p]$ and $\bb E[D_m^p]$ on the right-hand side of
	\eqref{equ:stability conclusion before absorption} can be absorbed into
	the left-hand side. We thus obtain
	\begin{align*}
		\bb E[S_m^p]
		+
		\bb E[D_m^p]
		&\le
		C_{p,T}
		\left(
		1
		+
		\bb E[(\mathcal E_h^0)^p]
		+
		\bb E\left[
		\left(
		\tau\norm{\bff J_h^0}{\bb L^2}^2
		\right)^p
		\right]
		+
		\left(
		\sum_{n=1}^{N}
		\tau\norm{\bff f^n}{\bb H^{-1}}^2
		\right)^p
		\right)
		\\
		&\quad
		+
		C_{p,T}
		\sum_{n=1}^{m}
		\tau\bb E[(\mathcal E_h^{n-1})^p].
	\end{align*}
	Since $(\mathcal E_h^{n-1})^p\le S_{n-1}^p$, the discrete Gronwall
	lemma yields \eqref{equ:stoch stability revised main}. This completes the
	proof.
\end{proof}

\section{Convergence to a martingale solution}
\label{sec:stoch convergence}

In this section we prove that the finite element approximations generated by \eqref{equ:stoch fem} converge, along a subsequence, to a weak martingale solution of \eqref{equ:stoch mhd}. We consider a sequence of discretisation parameters, still denoted by $(h,\tau)$, such that $h\to0$ and $\tau\to0$.

\subsection{Interpolants}

For $t\in(t_{n-1},t_n]$, define the interpolants
\begin{align}\label{equ:conv interpolants}
	\bff{u}_{h,\tau}^{+}(t):=\bff{u}_h^n,\qquad \bff{u}_{h,\tau}^{-}(t):=\bff{u}_h^{n-1},\qquad
	\bff{B}_{h,\tau}^{+}(t):=\bff{B}_h^n,\qquad 
	\bff{B}_{h,\tau}^{-}(t):=\bff{B}_h^{n-1},
\end{align}
and
\begin{align}\label{equ:conv interpolants J}
	\bff{f}_\tau^+ := \bff{f}^n, \qquad
	\bff{J}_{h,\tau}^{+}(t):=\bff{J}_h^n,\qquad \bff{J}_{h,\tau}^{-}(t):=\bff{J}_h^{n-1}.
\end{align}
We also set for $t\in(t_{n-1},t_n]$,
\begin{align}\label{equ:conv noise coefficients}
	\bff{g}_{i,h,\tau}(t):= \Pi_h^{\bb{V}} \bff{g}_i(\bff{u}_h^{n-1},\bff{B}_h^{n-1}),\qquad
	\bff{H}_{i,h,\tau}(t):= \Pi_h^{\bb{X}}\bigl(\sigma\bff{\chi}_i+\eta\bff{\chi}_i\times\mathcal{R}_h\bff{B}_h^{n-1}\bigr).
\end{align}
Furthermore, for any sequence of time-discrete functions $\{\bff{v}_h^n\}$, we define the continuous piecewise affine interpolant
\[
\bff v_{h,\tau}(t)
:=
\frac{t-t_{n-1}}{\tau}\bff v_h^n+\frac{t_n-t}{\tau}\bff v_h^{n-1}.
\]

Recall that from Proposition~\ref{prop:stoch stability revised}, we obtain for every fixed $p\ge1$,
\begin{equation}\label{equ:conv uniform bounds}
\begin{aligned}
	\bb{E}\left[
	\norm{\bff{u}_{h,\tau}^{+}}{L^\infty(0,T;\bb{L}^2)}^{2p}
	+
	\norm{\bff{B}_{h,\tau}^{+}}{L^\infty(0,T;\bb{L}^2)}^{2p}
	+
	\norm{\bff{u}_{h,\tau}^{+}}{L^2(0,T;\bb{H}_0^1)}^{2p}
	+
	\norm{\bff{J}_{h,\tau}^{+}}{L^2(0,T;\bb{L}^2)}^{2p}
	\right]
	&\leq C,
	\\
	\bb{E}\left[
	\left(\sum_{n=1}^{N}
	\norm{\bff{u}_h^n-\bff{u}_h^{n-1}}{\bb{L}^2}^2
	\right)^p
	+
	\left(\sum_{n=1}^{N}
	\norm{\bff{B}_h^n-\bff{B}_h^{n-1}}{\bb{L}^2}^2
	\right)^p
	\right]
	&\leq C.
\end{aligned}
\end{equation}
In addition, the shifted interpolants satisfy, for every fixed $p\ge1$,
\begin{align}\label{equ:conv shifted uniform bounds}
	\bb E\left[
	\norm{\bff u_{h,\tau}^{-}}{L^\infty(0,T;\bb L^2)}^{2p}
	+
	\norm{\bff B_{h,\tau}^{-}}{L^\infty(0,T;\bb L^2)}^{2p}
	+
	\norm{\bff J_{h,\tau}^{-}}{L^2(0,T;\bb L^2)}^{2p}
	\right]
	\le C.
\end{align}
Consequently, as $\tau\to0$,
\begin{align}\label{equ:conv left right vanish}
	&\bb{E}\left[
	\norm{\bff{u}_{h,\tau}^{+}-\bff{u}_{h,\tau}^{-}}
	{L^2(0,T;\bb{L}^2)}^{2p}
	+
	\norm{\bff{B}_{h,\tau}^{+}-\bff{B}_{h,\tau}^{-}}
	{L^2(0,T;\bb{L}^2)}^{2p}
	\right] \leq C\tau^p \to 0.
\end{align}
Moreover, we have
\begin{align}\label{equ:conv discrete divB}
	\divg\bff{B}_{h,\tau}^{+}=0
	\quad\text{elementwise},
	\qquad
	\bff{B}_{h,\tau}^{+}\cdot\bff{n}=0
	\quad\text{on }\partial\mathscr{D}.
\end{align}
The same statements hold for $\bff B_{h,\tau}^{-}$.

Since the affine interpolants are convex combinations of the left and right piecewise constant interpolants, the same uniform bounds hold for $\bff u_{h,\tau}$ and $\bff B_{h,\tau}$ in $L^{2p}(\Omega;L^\infty(0,T;\bb L^2))$. Moreover,
\begin{equation}\label{equ:uht plus minus}
\norm{\bff u_{h,\tau}-\bff u_{h,\tau}^{+}}
{L^2(0,T;\bb L^2)}
+
\norm{\bff u_{h,\tau}-\bff u_{h,\tau}^{-}}
{L^2(0,T;\bb L^2)}
\le
2\norm{\bff u_{h,\tau}^{+}-\bff u_{h,\tau}^{-}}
{L^2(0,T;\bb L^2)},
\end{equation}
and similarly for $\bff B_{h,\tau}$. Hence, by
\eqref{equ:conv left right vanish}, we have
\[
\bff u_{h,\tau}-\bff u_{h,\tau}^{+}\to0,\qquad
\bff u_{h,\tau}-\bff u_{h,\tau}^{-}\to0,\qquad
\bff B_{h,\tau}-\bff B_{h,\tau}^{+}\to0,\qquad
\bff B_{h,\tau}-\bff B_{h,\tau}^{-}\to0
\]
strongly in $L^{2p}\bigl(\Omega;L^2(0,T;\bb L^2)\bigr)$. Note that we also have $\bff f_\tau^+\to \bff f$ strongly in $L^2_T(\bb L^2)$.

\subsection{Compactness estimates}

Let $m\ge3$ be fixed. We define the smooth solenoidal test space
\[
\bb V^m_\sigma:=\{\bff\phi\in \bb H^m(\mathscr D)\cap \bb H_0^1(\mathscr D):\divg\bff\phi=0\},
\]
equipped with the $\bb H^m$ norm, and denote its dual by $(\bb{V}^m_{\sigma})'$. We note the continuous embeddings $\bb{V}^m_\sigma \hookrightarrow \bb H^m \hookrightarrow \bb W^{1,\infty}$.
The following lemma records the time-translate estimates needed for compactness argument.

\begin{lemma}[Time-translate estimates]
	\label{lem:conv time translation}
	Let $m\ge3$. There exists $C>0$, independent of $h$ and $\tau$, such that
	for every $\kappa\in (0,T)$,
	\begin{align}
		\label{equ:conv time translation u}
		\bb E \left[
		\int_0^{T-\kappa}
		\norm{\bff u_{h,\tau}^{+}(t+\kappa)-\bff u_{h,\tau}^{+}(t)}
		{(\bb V_\sigma^m)'}^2 \dt
		\right]
		&\le C(\kappa+\tau+h^2),
		\\
		\label{equ:conv time translation B}
		\bb E \left[
		\int_0^{T-\kappa}
		\norm{\bff B_{h,\tau}^{+}(t+\kappa)-\bff B_{h,\tau}^{+}(t)}
		{\widetilde{\bb H}^{-m}}^2 \dt
		\right]
		&\le C(\kappa+\tau+h^2).
	\end{align}
\end{lemma}

\begin{proof}
	We first prove the estimates for grid translations. Let $\kappa=\ell\tau$,
	where $\ell\in\{1,\ldots,N\}$.
	
	\emph{Velocity estimate.}
	Fix $n\in\{0,\ldots,N-\ell\}$ and let
	$\bff\phi\in\bb V_\sigma^m$ with $\norm{\bff\phi}{\bb H^m}\le1$.
	Let $(\bff\phi_h,\pi_h):=\mathcal S_h(\bff\phi,0)$ be the Stokes projection of
	$(\bff\phi,0)$ onto $\bb V_h\times Q_h$. Since $\bff\phi$ is solenoidal, we
	have $\bff\phi_h\in\bb V_{h,\sigma}$.
	
	Summing \eqref{equ:stoch fem u} from $r=n+1$ to $r=n+\ell$, we obtain
	\begin{align}
		\label{equ:conv velocity increment testing}
		\inpro{\bff u_h^{n+\ell}-\bff u_h^n}{\bff\phi_h}
		&=
		-\nu\sum_{r=n+1}^{n+\ell}\tau
		\inpro{\nabla\bff u_h^r}{\nabla\bff\phi_h}
		-\frac12\sum_{r=n+1}^{n+\ell}\tau
		\left[
		\inpro{(\bff u_h^{r-1}\cdot\nabla)\bff u_h^r}{\bff\phi_h}
		-
		\inpro{(\bff u_h^{r-1}\cdot\nabla)\bff\phi_h}{\bff u_h^r}
		\right]
		\nonumber\\
		&\quad
		+\sum_{r=n+1}^{n+\ell}\tau
		\inpro{\bff J_h^r\times\bff B_h^{r-1}}{\bff\phi_h}
		+
		\sum_{r=n+1}^{n+\ell}\tau
		\inpro{\bff f^r}{\bff\phi_h}
		\nonumber\\
		&\quad
		+
		\sum_{r=n+1}^{n+\ell}\sum_{i=1}^{\infty}
		\inpro{\bff g_{i,h}^{r-1}}{\bff\phi_h}\Delta_r\beta_i^u,
	\end{align}
	where $\bff g_{i,h}^{r-1}
	:=
	\Pi_h^{\bb V}
	\bff g_i(\bff u_h^{r-1},\bff B_h^{r-1})$.
	We will estimate each term in \eqref{equ:conv velocity increment testing}.
	First, we write
	\begin{equation}\label{equ:uhnl phi}
	\inpro{\bff u_h^{n+\ell}-\bff u_h^n}{\bff\phi_h}
		=
	\inpro{\bff u_h^{n+\ell}-\bff u_h^n}{\bff\phi}
	-
	\inpro{\bff u_h^{n+\ell}-\bff u_h^n}{\bff\phi-\bff\phi_h}.
	\end{equation}
	By the approximation property \eqref{equ:Stokes approx} and Cauchy--Schwarz inequality,
	\[
	\abs{
		\inpro{\bff u_h^{n+\ell}-\bff u_h^n}{\bff\phi-\bff\phi_h}
	}
	\le
	Ch\left(
	\norm{\bff u_h^{n+\ell}}{\bb L^2}
	+
	\norm{\bff u_h^n}{\bb L^2}
	\right).
	\]
	By \eqref{equ:Stokes approx}, \eqref{equ:Stokes bdd}, and the embedding
	$\bb H^m\hookrightarrow \bb W^{1,\infty}$, noting that $\norm{\bff\phi}{\bb H^m}\leq 1$, we have
	\begin{equation}\label{equ:phi bound}
	\norm{\bff\phi_h}{\bb H^1}
	+
	\norm{\bff\phi_h}{\bb L^\infty}
	+
	\norm{\bff\phi_h}{\bb W^{1,3}}
	\le
	C\norm{\bff\phi}{\bb H^m}
	\le
	C,
	\end{equation}
	where $C$ is independent of $\bff{\phi}$.
	Therefore,
	\begin{align*}
		\abs{
			\inpro{(\bff u_h^{r-1}\cdot\nabla)\bff u_h^r}{\bff\phi_h}
		}
		&\le
		\norm{\bff u_h^{r-1}}{\bb L^2}
		\norm{\nabla\bff u_h^r}{\bb L^2}
		\norm{\bff\phi_h}{\bb L^\infty}
		\le
		C\norm{\bff u_h^{r-1}}{\bb L^2}
		\norm{\nabla\bff u_h^r}{\bb L^2},
	\end{align*}
	and similarly, using the Sobolev embedding $\bb H_0^1\hookrightarrow \bb L^6$,
	\begin{align*}
		\abs{
			\inpro{(\bff u_h^{r-1}\cdot\nabla)\bff\phi_h}{\bff u_h^r}
		}
		&\le
		\norm{\bff u_h^{r-1}}{\bb L^2}
		\norm{\nabla\bff\phi_h}{\bb L^3}
		\norm{\bff u_h^r}{\bb L^6}
		\le
		C\norm{\bff u_h^{r-1}}{\bb L^2}
		\norm{\nabla\bff u_h^r}{\bb L^2}.
	\end{align*}
	Furthermore, by \eqref{equ:phi bound} and H\"older's inequality, we have the estimates
	\[
	\abs{
		\inpro{\bff J_h^r\times\bff B_h^{r-1}}{\bff\phi_h}
	}
	\le
	C\norm{\bff J_h^r}{\bb L^2}
	\norm{\bff B_h^{r-1}}{\bb L^2},
	\]
	and
	\[
	\abs{\inpro{\bff f^r}{\bff\phi_h}}
	\le
	\norm{\bff f^r}{\widetilde{\bb H}^{-1}}
	\norm{\bff\phi_h}{\bb H^1}
	\le
	C\norm{\bff f^r}{\widetilde{\bb H}^{-1}}.
	\]
	Taking the supremum over
	$\bff\phi\in\bb V_\sigma^m$ with $\norm{\bff\phi}{\bb H^m}\le1$, noting \eqref{equ:uhnl phi} and \eqref{equ:phi bound}, we infer
	\begin{align}
		\label{equ:conv velocity increment bound}
		\norm{\bff u_h^{n+\ell}-\bff u_h^n}{(\bb{V}^m_{\sigma})'}
		&\le
		Ch\left(
		\norm{\bff u_h^{n+\ell}}{\bb L^2}
		+
		\norm{\bff u_h^n}{\bb L^2}
		\right)
		+
		C\sum_{r=n+1}^{n+\ell}\tau
		\norm{\nabla\bff u_h^r}{\bb L^2}
		\nonumber\\
		&\quad
		+
		C\sum_{r=n+1}^{n+\ell}\tau
		\norm{\bff u_h^{r-1}}{\bb L^2}
		\norm{\nabla\bff u_h^r}{\bb L^2}
		+
		C\sum_{r=n+1}^{n+\ell}\tau
		\norm{\bff J_h^r}{\bb L^2}
		\norm{\bff B_h^{r-1}}{\bb L^2}
		\nonumber\\
		&\quad
		+
		C\sum_{r=n+1}^{n+\ell}\tau
		\norm{\bff f^r}{\widetilde{\bb H}^{-1}}
		+
		\sup_{\substack{\bff\phi\in\bb V_\sigma^m\\
				\norm{\bff\phi}{\bb H^m}\le1}}
		\abs{
			\sum_{r=n+1}^{n+\ell}\sum_{i=1}^{\infty}
			\inpro{\bff g_{i,h}^{r-1}}{\mathcal S_h\bff\phi}
			\Delta_r\beta_i^u
		}.
	\end{align}
	
	Using Cauchy's inequality in the time sums gives
	\begin{align*}
		\left(
		\sum_{r=n+1}^{n+\ell}\tau
		\norm{\nabla\bff u_h^r}{\bb L^2}
		\right)^2
		&\le
		\ell\tau
		\sum_{r=n+1}^{n+\ell}\tau
		\norm{\nabla\bff u_h^r}{\bb L^2}^{2},
		\\
		\left(
		\sum_{r=n+1}^{n+\ell}\tau
		\norm{\bff u_h^{r-1}}{\bb L^2}
		\norm{\nabla\bff u_h^r}{\bb L^2}
		\right)^2
		&\le
		\ell\tau
		\max_{0\le j\le N}
		\norm{\bff u_h^j}{\bb L^2}^{2}
		\sum_{r=n+1}^{n+\ell}\tau
		\norm{\nabla\bff u_h^r}{\bb L^2}^{2},
		\\
		\left(
		\sum_{r=n+1}^{n+\ell}\tau
		\norm{\bff J_h^r}{\bb L^2}
		\norm{\bff B_h^{r-1}}{\bb L^2}
		\right)^2
		&\le
		\ell\tau
		\max_{0\le j\le N}
		\norm{\bff B_h^j}{\bb L^2}^{2}
		\sum_{r=n+1}^{n+\ell}\tau
		\norm{\bff J_h^r}{\bb L^2}^{2},
		\\
		\left(
		\sum_{r=n+1}^{n+\ell}\tau
		\norm{\bff f^r}{\widetilde{\bb H}^{-1}}
		\right)^2
		&\le
		\ell\tau
		\sum_{r=n+1}^{n+\ell}\tau
		\norm{\bff f^r}{\widetilde{\bb H}^{-1}}^2.
	\end{align*}
	For the stochastic term, since
	$\norm{\mathcal S_h\bff\phi}{\bb L^2}\le C\norm{\bff\phi}{\bb H^m}$, the
	conditional It\^o isometry, and the growth assumption \eqref{equ:g growth}
	imply
	\begin{align*}
		\bb E
		\left[
		\sup_{\substack{\bff\phi\in\bb V_\sigma^m\\
				\norm{\bff\phi}{\bb H^m}\le1}}
		\abs{
			\sum_{r=n+1}^{n+\ell}\sum_{i=1}^{\infty}
			\inpro{\bff g_{i,h}^{r-1}}{\mathcal S_h\bff\phi}
			\Delta_r\beta_i^u
		}^{2}
		\right]
		&\le
		C\bb E
		\norm{
			\sum_{r=n+1}^{n+\ell}\sum_{i=1}^{\infty}
			\bff g_{i,h}^{r-1}\Delta_r\beta_i^u
		}{\bb L^2}^{2}
		\\
		&\le
		C\bb E
		\sum_{r=n+1}^{n+\ell}\tau
		\sum_{i=1}^{\infty}
		\norm{\bff g_{i,h}^{r-1}}{\bb L^2}^{2}
		\\
		&\le
		C\bb E
		\sum_{r=n+1}^{n+\ell}\tau
		\left(1+\mathcal E_h^{r-1}\right).
	\end{align*}
	
	We now multiply the square of \eqref{equ:conv velocity increment bound} by $\tau$
	and sum over $n=0,\ldots,N-\ell$. We use an elementary finite-overlap estimate
	\begin{equation}\label{equ:finite overlap}
	\sum_{n=0}^{N-\ell}
	\sum_{r=n+1}^{n+\ell} a_r
	\le
	\ell\sum_{r=1}^{N}a_r
	\qquad
	\text{for every nonnegative sequence } \{a_r\}_{r=1}^N .
	\end{equation}
	Using \eqref{equ:conv uniform bounds} with $p=2$ and the stability estimate \eqref{equ:stoch stability revised main}, we thus obtain
	\[
	\bb E\left[\tau\sum_{n=0}^{N-\ell}
	\norm{\bff u_h^{n+\ell}-\bff u_h^n}{(\bb{V}^m_{\sigma})'}^2 \right]
	\le
	C(\ell\tau+h^2).
	\]

	\emph{Magnetic estimate.}
	Fix again $n\in\{0,\ldots,N-\ell\}$ and let
	$\bff\psi\in\bb H^m(\mathscr D)$ with
	$\norm{\bff\psi}{\bb H^m}\le1$. Set $\bff\psi_h:=\Pi_h^{\bb{RT}}\bff\psi$.
	Since
	$\bff B_h^{n+\ell}-\bff B_h^n\in\bb{RT}_h^0$ and
	$\Pi_h^{\bb{RT}}$ is the $\bb L^2$-orthogonal projection onto
	$\bb{RT}_h^0$, we have
	\[
	\inpro{\bff B_h^{n+\ell}-\bff B_h^n}{\bff\psi}
	=
	\inpro{\bff B_h^{n+\ell}-\bff B_h^n}{\bff\psi_h}.
	\]
	Summing \eqref{equ:stoch fem B} from $r=n+1$ to $r=n+\ell$ gives
	\begin{align}
		\label{equ:conv magnetic increment testing}
		\inpro{\bff B_h^{n+\ell}-\bff B_h^n}{\bff\psi}
		&=
		-\sum_{r=n+1}^{n+\ell}\tau
		\inpro{\curl\bff E_h^r}{\bff\psi_h}
		-\sum_{r=n+1}^{n+\ell}\sum_{i=1}^{\infty}
		\inpro{\curl\bff H_{i,h}^{r-1}}{\bff\psi_h}
		\Delta_r\beta_i^B,
	\end{align}
	where
	$\bff H_{i,h}^{r-1}
	:=
	\Pi_h^{\bb X}
	\bigl(
	\sigma\bff\chi_i
	+
	\eta\bff\chi_i\times\mathcal R_h\bff B_h^{r-1}
	\bigr)$.
	By \eqref{equ:conv curlh projection} and the definition of $\curlh$,
	\[
	\inpro{\curl\bff E_h^r}{\bff\psi_h}
	=
	\inpro{\bff E_h^r}{\curlh\bff\psi_h}
	=
	\inpro{\bff E_h^r}{\Pi_h^{\bb X}\curl\bff\psi}.
	\]
	Testing Ohm's law \eqref{equ:stoch fem Ohm} with
	$\bff\chi_h=\Pi_h^{\bb X}\curl\bff\psi$ gives
	\[
	\inpro{\bff E_h^r}{\Pi_h^{\bb X}\curl\bff\psi}
	=
	\sigma\inpro{\bff J_h^r}{\Pi_h^{\bb X}\curl\bff\psi}
	+
	\eta\inpro{\bff J_h^r\times\bff B_h^{r-1}}
	{\Pi_h^{\bb X}\curl\bff\psi}
	-
	\inpro{\bff u_h^r\times\bff B_h^{r-1}}
	{\Pi_h^{\bb X}\curl\bff\psi}.
	\]
	By \eqref{equ:stab magnetic projection}, \eqref{equ:conv magnetic projection Linfty},
	and the embedding
	$\bb H^m\hookrightarrow \bb W^{1,\infty}$, noting that $\norm{\bff\psi}{\bb H^m}\leq 1$, we have
	\[
	\norm{\Pi_h^{\bb X}\curl\bff\psi}{\bb L^2}
	+
	\norm{\Pi_h^{\bb X}\curl\bff\psi}{\bb L^\infty}
	\le
	C\norm{\bff\psi}{\bb H^m}
	\le
	C.
	\]
	Consequently, we obtain
	\begin{align}
		\label{equ:conv curlE bound}
		\abs{\inpro{\curl\bff E_h^r}{\bff\psi_h}}
		&\le
		C\norm{\bff J_h^r}{\bb L^2}
		+
		C\norm{\bff J_h^r}{\bb L^2}
		\norm{\bff B_h^{r-1}}{\bb L^2}
		+
		C\norm{\bff u_h^r}{\bb L^2}
		\norm{\bff B_h^{r-1}}{\bb L^2}.
	\end{align}
	For the stochastic term, the $\bb L^2$-stability of
	$\Pi_h^{\bb{RT}}$, the conditional It\^o isometry, and
	\eqref{equ:hall current coeff bound final} imply
	\begin{align*}
		&\bb E
		\left[
		\sup_{\substack{\bff\psi\in\bb H^m\\
				\norm{\bff\psi}{\bb H^m}\le1}}
		\abs{
			\sum_{r=n+1}^{n+\ell}\sum_{i=1}^{\infty}
			\inpro{\curl\bff H_{i,h}^{r-1}}{\Pi_h^{\bb{RT}}\bff\psi}
			\Delta_r\beta_i^B
		}^{2}
		\right]
		\\
		&\qquad\le
		C\bb E
		\norm{
			\sum_{r=n+1}^{n+\ell}\sum_{i=1}^{\infty}
			\curl\bff H_{i,h}^{r-1}\Delta_r\beta_i^B
		}{\bb L^2}^{2}
		\\
		&\qquad\le
		C\bb E \left[
		\sum_{r=n+1}^{n+\ell}\tau
		\sum_{i=1}^{\infty}
		\norm{\curl\bff H_{i,h}^{r-1}}{\bb L^2}^{2} \right]
		\\
		&\qquad\le
		C\bb E \left[
		\sum_{r=n+1}^{n+\ell}\tau
		\left(
		1+\norm{\bff B_h^{r-1}}{\bb L^2}^{2}
		\right)\right]
		+
		C\delta_\chi
		\bb E \left[
		\sum_{r=n+1}^{n+\ell}\tau
		\norm{\bff J_h^{r-1}}{\bb L^2}^{2}\right].
	\end{align*}
	
	Taking the supremum over
	$\bff\psi\in\bb H^m$ with $\norm{\bff\psi}{\bb H^m}\le1$ in
	\eqref{equ:conv magnetic increment testing}, using
	\eqref{equ:conv curlE bound}, and applying Cauchy's inequality in the time
	sums, we obtain
	\begin{align}\label{equ:Bn H minus m sq}
		\norm{\bff B_h^{n+\ell}-\bff B_h^n}{\widetilde{\bb H}^{-m}}^2
		&\le
		C\left(
		\sum_{r=n+1}^{n+\ell}\tau
		\norm{\bff J_h^r}{\bb L^2}
		\right)^2
		+
		C\left(
		\sum_{r=n+1}^{n+\ell}\tau
		\norm{\bff J_h^r}{\bb L^2}
		\norm{\bff B_h^{r-1}}{\bb L^2}
		\right)^2
		\nonumber\\
		&\quad
		+
		C\left(
		\sum_{r=n+1}^{n+\ell}\tau
		\norm{\bff u_h^r}{\bb L^2}
		\norm{\bff B_h^{r-1}}{\bb L^2}
		\right)^2
		\nonumber\\
		&\quad
		+
		C
		\sup_{\substack{\bff\psi\in\bb H^m\\
				\norm{\bff\psi}{\bb H^m}\le1}}
		\abs{
			\sum_{r=n+1}^{n+\ell}\sum_{i=1}^{\infty}
			\inpro{\curl\bff H_{i,h}^{r-1}}{\Pi_h^{\bb{RT}}\bff\psi}
			\Delta_r\beta_i^B
		}^{2}.
	\end{align}
	Multiplying by $\tau$, summing over $n=0,\ldots,N-\ell$, using the same
	finite-overlap estimate \eqref{equ:finite overlap} as before, and noting
	\eqref{equ:conv uniform bounds}, we infer that
	\[
	\bb E\left[\tau\sum_{n=0}^{N-\ell}
	\norm{\bff B_h^{n+\ell}-\bff B_h^n}{\widetilde{\bb H}^{-m}}^2\right]
	\le
	C(\ell\tau+h^2).
	\]
	
	It remains to pass from grid translations to arbitrary
	$\kappa\in(0,T)$. Choose $\ell\in\{0,\ldots,N-1\}$ such that $\ell\tau\le\kappa<(\ell+1)\tau$.
	Since $\bff u_{h,\tau}^{+}$ and $\bff B_{h,\tau}^{+}$ are piecewise
	constant in time, for almost every $t\in(0,T-\kappa)$ the time levels
	corresponding to $t$ and $t+\kappa$ differ by either $\ell$ or $\ell+1$.
	Therefore, up to a harmless relabelling of the indices,
	\begin{align*}
		&\int_0^{T-\kappa}
		\norm{\bff u_{h,\tau}^{+}(t+\kappa)-\bff u_{h,\tau}^{+}(t)}
		{(\bb{V}^m_{\sigma})'}^2\dt
		\le
		C\tau\sum_{n=0}^{N-\ell}
		\norm{\bff u_h^{n+\ell}-\bff u_h^n}{(\bb{V}^m_{\sigma})'}^2
		+
		C\tau\sum_{n=0}^{N-\ell-1}
		\norm{\bff u_h^{n+\ell+1}-\bff u_h^n}{(\bb{V}^m_{\sigma})'}^2,
	\end{align*}
	with the obvious convention when $\ell=0$. Applying the grid-translation
	estimate with $\ell$ and $\ell+1$, and noting that
	$(\ell+1)\tau\le\kappa+\tau$,
	we deduce \eqref{equ:conv time translation u}. The same argument, using the
	magnetic grid-translation estimate, gives \eqref{equ:conv time translation B}.
	The proof is now complete.
\end{proof}

We have the following lemma on joint tightness of laws.

\begin{lemma}[Joint tightness]
	\label{lem:conv joint tightness}
	Let $m\ge3$. For every sequence $(h_k,\tau_k)\to(0,0)$, the laws of
	\[
	\mathcal Z_{h_k,\tau_k}
	:=
	\bigl(
	\bff u_{h_k,\tau_k},
	\bff B_{h_k,\tau_k},
	\bff u_{h_k,\tau_k}^{+},
	\bff u_{h_k,\tau_k}^{-},
	\bff B_{h_k,\tau_k}^{+},
	\bff B_{h_k,\tau_k}^{-},
	\bff J_{h_k,\tau_k}^{+},
	\bff J_{h_k,\tau_k}^{-}
	\bigr)
	\]
	are tight on the product space
	\[
	\mathcal X
	:=
	\mathcal X_u
	\times
	\mathcal X_B
	\times
	\mathcal X_{u}^{+}
	\times
	\mathcal X_u^{-}
	\times
	\mathcal X_B^{+}
	\times
	\mathcal X_B^{-}
	\times
	\mathcal X_J^{+}
	\times
	\mathcal X_J^{-},
	\]
	where
	\begin{subequations}
		\begin{alignat*}{2}
	&\mathcal X_u
	:=
	L^2_T(\bb L^2)
	\cap
	C([0,T];(\bb V^m_\sigma)'),
	\qquad
	&&\mathcal X_B
	:=
	L^2_T(\bb L^2)
	\cap
	C([0,T];\widetilde{\bb H}^{-m}),
	\\
	&\mathcal X_u^{+}
	:=
	L^2_T(\bb L^2)
	\cap
	L^2_T(\bb H_0^1)_{\mathrm w},
	\qquad
	&&\mathcal X_B^{+}
	:=
	L^2_T(\bb L^2)
	\cap
	L^\infty_T(\bb L^2)_{\mathrm w^*},
	\\
	&\mathcal X_u^{-}= \mathcal X_B^{-}:= L^2_T(\bb L^2),
	\qquad
	&&\mathcal X_J^{+}= \mathcal X_J^{-}:= L^2_T(\bb L^2)_{\mathrm w}.
		\end{alignat*}
	\end{subequations}
	All intersection spaces are endowed with the initial topology induced by their
	natural embeddings into the two factors.
\end{lemma}

\begin{proof}
	Fix a sequence $(h_k,\tau_k)\to(0,0)$. By
	Lemma~\ref{lem:tightness negative path spaces}, the laws of
	$\{\bff u_{h_k,\tau_k}\}_{k\ge1}$ are tight in
	$C([0,T];(\bb V^m_\sigma)')$, and the laws of
	$\{\bff B_{h_k,\tau_k}\}_{k\ge1}$ are tight in
	$C([0,T];\widetilde{\bb H}^{-m})$.
	
	We next justify the compactness in $L^2_T(\bb L^2)$. For the velocity, set
	\[
	X_0=\bb H_0^1,\qquad
	X=\bb L^2,\qquad
	X_1=(\bb V^m_\sigma)'.
	\]
	Then $X_0\Subset X\hookrightarrow X_1$.
	Simon's compactness criterion~\cite{Sim87} states that a deterministic set
	bounded in $L^2(0,T;X_0)$ and uniformly equicontinuous with respect to time
	translations in $L^2(0,T;X_1)$ is relatively compact in $L^2(0,T;X)$.
	In the stochastic setting, this is applied as follows. The stability
	estimate \eqref{equ:stoch stability revised main} gives
	\[
	\sup_k
	\bb E \left[\norm{\bff u_{h_k,\tau_k}^{+}}{L^2_T(\bb H_0^1)}^2 \right]
	\le C,
	\]
	while \eqref{equ:conv time translation u} gives, for every $\kappa>0$,
	\[
	\limsup_{k\to\infty}
	\bb E \left[
	\int_0^{T-\kappa}
	\norm{
		\bff u_{h_k,\tau_k}^{+}(t+\kappa)
		-
		\bff u_{h_k,\tau_k}^{+}(t)}
	{(\bb V^m_\sigma)'}^2\dt \right]
	\le C\kappa.
	\]
	Hence, by the Markov inequality, the deterministic compact sets supplied by
	Simon's criterion carry probability arbitrarily close to one. Therefore the
	laws of $\{\bff u_{h_k,\tau_k}^{+}\}_{k\ge1}$ are tight in
	$L^2_T(\bb L^2)$.
	
	For the magnetic field, we apply the same argument to the reconstructed
	fields $\mathcal R_{h_k}\bff B_{h_k,\tau_k}^{+}$. Namely, by the stability of
	the reconstruction \eqref{equ:Maxwell reconstruction stability} and \eqref{equ:stoch stability revised main},
	\[
	\sup_k
	\bb E \left[ \norm{\mathcal R_{h_k}\bff B_{h_k,\tau_k}^{+}}{L^2_T(\bb H^1)}^2 \right]
	\le C.
	\]
	Moreover, using \eqref{equ:conv time translation B} and the consistency of
	$\mathcal R_h$ in \eqref{equ:Rh L2 consistency asymptotic}, one obtains
	\[
	\limsup_{k\to\infty}
	\bb E \left[
	\int_0^{T-\kappa}
	\norm{
		\mathcal R_{h_k}\bff B_{h_k,\tau_k}^{+}(t+\kappa)
		-
		\mathcal R_{h_k}\bff B_{h_k,\tau_k}^{+}(t)}
	{\widetilde{\bb H}^{-m}}^2\dt \right]
	\le C\kappa .
	\]
	Thus, as before, Simon's criterion implies tightness
	of the laws of
	$\{\mathcal R_{h_k}\bff B_{h_k,\tau_k}^{+}\}_{k\ge1}$ in
	$L^2_T(\bb L^2)$. Since
	\[
	\norm{
		\mathcal R_{h_k}\bff B_{h_k,\tau_k}^{+}
		-
		\bff B_{h_k,\tau_k}^{+}}
	{L^2_T(\bb L^2)}
	\longrightarrow 0 \;\text{ in probability},
	\]
	the laws of $\{\bff B_{h_k,\tau_k}^{+}\}_{k\ge1}$ are also
	tight in $L^2_T(\bb L^2)$.
	Finally, the affine interpolants have the same $L^2_T(\bb L^2)$ tightness by \eqref{equ:uht plus minus}, since tightness is preserved under convergence in probability to a tight family. We have thus shown that the laws of $\bff u_{h_k,\tau_k},
	\bff u_{h_k,\tau_k}^{+},
	\bff u_{h_k,\tau_k}^{-}$
	and $\bff B_{h_k,\tau_k},
	\bff B_{h_k,\tau_k}^{+},
	\bff B_{h_k,\tau_k}^{-}$
	are tight in $L^2_T(\bb L^2)$.
	
	The additional weak and weak-star tightness follows directly from the
	stability estimates \eqref{equ:stoch stability revised main}.
	Therefore the laws of $\bff u_{h_k,\tau_k}^{+}$ are tight in
	$L^2_T(\bb H_0^1)_{\mathrm w}$, and the laws of
	$\bff B_{h_k,\tau_k}^{+}$ are tight in
	$L^\infty_T(\bb L^2)_{\mathrm w^*}$. Combining these weak compactness
	properties with the strong $L^2_T(\bb L^2)$ tightness gives tightness of
	$\bff u_{h_k,\tau_k}^{+}$ on $\mathcal X_u^{+}$ and of
	$\bff B_{h_k,\tau_k}^{+}$ on $\mathcal X_B^{+}$. Similarly, the
	laws of $\bff J_{h_k,\tau_k}^{+}$ and $\bff J_{h_k,\tau_k}^{-}$ are tight in
	$\mathcal X_J^{+}$ and $\mathcal X_J^{-}$, respectively.
	This proves the claim.
\end{proof}

\subsection{Convergence to a martingale solution}

Having established the uniform stability estimates and tightness of the discrete laws, we now pass to the limit as $(h,\tau)\to(0,0)$. We first apply the Jakubowski--Skorokhod representation theorem to realise a subsequence of the discrete solutions on a new probability space. We then identify the limits, and finally recover the limiting martingale structure and its quadratic covariations. These steps yield a weak martingale solution of \eqref{equ:stoch mhd}.

\begin{lemma}[Skorokhod representation]
	\label{lem:conv sk representation}
	There exist a subsequence of $(h_k,\tau_k)\to(0,0)$, not
	relabelled, a probability space $(\widetilde\Omega,\widetilde{\mathscr F},\widetilde{\bb P})$, random variables
	\[
	(\widetilde{\bff u}_{h_k,\tau_k},
	\widetilde{\bff B}_{h_k,\tau_k},
	\widetilde{\bff u}_{h_k,\tau_k}^{+},
	\widetilde{\bff u}_{h_k,\tau_k}^{-},
	\widetilde{\bff B}_{h_k,\tau_k}^{+},
	\widetilde{\bff B}_{h_k,\tau_k}^{-},
	\widetilde{\bff J}_{h_k,\tau_k}^{+},
	\widetilde{\bff J}_{h_k,\tau_k}^{-})
	\]
	with the same laws as the original variables, and random variables
	$(\bff u,\bff B,\bff J)$ such that, $\widetilde{\bb P}$-a.s.,
	\begin{subequations}
	\begin{alignat}{2}
		\label{equ:conv sk u strong}
		\widetilde{\bff u}_{h_k,\tau_k},
		\widetilde{\bff u}_{h_k,\tau_k}^{+},
		\widetilde{\bff u}_{h_k,\tau_k}^{-}
		&\to\bff u
		\quad&&\text{strongly in }L^2_T(\bb L^2),
		\\
		\label{equ:conv sk B strong}
		\widetilde{\bff B}_{h_k,\tau_k},
		\widetilde{\bff B}_{h_k,\tau_k}^{+},
		\widetilde{\bff B}_{h_k,\tau_k}^{-}
		&\to\bff B
		\quad&&\text{strongly in }L^2_T(\bb L^2),
	\end{alignat}
	\end{subequations}
	and
	\begin{subequations}
		\begin{alignat}{2}
		\label{equ:conv sk path}
		\widetilde{\bff u}_{h_k,\tau_k}
		&\to\bff u
		\quad&&\text{in }C([0,T];(\bb V^m_\sigma)'),
		\\
		\label{equ:conv Bk path}
		\widetilde{\bff B}_{h_k,\tau_k}
		&\to\bff B
		\quad&&\text{in }C([0,T];\widetilde{\bb H}^{-m}).
	\end{alignat}
	\end{subequations}
	Moreover,
	\begin{subequations}
		\begin{alignat}{2}
		\label{equ:conv sk weak}
		\widetilde{\bff u}_{h_k,\tau_k}^{+}
		&\rightharpoonup\bff u
		\quad\text{weakly in }L^2_T(\bb H_0^1),
		\\
		\label{equ:conv Bk weak}
		\widetilde{\bff B}_{h_k,\tau_k}^{+}
		&\overset{*}{\rightharpoonup}\bff B
		\quad\text{weakly-star in }L^\infty_T(\bb L^2),
		\\
		\label{equ:conv Jk weak}
		\widetilde{\bff J}_{h_k,\tau_k}^{+},
		\widetilde{\bff J}_{h_k,\tau_k}^{-}
		&\rightharpoonup\bff J
		\quad\text{weakly in }L^2_T(\bb L^2).
	\end{alignat}
	\end{subequations}
\end{lemma}

\begin{proof}
	By Lemma~\ref{lem:conv joint tightness}, the joint laws are tight on
	$\mathcal X$. Since the weak and weak-star factors of $\mathcal X$ are
	Jakubowski spaces, the Jakubowski--Skorokhod theorem gives a new probability
	space and copies of the numerical variables converging almost surely in the
	product topology of $\mathcal X$. The strong identifications of the left and
	right interpolants with the same limits follow from
	\eqref{equ:conv left right vanish} and equality of laws. This proves the required results.
\end{proof}

\begin{lemma}[Identification of constraints and current]
	\label{lem:conv constraints current}
	The limit $(\bff u,\bff B,\bff J)$ obtained in
	Lemma~\ref{lem:conv sk representation} satisfies, $\widetilde{\bb P}$-a.s.,
	\[
	\divg\bff u=0
	\quad\text{in }\mathcal D'((0,T)\times\mathscr D),
	\]
	and
	\[
	\divg\bff B=0
	\quad\text{in }\mathcal D'((0,T)\times\mathscr D),
	\qquad
	\bff B\cdot\bff n=0
	\quad\text{in the weak normal-trace sense}.
	\]
	Moreover,
	\begin{align}
		\label{equ:conv current identified}
		\inpro{\bff J(t)}{\bff\omega}
		-
		\inpro{\bff B(t)}{\curl\bff\omega}
		=0,
		\qquad
		\forall\bff\omega\in\hzerocurl,
	\end{align}
	for a.e. $t\in(0,T)$, $\widetilde{\bb P}$-a.s. Consequently, $\bff B\in
	L^2\bigl(\widetilde\Omega;L^2_T(\bb H^1)\bigr)$.
\end{lemma}

\begin{proof}
	Since $\widetilde{\bff B}_{h,\tau}^{+}$ is divergence-free in
	$\bb{RT}_h^0$ and has zero normal trace, for every $\rho\in C_c^\infty(0,T)$ and $\zeta\in H^1(\mathscr D)$,
	\[
	\int_0^T \rho(t)
	\inpro{\widetilde{\bff B}_{h,\tau}^{+}}{\nabla\zeta}\dt=0.
	\]
	Passing to the limit by \eqref{equ:conv sk B strong} gives
	\[
	\int_0^T \rho(t)
	\inpro{\bff B}{\nabla\zeta}\dt=0,
	\]
	which gives $\divg\bff B=0$ in the distributional sense and $\bff B\cdot\bff n=0$ in the weak
	normal-trace sense.
	
	The velocity constraint follows from \eqref{equ:stoch fem div u}. Let
	$\lambda\in C^\infty(\overline{\mathscr D})\cap L^2_0$ and $\rho\in C_c^\infty(0,T)$.
	Choose a deterministic $\lambda_h\in Q_h^0$ such that
	$\lambda_h\to\lambda$ strongly in $L^2(\mathscr D)$. Since the tilded
	variables have the same laws as the original variables, \eqref{equ:stoch fem div u}
	gives, $\widetilde{\bb P}$-a.s.,
	\[
	\int_0^T \rho(t)
	\inpro{\divg\widetilde{\bff u}_{h,\tau}^{+}}{\lambda_h}\dt
	=0 .
	\]
	Passing to the limit using
	$\widetilde{\bff u}_{h,\tau}^{+}\rightharpoonup\bff u$ in
	$L^2_T(\bb H_0^1)$ and $\lambda_h\to\lambda$ in $L^2(\mathscr D)$, we obtain
	\[
	\int_0^T \rho(t)
	\inpro{\divg\bff u}{\lambda}\dt
	=0 .
	\]
	By density, $\divg\bff u=0$ in $\mathcal D'((0,T)\times\mathscr D)$,
	$\widetilde{\bb P}$-a.s.
	
	Next, we identify the current. First take
	$\bff\omega\in C^\infty(\overline{\mathscr D})\cap\hzerocurl$ and set
	$\bff\omega_h:=\mathcal I_h^{\bb X}\bff\omega\in\bb X_h^0$,
	where $\mathcal I_h^{\bb X}$ is a Nedelec interpolant satisfying
	\[
	\bff\omega_h\to\bff\omega
	\quad\text{in }\bb L^2,
	\qquad
	\curl\bff\omega_h\to\curl\bff\omega
	\quad\text{in }\bb L^2.
	\]
	Using \eqref{equ:stoch fem J} and equality of laws, For any $\rho\in C_c^\infty(0,T)$,
	\[
	\int_0^T \rho(t)
	\inpro{\widetilde{\bff J}_{h,\tau}^{+}}{\bff\omega_h}\dt
	=
	\int_0^T \rho(t)
	\inpro{\widetilde{\bff B}_{h,\tau}^{+}}{\curl\bff\omega_h}\dt,
	\qquad \widetilde{\bb P}\text{-a.s.}
	\]
	Passing to the limit using \eqref{equ:conv sk B strong} and
	\eqref{equ:conv sk weak}, we obtain
	\[
	\int_0^T \rho(t)
	\inpro{\bff J}{\bff\omega}\dt
	=
	\int_0^T \rho(t)
	\inpro{\bff B}{\curl\bff\omega}\dt.
	\]
	By density of
	$C^\infty(\overline{\mathscr D})\cap\hzerocurl$ in $\hzerocurl$,
	\eqref{equ:conv current identified} follows.
	
	Finally, using the Gaffney estimate~\cite{AmrBerDauGir98} on $\mathscr D$,
	\[
	\norm{\bff v}{\bb H^1}
	\le
	C\left(
	\norm{\bff v}{\bb L^2}
	+
	\norm{\curl\bff v}{\bb L^2}
	+
	\norm{\divg\bff v}{L^2}
	\right),
	\qquad
	\bff v\in \hcurl \cap \hzerodiv.
	\]
	Noting that $\divg\bff B=0$, $\bff B\cdot\bff n=0$, and
	$\curl\bff B=\bff J$, we obtain
	\[
	\norm{\bff B}{\bb H^1}
	\le
	C\left(
	\norm{\bff B}{\bb L^2}
	+
	\norm{\bff J}{\bb L^2}
	\right).
	\]
	Consequently, $\bff B\in
	L^2\bigl(\widetilde\Omega;L^2_T(\bb H^1)\bigr)$. This completes the proof of the lemma.
\end{proof}

\begin{lemma}[Convergence of the deterministic terms]
	\label{lem:conv deterministic terms}
	Let $\bff\phi\in C^\infty_0(\mathscr D)\cap\bb V_\sigma $ and let
	$\bff\phi_h=\mathcal S_h\bff\phi$, where $\mathcal{S}_h$ is the Stokes projection defined in \eqref{equ:stokes proj}, be a discretely
	divergence-free approximation satisfying
	\[
	\bff\phi_h\to\bff\phi
	\quad\text{in }\bb H^1,
	\qquad
	\norm{\bff\phi_h}{\bb W^{1,3}\cap \bb{L}^\infty}\le C(\bff\phi).
	\]
	Then, for every $t\in[0,T]$, $\widetilde{\bb P}$-a.s.,
	\begin{align}
		\label{equ:conv viscous}
		\int_0^t
		\inpro{\nabla\widetilde{\bff u}_{h,\tau}^{+}}
		{\nabla\bff\phi_h}\ds
		&\to
		\int_0^t
		\inpro{\nabla\bff u}{\nabla\bff\phi}\ds,
		\\
		\label{equ:conv convection}
		\frac12\int_0^t
		\Big[
		\inpro{(\widetilde{\bff u}_{h,\tau}^{-}\cdot\nabla)
			\widetilde{\bff u}_{h,\tau}^{+}}{\bff\phi_h}
		-
		\inpro{(\widetilde{\bff u}_{h,\tau}^{-}\cdot\nabla)
			\bff\phi_h}{\widetilde{\bff u}_{h,\tau}^{+}}
		\Big]\ds
		&\to
		\int_0^t
		\inpro{(\bff u\cdot\nabla)\bff u}{\bff\phi}\ds,
		\\
		\label{equ:conv Lorentz}
		\int_0^t
		\inpro{\widetilde{\bff J}_{h,\tau}^{+}
			\times\widetilde{\bff B}_{h,\tau}^{-}}{\bff\phi_h}\ds
		&\to
		\int_0^t
		\inpro{\bff J\times\bff B}{\bff\phi}\ds .
	\end{align}
	Moreover, let
	$\bff\psi\in C^\infty(\overline{\mathscr D})\cap\hzerodiv$ and let $\bff\psi_h:=\Pi_h^{\bb{RT}}\bff\psi$, so that (by \eqref{equ:conv curlh projection}--\eqref{equ:conv magnetic projection Linfty}):
	\[
	\bff\psi_h\to\bff\psi
	\quad\text{in }\bb L^2,
	\qquad
	\curlh\bff\psi_h\to\curl\bff\psi
	\quad\text{in }\bb L^3,
	\qquad
	\norm{\curlh\bff\psi_h}{\bb L^\infty}\le C(\bff\psi).
	\]
	Then, for every $t\in[0,T]$, $\widetilde{\bb P}$-a.s.,
	\begin{align}
		\label{equ:conv magnetic resistive}
		\int_0^t
		\inpro{\widetilde{\bff J}_{h,\tau}^{+}}{\curlh\bff\psi_h}\ds
		&\to
		\int_0^t
		\inpro{\bff J}{\curl\bff\psi}\ds,
		\\
		\label{equ:conv magnetic Hall}
		\int_0^t
		\inpro{\widetilde{\bff J}_{h,\tau}^{+}
			\times\widetilde{\bff B}_{h,\tau}^{-}}
		{\curlh\bff\psi_h}\ds
		&\to
		\int_0^t
		\inpro{\bff J\times\bff B}{\curl\bff\psi}\ds,
		\\
		\label{equ:conv magnetic transport}
		\int_0^t
		\inpro{\widetilde{\bff u}_{h,\tau}^{+}
			\times\widetilde{\bff B}_{h,\tau}^{-}}
		{\curlh\bff\psi_h}\ds
		&\to
		\int_0^t
		\inpro{\bff u\times\bff B}{\curl\bff\psi}\ds .
	\end{align}
\end{lemma}

\begin{proof}
	The convergence \eqref{equ:conv viscous} follows directly from
	$\widetilde{\bff u}_{h,\tau}^{+}\rightharpoonup\bff u$ in
	$L^2_T(\bb H_0^1)$ and $\bff\phi_h\to\bff\phi$ in $\bb H^1$.
	
	For the convection term, the first part of the skew-symmetric form is treated by
	writing
	\[
	\inpro{(\widetilde{\bff u}_{h,\tau}^{-}\cdot\nabla)
		\widetilde{\bff u}_{h,\tau}^{+}}{\bff\phi_h}
	=
	\int_{\mathscr D}
	\nabla\widetilde{\bff u}_{h,\tau}^{+}:
	(\bff\phi_h \otimes \widetilde{\bff u}_{h,\tau}^{-})\dx .
	\]
	Here $\nabla\widetilde{\bff u}_{h,\tau}^{+}\rightharpoonup\nabla\bff u$
	in $L^2_T(\bb L^2)$, while
	$\bff\phi_h \otimes \widetilde{\bff u}_{h,\tau}^{-} 
	\to \bff\phi \otimes \bff u$ strongly in $L^2_T(\bb L^2)$. The second part of the skew form follows from the fact that
	$\widetilde{\bff u}_{h,\tau}^{-}\to\bff u$ strongly in
	$L^2_T(\bb L^2)$,
	$\widetilde{\bff u}_{h,\tau}^{+}\rightharpoonup\bff u$ weakly in
	$L^2_T(\bb L^6)$, the uniform $\bb W^{1,3}$ bound on $\bff\phi_h$,
	and $\bff\phi_h\to\bff\phi$ in $\bb H^1$. Since $\divg\bff u=0$, the
	limit of the skew form is
	$\int_0^t\inpro{(\bff u\cdot\nabla)\bff u}{\bff\phi}\ds$, showing \eqref{equ:conv convection}.
	
	For the Lorentz term,
	\[
	\inpro{\widetilde{\bff J}_{h,\tau}^{+}
		\times\widetilde{\bff B}_{h,\tau}^{-}}{\bff\phi_h}
	=
	\inpro{\widetilde{\bff J}_{h,\tau}^{+}}
	{\widetilde{\bff B}_{h,\tau}^{-}\times\bff\phi_h}.
	\]
	The first factor converges weakly in $L^2_T(\bb L^2)$, while the second
	converges strongly in $L^2_T(\bb L^2)$. This proves
	\eqref{equ:conv Lorentz}.
	
	The convergence \eqref{equ:conv magnetic resistive} follows from
	$\widetilde{\bff J}_{h,\tau}^{+}\rightharpoonup\bff J$ in
	$L^2_T(\bb L^2)$ and
	$\curlh\bff\psi_h\to\curl\bff\psi$ in $\bb L^2$. For
	\eqref{equ:conv magnetic Hall}, write
	\[
	\inpro{\widetilde{\bff J}_{h,\tau}^{+}
		\times\widetilde{\bff B}_{h,\tau}^{-}}{\curlh\bff\psi_h}
	=
	\inpro{\widetilde{\bff J}_{h,\tau}^{+}}
	{\widetilde{\bff B}_{h,\tau}^{-}\times\curlh\bff\psi_h}.
	\]
	The second factor converges strongly in $L^2_T(\bb L^2)$ because
	\[
	\norm{(\widetilde{\bff B}_{h,\tau}^{-}-\bff B)
		\times\curlh\bff\psi_h}{L^2_T(\bb L^2)}
	\le
	C(\bff\psi)
	\norm{\widetilde{\bff B}_{h,\tau}^{-}-\bff B}{L^2_T(\bb L^2)}
	\to0,
	\]
	and, using the fact $\bff B\in L^2_T(\bb L^6)$,
	\[
	\norm{\bff B\times(\curlh\bff\psi_h-\curl\bff\psi)}
	{L^2_T(\bb L^2)}
	\le
	\norm{\bff B}{L^2_T(\bb L^6)}
	\norm{\curlh\bff\psi_h-\curl\bff\psi}{\bb L^3}
	\to0.
	\]
	This proves \eqref{equ:conv magnetic Hall}. 
	
	Finally, \eqref{equ:conv magnetic transport} follows from the strong
	convergence of $\widetilde{\bff u}_{h,\tau}^{+}$ and
	$\widetilde{\bff B}_{h,\tau}^{-}$ in $L^2_T(\bb L^2)$, the uniform
	$\bb L^\infty$ bound on $\curlh\bff\psi_h$, and
	$\curlh\bff\psi_h\to\curl\bff\psi$ in $\bb L^3$, noting that
	$\bff u\times\bff B\in L^1_T(\bb L^{3/2})$.
\end{proof}

\begin{lemma}[Identification of the limiting martingales]
	\label{lem:conv martingale identification}
	Let $\bff\phi$, $\bff\phi_h$, $\bff\psi$, and $\bff\psi_h$ be as in
	Lemma~\ref{lem:conv deterministic terms}, and define
	\begin{align}
		\label{equ:lim mart u theorem}
		M_u^{\bff\phi}(t)
		&:=
		\inpro{\bff u(t)}{\bff\phi}
		-
		\inpro{\bff u_0}{\bff\phi}
		+
		\nu\int_0^t
		\inpro{\nabla\bff u}{\nabla\bff\phi}\ds
		+
		\int_0^t
		\inpro{(\bff u\cdot\nabla)\bff u}{\bff\phi}\ds
		\nonumber\\
		&\quad
		-
		\int_0^t
		\inpro{\bff J\times\bff B}{\bff\phi}\ds
		-
		\int_0^t
		\inpro{\bff f}{\bff\phi}\ds,
	\end{align}
	and
	\begin{align}
		\label{equ:lim mart B theorem}
		M_B^{\bff\psi}(t)
		&:=
		\inpro{\bff B(t)}{\bff\psi}
		-
		\inpro{\bff B_0}{\bff\psi}
		+
		\sigma\int_0^t
		\inpro{\bff J}{\curl\bff\psi}\ds
		\nonumber\\
		&\quad
		+
		\eta\int_0^t
		\inpro{\bff J\times\bff B}{\curl\bff\psi}\ds
		-
		\int_0^t
		\inpro{\bff u\times\bff B}{\curl\bff\psi}\ds .
	\end{align}
	Let $\{\widetilde{\mathscr F}_t\}_{t\in[0,T]}$ be the usual augmentation
	of the natural filtration generated by $(\bff u,\bff B)$. Then
	$M_u^{\bff\phi}$ and $M_B^{\bff\psi}$ are continuous square-integrable
	$\{\widetilde{\mathscr F}_t\}$-martingales.
	
	Moreover, for all admissible test functions
	$\bff\phi_1,\bff\phi_2,\bff\psi_1,\bff\psi_2$,
	\begin{align}
		\label{equ:lim qv u}
		\left\langle
		M_u^{\bff\phi_1},
		M_u^{\bff\phi_2}
		\right\rangle_t
		&=
		\int_0^t
		\sum_{i=1}^{\infty}
		\inpro{\bff g_i(\bff u,\bff B)}{\bff\phi_1}
		\inpro{\bff g_i(\bff u,\bff B)}{\bff\phi_2}
		\ds,
		\\
		\label{equ:lim qv B}
		\left\langle
		M_B^{\bff\psi_1},
		M_B^{\bff\psi_2}
		\right\rangle_t
		&=
		\int_0^t
		\sum_{i=1}^{\infty}
		\inpro{
			\sigma\bff\chi_i+\eta\bff\chi_i\times\bff B}
		{\curl\bff\psi_1}
		\inpro{
			\sigma\bff\chi_i+\eta\bff\chi_i\times\bff B}
		{\curl\bff\psi_2}
		\ds,
	\end{align}
	and
	\begin{align}
		\label{equ:lim cross qv zero}
		\left\langle
		M_u^{\bff\phi},
		M_B^{\bff\psi}
		\right\rangle_t
		=0.
	\end{align}
\end{lemma}

\begin{proof}
	For $t\in[0,T]$, let $t_\tau:=\max\{t_n:t_n\le t\}$.
	On the original probability space, define
	\begin{align}
		\label{equ:discrete residual u}
		M_{u,h}^{\bff\phi}(t)
		&:=
		\inpro{\bff u_{h,\tau}(t_\tau)}{\bff\phi_h}
		-
		\inpro{\bff u_h^0}{\bff\phi_h}
		+
		\nu\int_0^{t_\tau}
		\inpro{\nabla\bff u_{h,\tau}^{+}}
		{\nabla\bff\phi_h}\ds
		\nonumber\\
		&\quad
		+
		\frac12\int_0^{t_\tau}
		\left[
		\inpro{
			(\bff u_{h,\tau}^{-}\cdot\nabla)
			\bff u_{h,\tau}^{+}}
		{\bff\phi_h}
		-
		\inpro{
			(\bff u_{h,\tau}^{-}\cdot\nabla)\bff\phi_h}
		{\bff u_{h,\tau}^{+}}
		\right]\ds
		\nonumber\\
		&\quad
		-
		\int_0^{t_\tau}
		\inpro{
			\bff J_{h,\tau}^{+}\times\bff B_{h,\tau}^{-}}
		{\bff\phi_h}\ds
		-
		\int_0^{t_\tau}
		\inpro{\bff f_\tau^{+}}{\bff\phi_h}\ds .
	\end{align}
	Since $\bff\phi_h$ is discretely divergence-free,
	\eqref{equ:stoch fem u} and \eqref{equ:stoch fem div u} imply
	\begin{align*}
		M_{u,h}^{\bff\phi}(t)
		=
		\sum_{t_n\le t}
		\sum_{i=1}^{\infty}
		\inpro{
			\bff g_i(\bff u_h^{n-1},\bff B_h^{n-1})}
		{\bff\phi_h}
		\Delta_n\beta_i^u .
	\end{align*}
	
	Similarly, define
	\begin{align}
		\label{equ:discrete residual B}
		M_{B,h}^{\bff\psi}(t)
		&:=
		\inpro{\bff B_{h,\tau}(t_\tau)}{\bff\psi_h}
		-
		\inpro{\bff B_h^0}{\bff\psi_h}
		+
		\sigma\int_0^{t_\tau}
		\inpro{\bff J_{h,\tau}^{+}}{\curlh\bff\psi_h}\ds
		\nonumber\\
		&\quad
		+
		\eta\int_0^{t_\tau}
		\inpro{
			\bff J_{h,\tau}^{+}\times\bff B_{h,\tau}^{-}}
		{\curlh\bff\psi_h}\ds
		-
		\int_0^{t_\tau}
		\inpro{
			\bff u_{h,\tau}^{+}\times\bff B_{h,\tau}^{-}}
		{\curlh\bff\psi_h}\ds .
	\end{align}
	Using \eqref{equ:stoch fem B} and
	\eqref{equ:stoch fem Ohm}, with
	$\bff\chi_h=\curlh\bff\psi_h\in\bb X_h^0$, gives
	\begin{align*}
		M_{B,h}^{\bff\psi}(t)
		=
		-\sum_{t_n\le t}
		\sum_{i=1}^{\infty}
		\inpro{
			\sigma\bff\chi_i
			+
			\eta\bff\chi_i\times\mathcal R_h\bff B_h^{n-1}}
		{\curlh\bff\psi_h}
		\Delta_n\beta_i^B .
	\end{align*}
	Indeed,
	\[
	\inpro{\curl\Pi_h^{\bb X}\bff Z}{\bff\psi_h}
	=
	\inpro{\Pi_h^{\bb X}\bff Z}{\curlh\bff\psi_h}
	=
	\inpro{\bff Z}{\curlh\bff\psi_h},
	\]
	since $\curlh\bff\psi_h\in\bb X_h^0$ and
	$\Pi_h^{\bb X}$ is the $\bb L^2$-orthogonal projection onto
	$\bb X_h^0$.
	Thus $M_{u,h}^{\bff\phi}$ and $M_{B,h}^{\bff\psi}$ are piecewise constant
	square-integrable martingales with respect to the discrete numerical
	filtration.
	
	After the Skorokhod representation, define
	$\widetilde M_{u,h}^{\bff\phi}$ and
	$\widetilde M_{B,h}^{\bff\psi}$ by the defining expressions in
	\eqref{equ:discrete residual u} and
	\eqref{equ:discrete residual B}, respectively, with all numerical
	variables replaced by their tilded copies. The residuals and the bounded
	cylindrical functionals of the discrete history used below are measurable
	functions of the joint numerical variables. Hence, equality of the joint
	laws transfers the corresponding martingale-test and product-martingale
	expectation identities to the new probability space.
	
	\emph{Step 1: Convergence of the covariance coefficients.}
	By the Lipschitz assumption on the velocity noise,
	\eqref{equ:conv sk u strong}, \eqref{equ:conv sk B strong}, and the
	convergence of $\bff\phi_h$, we have
	\begin{align}
		\label{equ:conv velocity martingale coefficients}
		\widetilde{\bb E} \left[
		\int_0^T
		\sum_{i=1}^{\infty}
		\abs{
			\inpro{
				\bff g_i(
				\widetilde{\bff u}_{h,\tau}^{-},
				\widetilde{\bff B}_{h,\tau}^{-})}
			{\bff\phi_h}
			-
			\inpro{\bff g_i(\bff u,\bff B)}{\bff\phi}
		}^2
		\ds \right]
		\longrightarrow0.
	\end{align}
	Indeed, the almost-sure strong convergences from
	Lemma~\ref{lem:conv sk representation}, together with the higher-moment bounds from Proposition~\ref{prop:stoch stability revised}, give the corresponding convergence in mean by uniform integrability.
	
	Likewise, the reconstruction strong convergence	$\mathcal R_h\widetilde{\bff B}_{h,\tau}^{-}
	\to\bff B$ in $L^2_T(\bb L^2)$,
	the convergence of $\curlh\bff\psi_h$, and the summability assumptions on
	$\{\bff\chi_i\}_{i=1}^{\infty}$ imply
	\begin{align}
		\label{equ:conv magnetic martingale coefficients}
		&\widetilde{\bb E} \left[
		\int_0^T
		\sum_{i=1}^{\infty}
		\abs{
			\inpro{
				\sigma\bff\chi_i
				+
				\eta\bff\chi_i\times
				\mathcal R_h\widetilde{\bff B}_{h,\tau}^{-}}
			{\curlh\bff\psi_h}
			-
			\inpro{
				\sigma\bff\chi_i+\eta\bff\chi_i\times\bff B}
			{\curl\bff\psi}
		}^2
		\ds \right]
		\longrightarrow0.
	\end{align}
	
	\emph{Step 2: Martingale property of the limiting residuals.}
	Fix $0\le s<t\le T$, an integer $\ell\ge1$, and times
	$0\le s_1<\cdots<s_\ell\le s$.
	For $j=1,\ldots,\ell$, set $s_{j,\tau}:=\max\{t_n:t_n\le s_j\}$,
	and let $w_j\in C_{\mathrm b} \left((\bb V^m_\sigma)'\times\widetilde{\bb H}^{-m};\mathbb R\right)$.
	Since $s_{j,\tau}\le s_\tau$ and the affine interpolants agree with the
	nodal values at grid points, the random variable
	$\prod_{j=1}^{\ell}
	w_j\left(
	\bff u_{h,\tau}(s_{j,\tau}),
	\bff B_{h,\tau}(s_{j,\tau})
	\right)$
	is $\mathscr F_{s_\tau}$-measurable. Hence,
	\begin{align}
		\label{equ:discrete martingale test u}
		&\bb E\left[
		\left(
		M_{u,h}^{\bff\phi}(t)
		-
		M_{u,h}^{\bff\phi}(s)
		\right)
		\prod_{j=1}^{\ell}
		w_j\left(
		\bff u_{h,\tau}(s_{j,\tau}),
		\bff B_{h,\tau}(s_{j,\tau})
		\right)
		\right]
		=0,
		\\
		\label{equ:discrete martingale test B}
		&\bb E\left[
		\left(
		M_{B,h}^{\bff\psi}(t)
		-
		M_{B,h}^{\bff\psi}(s)
		\right)
		\prod_{j=1}^{\ell}
		w_j\left(
		\bff u_{h,\tau}(s_{j,\tau}),
		\bff B_{h,\tau}(s_{j,\tau})
		\right)
		\right]
		=0.
	\end{align}
	By equality of the joint laws, the analogous identities hold for the
	tilded variables.
	
	From \eqref{equ:conv sk path}--\eqref{equ:conv Bk path}, the continuity of the limiting paths, and
	$s_{j,\tau}\to s_j$, we obtain, $\widetilde{\bb P}$-a.s.
	\[
	\prod_{j=1}^{\ell}
	w_j\left(
	\widetilde{\bff u}_{h,\tau}(s_{j,\tau}),
	\widetilde{\bff B}_{h,\tau}(s_{j,\tau})
	\right)
	\longrightarrow
	\prod_{j=1}^{\ell}
	w_j\left(
	\bff u(s_j),\bff B(s_j)
	\right).
	\]
	Moreover, for every fixed $r\in[0,T]$,
	\begin{align}
		\label{equ:conv residuals}
		\widetilde M_{u,h}^{\bff\phi}(r)
		&\longrightarrow
		M_u^{\bff\phi}(r),
		\qquad
		\widetilde M_{B,h}^{\bff\psi}(r)
		\longrightarrow
		M_B^{\bff\psi}(r)
	\end{align}
	in probability. Indeed, the path convergences
	\eqref{equ:conv sk path}--\eqref{equ:conv Bk path}, the continuity of the
	limiting paths, $r_\tau\to r$, the convergence of the discrete test
	functions, and the uniform energy bounds imply
	\[
	\inpro{
		\widetilde{\bff u}_{h,\tau}(r_\tau)}
	{\bff\phi_h}
	\longrightarrow
	\inpro{\bff u(r)}{\bff\phi},
	\qquad
	\inpro{
		\widetilde{\bff B}_{h,\tau}(r_\tau)}
	{\bff\psi_h}
	\longrightarrow
	\inpro{\bff B(r)}{\bff\psi}
	\]
	in probability. The assumptions on the approximations of the initial data and the forcing term give
	\[
	\inpro{\widetilde{\bff u}_h^0}{\bff\phi_h}
	\longrightarrow
	\inpro{\bff u_0}{\bff\phi},
	\qquad
	\inpro{\widetilde{\bff B}_h^0}{\bff\psi_h}
	\longrightarrow
	\inpro{\bff B_0}{\bff\psi},
	\qquad
	\int_0^{r_\tau}
	\inpro{\bff f_\tau^+}{\bff\phi_h}\ds
	\longrightarrow
	\int_0^r
	\inpro{\bff f}{\bff\phi}\ds.
	\]
	For the remaining deterministic terms, apply
	Lemma~\ref{lem:conv deterministic terms} with the fixed upper limit $r$.
	The difference between the integrals over $[0,r]$ and $[0,r_\tau]$ is
	supported on the interval between $r_\tau$ and $r$. The same H\"older
	estimates used in the proof of
	Lemma~\ref{lem:conv deterministic terms}, together with the uniform
	stability bounds and $\abs{r-r_\tau}\le\tau$, show that all these
	remainder terms converge to zero in probability.
	
	By the discrete BDG inequality (Lemma~\ref{lem:discrete-BDG}), the noise-growth assumptions,
	Proposition~\ref{prop:stoch stability revised}, and equality of the joint
	laws, we have
	\begin{align}
		\label{equ:uniform residual moments}
		\sup_{h,\tau}
		\widetilde{\bb E}\left[
		\max_{0\le r\le T}
		\left(
		\abs{\widetilde M_{u,h}^{\bff\phi}(r)}^4
		+
		\abs{\widetilde M_{B,h}^{\bff\psi}(r)}^4
		\right)
		\right]
		<\infty .
	\end{align}
	Therefore, the products in
	\eqref{equ:discrete martingale test u} and
	\eqref{equ:discrete martingale test B} are uniformly integrable.
	Passing to the limit yields
	\begin{align*}
	\widetilde{\bb E}\left[
	\left(
	M_u^{\bff\phi}(t)-M_u^{\bff\phi}(s)
	\right)
	\prod_{j=1}^{\ell}
	w_j\left(\bff u(s_j),\bff B(s_j)\right)
	\right]
	&=0,
	\\
	\widetilde{\bb E}\left[
	\left(
	M_B^{\bff\psi}(t)-M_B^{\bff\psi}(s)
	\right)
	\prod_{j=1}^{\ell}
	w_j\left(\bff u(s_j),\bff B(s_j)\right)
	\right]
	&=0.
	\end{align*}
	Since the state spaces are separable and $(\bff u,\bff B)$ has
	continuous paths, the preceding identities for bounded continuous cylindrical
	functions imply, by a monotone-class argument, that
	$M_u^{\bff\phi}$ and $M_B^{\bff\psi}$ are martingales with respect to the
	usual augmentation of the natural filtration of $(\bff u,\bff B)$.
	By \eqref{equ:conv current identified}, this is also the natural filtration of
	$(\bff u,\bff B,\bff J)$. Their square integrability follows from
	\eqref{equ:uniform residual moments} and Fatou's lemma, while continuity
	follows from \eqref{equ:lim mart u theorem} and
	\eqref{equ:lim mart B theorem}.

	\emph{Step 3: Identification of the covariations.}
	Let $\bff\phi_1,\bff\phi_2$ be two velocity test functions, $\bff\psi_1,\bff\psi_2$ be two magnetic test functions, and define $\bff\phi_{j,h}:=\mathcal S_h\bff\phi_j$ and $\bff\psi_{j,h}:=\Pi_h^{\bb{RT}}\bff\psi_j$, for $j=1,2$.
	For the velocity test functions $\bff\phi_1,\bff\phi_2$, the predictable
	covariation of the discrete residuals is
	\begin{align}
		\label{equ:discrete qv u pair}
		\inpro{M_{u,h}^{\bff\phi_1}}{M_{u,h}^{\bff\phi_2}}_t
		&=
		\int_0^{t_\tau}
		\sum_{i=1}^{\infty}
		\inpro{
			\bff g_i(\bff u_{h,\tau}^{-},\bff B_{h,\tau}^{-})}
		{\bff\phi_{1,h}}
		\inpro{
			\bff g_i(\bff u_{h,\tau}^{-},\bff B_{h,\tau}^{-})}
		{\bff\phi_{2,h}}
		\ds .
	\end{align}
	Correspondingly, for the magnetic test functions,
	\begin{align}
		\label{equ:discrete qv B pair}
		\inpro{M_{B,h}^{\bff\psi_1}}{M_{B,h}^{\bff\psi_2}}_t
		&=
		\int_0^{t_\tau}
		\sum_{i=1}^{\infty}
		\inpro{
			\sigma\bff\chi_i
			+
			\eta\bff\chi_i\times
			\mathcal R_h\bff B_{h,\tau}^{-}}
		{\curlh\bff\psi_{1,h}}
		\inpro{
			\sigma\bff\chi_i
			+
			\eta\bff\chi_i\times
			\mathcal R_h\bff B_{h,\tau}^{-}}
		{\curlh\bff\psi_{2,h}}
		\ds .
	\end{align}
	By equality of the joint laws, the corresponding martingale-test identities
	hold for the tilded variables. Hence the tilded residuals are discrete
	martingales with respect to the natural discrete filtration generated by
	the tilded numerical variables, and their predictable covariations are
	given by the tilded versions of \eqref{equ:discrete qv u pair} and
	\eqref{equ:discrete qv B pair}.
	
	By applying
	\eqref{equ:conv velocity martingale coefficients} and
	\eqref{equ:conv magnetic martingale coefficients} to each of the two
	test functions, and using the Cauchy--Schwarz inequality in $\ell^2$,
	we obtain
	\begin{align*}
		&\widetilde{\bb E}\left[
		\sup_{t\in[0,T]}
		\abs{
			\left\langle
			\widetilde M_{u,h}^{\bff\phi_1},
			\widetilde M_{u,h}^{\bff\phi_2}
			\right\rangle_t
			-
			\int_0^t
			\sum_{i=1}^{\infty}
			\inpro{\bff g_i(\bff u,\bff B)}{\bff\phi_1}
			\inpro{\bff g_i(\bff u,\bff B)}{\bff\phi_2}
			\ds
		}
		\right]
		\longrightarrow0,
		\\
		&\widetilde{\bb E}\left[
		\sup_{t\in[0,T]}
		\abs{
			\left\langle
			\widetilde M_{B,h}^{\bff\psi_1},
			\widetilde M_{B,h}^{\bff\psi_2}
			\right\rangle_t
			-
			\int_0^t
			\sum_{i=1}^{\infty}
			\inpro{
				\sigma\bff\chi_i+\eta\bff\chi_i\times\bff B}
			{\curl\bff\psi_1}
			\inpro{
				\sigma\bff\chi_i+\eta\bff\chi_i\times\bff B}
			{\curl\bff\psi_2}
			\ds
		}
		\right]
		\longrightarrow0.
	\end{align*}
	The convergence of the covariance densities follows from
	\eqref{equ:conv velocity martingale coefficients} and
	\eqref{equ:conv magnetic martingale coefficients}, together with the
	Cauchy--Schwarz inequality in $\ell^2$. The contributions over $[t_\tau,t]$ converge uniformly in $t$ to zero a.s. by the absolute continuity of the limiting covariance integrals, and the convergence also holds in $L^1(\widetilde\Omega)$ by dominated convergence, using the noise-growth assumptions and the moment bounds.
	
	For each pair of velocity test functions, the process
	\[
	\widetilde M_{u,h}^{\bff\phi_1}(t)
	\widetilde M_{u,h}^{\bff\phi_2}(t)
	-
	\left\langle
	\widetilde M_{u,h}^{\bff\phi_1},
	\widetilde M_{u,h}^{\bff\phi_2}
	\right\rangle_t
	\]
	is a discrete martingale, and the analogous assertion holds for two
	magnetic residuals. Testing these product-martingale identities by the same
	products of bounded continuous functions used in Step~2 and passing to
	the limit shows that
	\[
	M_u^{\bff\phi_1}(t)M_u^{\bff\phi_2}(t)
	-
	\int_0^t
	\sum_{i=1}^{\infty}
	\inpro{\bff g_i(\bff u,\bff B)}{\bff\phi_1}
	\inpro{\bff g_i(\bff u,\bff B)}{\bff\phi_2}
	\ds
	\]
	and
	\[
	M_B^{\bff\psi_1}(t)M_B^{\bff\psi_2}(t)
	-
	\int_0^t
	\sum_{i=1}^{\infty}
	\inpro{
		\sigma\bff\chi_i+\eta\bff\chi_i\times\bff B}
	{\curl\bff\psi_1}
	\inpro{
		\sigma\bff\chi_i+\eta\bff\chi_i\times\bff B}
	{\curl\bff\psi_2}
	\ds
	\]
	are martingales. The necessary uniform integrability follows from
	\eqref{equ:uniform residual moments}. By uniqueness of predictable
	covariations of continuous square-integrable martingales,
	\eqref{equ:lim qv u} and \eqref{equ:lim qv B} follow.
	
	Finally, independence of the two Brownian families and predictability of the corresponding coefficients imply that
	$M_{u,h}^{\bff\phi} M_{B,h}^{\bff\psi}$ is a discrete martingale on the
	original probability space. Hence, with the notation of Step~2,
	\[
	\bb E\left[
	\left(
	M_{u,h}^{\bff\phi}(t)M_{B,h}^{\bff\psi}(t)
	-
	M_{u,h}^{\bff\phi}(s)M_{B,h}^{\bff\psi}(s)
	\right)
	\prod_{j=1}^{\ell}
	w_j\left(
	\bff u_{h,\tau}(s_{j,\tau}),
	\bff B_{h,\tau}(s_{j,\tau})
	\right)
	\right]
	=0.
	\]
	By equality of joint laws, the analogous identity holds for the
	tilded variables. Passing to the limit as in Step~2 shows that
	$M_u^{\bff\phi}M_B^{\bff\psi}$ is a martingale. Therefore,
	by uniqueness of the predictable covariation of continuous
	square-integrable martingales, \eqref{equ:lim cross qv zero} follows.
	This completes the proof.
\end{proof}

We are now in a position to state our main theorem.

\begin{theorem}[Convergence to a martingale solution]
	\label{the:conv martingale solution}
	Assume the hypotheses of Proposition~\ref{prop:stoch stability revised},
	with $\delta_\chi$ in \eqref{equ:hall noise small} chosen so that
	\eqref{equ:stoch stability revised main} holds. Assume further that the
	discrete initial data satisfy
	\[
	\bff u_h^0\to\bff u_0
	\quad\text{strongly in }\bb L^2(\mathscr D),
	\qquad
	\bff B_h^0\to\bff B_0
	\quad\text{strongly in }\bb L^2(\mathscr D),
	\]
	with
	\[
	\divg\bff B_h^0=0,
	\qquad
	\sup_h\norm{\bff J_h^0}{\bb L^2}<\infty .
	\]
	Then, for every sequence $(h_k,\tau_k)\to(0,0)$, there exist a
	subsequence, not relabelled, a stochastic basis $\bigl(\widetilde\Omega,\widetilde{\mathscr F}, \{\widetilde{\mathscr F}_t\}_{t\in[0,T]}, \widetilde{\bb P}\bigr)$,
	two independent families of real Brownian motions
	$\{\widetilde\beta_i^u\}_{i=1}^{\infty}$ and
	$\{\widetilde\beta_i^B\}_{i=1}^{\infty}$, and progressively measurable
	processes $(\bff u,\bff B,\bff J)$ such that
	$(\bff u,\bff B,\bff J)$ is a weak martingale solution of
	\eqref{equ:stoch mhd} in the sense of Definition~\ref{def:weak sol}.
	
	Moreover, the discrete variables along this subsequence admit copies on
	the new probability space, with the same joint laws as the original
	ones, which converge $\widetilde{\bb P}$-a.s. to the corresponding
	limiting variables in the product topology of
	Lemma~\ref{lem:conv sk representation}.
\end{theorem}

\begin{proof}
	Let $(h_k,\tau_k)\to(0,0)$ be arbitrary. By
	Lemma~\ref{lem:conv sk representation}, after passing to a subsequence,
	there exist copies of the discrete variables on a new probability space
	having the same joint laws as the original ones and converging
	$\widetilde{\bb P}$-a.s. in the product topology specified there.
	
	By equality of laws, Fatou's lemma, and
	Proposition~\ref{prop:stoch stability revised}, the limiting processes
	satisfy
	\[
	\bff u\in L^2\left(\widetilde\Omega; L^\infty_T(\bb L^2)\cap L^2_T(\bb H_0^1) \right),
	\qquad
	\bff B\in L^2\left( \widetilde\Omega; L^\infty_T(\bb L^2) \right),
	\qquad
	\bff J\in L^2\left( \widetilde\Omega;L^2_T(\bb L^2) \right).
	\]
	The convergence in the negative path spaces and the convergence of the
	discrete initial data imply $\bff u(0)=\bff u_0$ and $\bff B(0)=\bff B_0$.
	Moreover, the continuity in the negative spaces together with the
	$L^\infty_T(\bb L^2)$ bounds yields
	\[
	\bff u,\bff B\in C([0,T];\bb L^2_{\mathrm w}),
	\qquad
	\widetilde{\bb P}\text{-a.s.}
	\]
	Lemma~\ref{lem:conv constraints current} identifies the divergence and
	boundary constraints and gives
	\[
	\bff J=\curl\bff B,
	\qquad
	\bff B\in
	L^2\left(
	\widetilde\Omega;
	L^2_T(\bb H^1)
	\right),
	\]
	in the corresponding weak sense. Lemma~\ref{lem:conv deterministic terms}
	identifies all deterministic terms in the limiting velocity and magnetic
	equations.
	
	It remains to identify the stochastic terms. By
	Lemma~\ref{lem:conv martingale identification}, the residuals
	$M_u^{\bff\phi}$ and $M_B^{\bff\psi}$ form a joint continuous
	square-integrable cylindrical martingale whose predictable covariations
	are given by \eqref{equ:lim qv u}--\eqref{equ:lim cross qv zero}. The covariance factorises through the coefficient maps $\bff\phi \mapsto \left\{\inpro{\bff g_i(\bff u,\bff B)}{\bff\phi}\right\}_{i\ge1}$,
	and $\bff\psi \mapsto \left\{-\inpro{\sigma\bff\chi_i+\eta\bff\chi_i\times\bff B}{\curl\bff\psi} \right\}_{i\ge1}$.
	Hence, by the Brownian representation theorem for cylindrical
	martingales~\cite{Ond05}, possibly after enlarging
	the stochastic basis, there exist two independent families of standard
	real Brownian motions $\{\widetilde\beta_i^u\}_{i\ge1}$ and $\{\widetilde\beta_i^B\}_{i\ge1}$ such that
	\begin{align}
		M_u^{\bff\phi}(t)
		&=
		\sum_{i=1}^{\infty}
		\int_0^t \inpro{\bff g_i(\bff u,\bff B)}{\bff\phi}
		\mathrm d\widetilde\beta_i^u(s),
		\label{equ:represented mart u}
		\\
		M_B^{\bff\psi}(t)
		&=
		-\sum_{i=1}^{\infty}
		\int_0^t \inpro{\sigma\bff\chi_i+\eta\bff\chi_i\times\bff B}
		{\curl\bff\psi} \mathrm d\widetilde\beta_i^B(s),
		\label{equ:represented mart B}
	\end{align}
	for every smooth admissible test function.
	
	Substituting the definitions of the residuals yields the two variational
	identities in Definition~\ref{def:weak sol} for $\bff\phi\in
	\mathcal D_\sigma$ and $\bff\psi \in\mathcal D_B$. By density and continuity of the deterministic and stochastic terms,
	the velocity identity extends to every $\bff\phi\in\bb V_\sigma$.
	Likewise, by the definition of $\bb{V}_B$ in \eqref{equ:V B}, the magnetic identity extends by continuity to every $\bff\psi\in\bb V_B$. Indeed,
	\[
	\bff J\times\bff B,\,
	\bff u\times\bff B
	\in L^1_T(\bb L^{3/2}),
	\]
	and the remaining deterministic terms are continuous with respect to
	strong convergence in $\bb W^{1,3}$. Continuity of the stochastic terms
	follows from the noise-growth assumptions, the summability assumptions on
	$\{\bff\chi_i\}_{i\ge1}$, and the It\^o isometry.
	
	The limiting identities hold for every $t\in[0,T]$ by continuity.
	With respect to the enlarged filtration, $\bff u$ and $\bff B$ remain progressively measurable, and the relation $\bff J=\curl\bff B$ allows
	us to choose a progressively measurable version of $\bff J$.
	Thus all requirements of Definition~\ref{def:weak sol} are satisfied, and $(\bff u,\bff B,\bff J)$ is a weak martingale solution of~\eqref{equ:stoch mhd}. This completes the proof of the theorem.
\end{proof}

\section{Numerical simulations}
\label{sec:numerical experiments}

We illustrate the qualitative behaviour and convergence of the proposed scheme on the unit cube $\mathscr D=(0,1)^3$ by running a numerical experiment showing magnetic islands reconnection.
The dimensionless parameters are chosen as
\[
	\nu=0.004,
	\qquad
	\sigma=0.008,
	\qquad
	\eta=0.15,
\]
and no deterministic body force is applied. We set $\bff u_0 = \bff0$ and
define the initial magnetic field from the vector potential
\[
\bff A_0(x,y,z)
=
\left(0,\, 0,\,
\frac14 \sin(\pi x)\sin(\pi y)\sin(\pi z)
\log \bigl(\cosh \bigl(4 (y- 0.5) \bigr)\bigr) \right)^\top,
\]
then set $B_0=\curl\bff A_0$.
The stochastic momentum coefficient and the current-fluctuation mode are chosen as
\[
\bff g_1(\bff u,\bff B)=0.08\bff u,
\qquad
\bff\chi_1(x,y,z)
=
0.025
\left(
0,\,
0,\,
\sin(\pi x)\sin(\pi y)\sin(\pi z)
\right)^\top,
\]
with all remaining modes set to zero. The momentum and magnetic equations are driven by independent Brownian motions.

\subsection{Structure preservation}

Figure~\ref{fig:sample path magnetic} shows selected snapshots from one realisation. The displayed quantity is the magnitude of the magnetic field on the slice (central plane) $z=\frac12$.

\begin{figure}[!htb]
	\centering
	\begin{subfigure}[t]{0.3\textwidth}
		\centering
		\includegraphics[width=\textwidth]{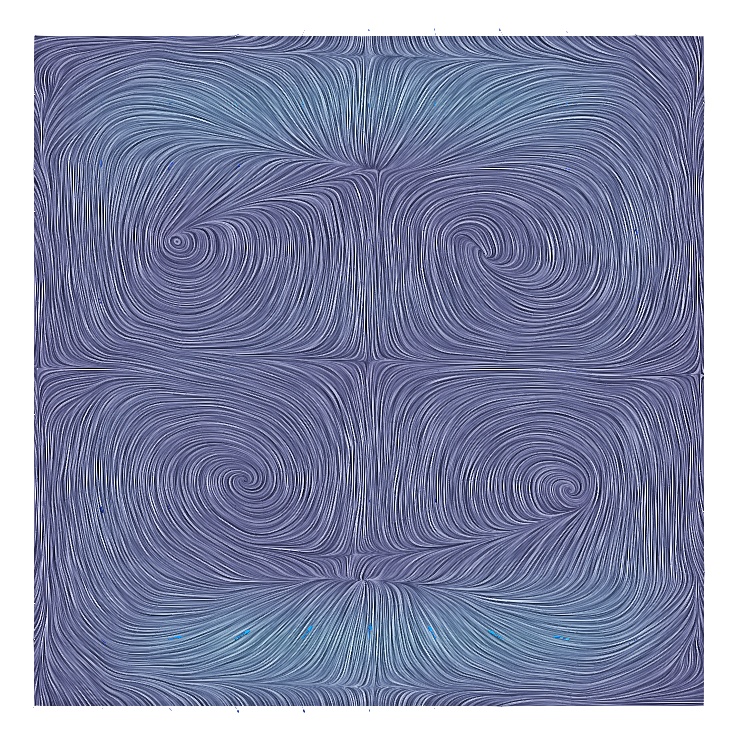}
		\caption{$t=0.025$}
	\end{subfigure}
	\begin{subfigure}[t]{0.3\textwidth}
		\centering
		\includegraphics[width=\textwidth]{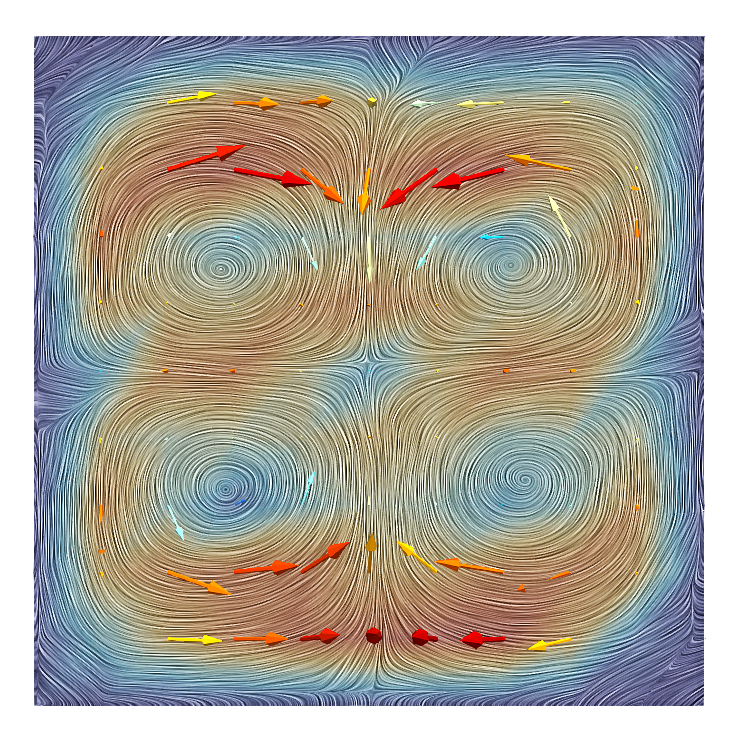}
		\caption{$t=0.50$}
	\end{subfigure}
	\begin{subfigure}[t]{0.3\textwidth}
		\centering
		\includegraphics[width=\textwidth]{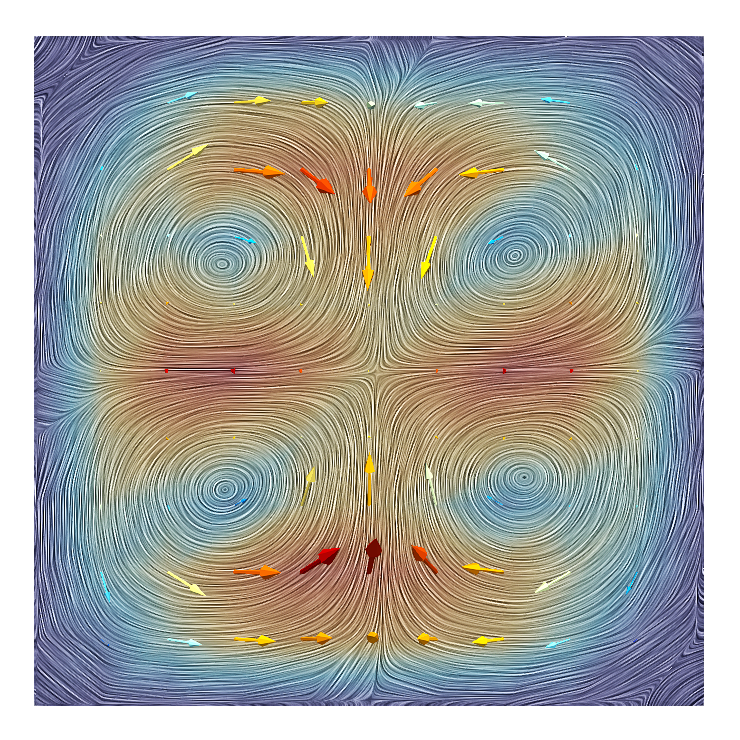}
		\caption{$t=1.50$}
	\end{subfigure}
	\begin{subfigure}[t]{0.07\textwidth}
		\centering
		\includegraphics[width=\textwidth]{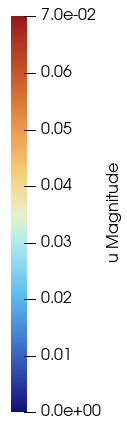}
	\end{subfigure}
	\caption{Snapshots of $\bff u_h^n$ on the plane $z=\frac12$ for one sample path. The colour indicates $\abs{\bff{u}_h^n}$. The vectors point in the direction of $\bff{u}_h^n$.}
	\label{fig:sample path velocity}
\end{figure}

\begin{figure}[!htb]
	\centering
	\begin{subfigure}[t]{0.3\textwidth}
		\centering
		\includegraphics[width=\textwidth]{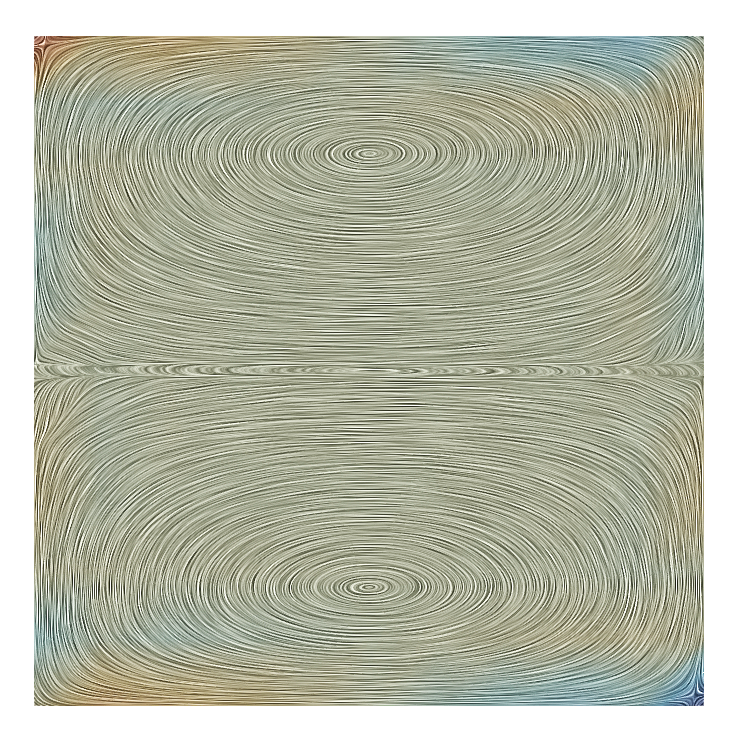}
		\caption{$t=0$}
	\end{subfigure}
	\begin{subfigure}[t]{0.3\textwidth}
		\centering
		\includegraphics[width=\textwidth]{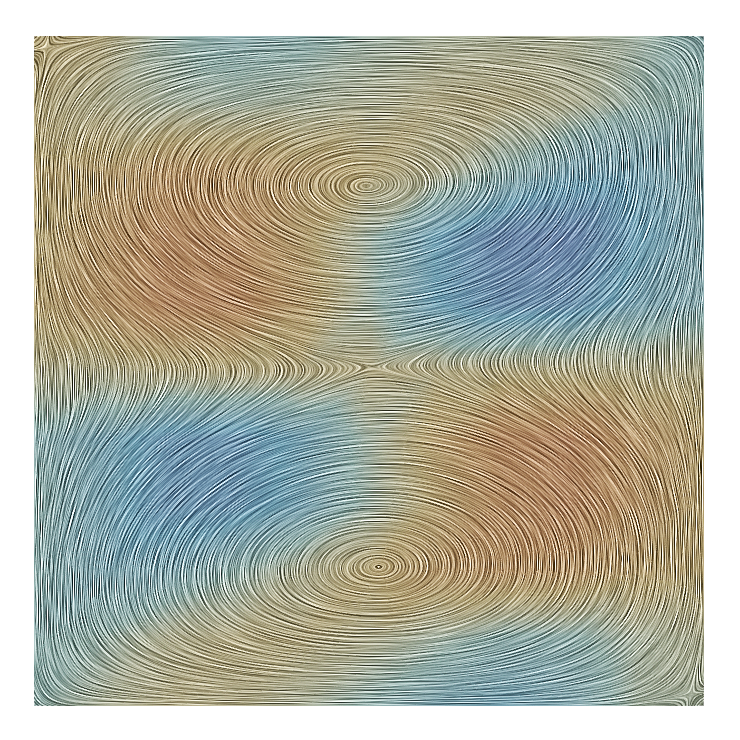}
		\caption{$t=0.50$}
	\end{subfigure}
	\begin{subfigure}[t]{0.3\textwidth}
		\centering
		\includegraphics[width=\textwidth]{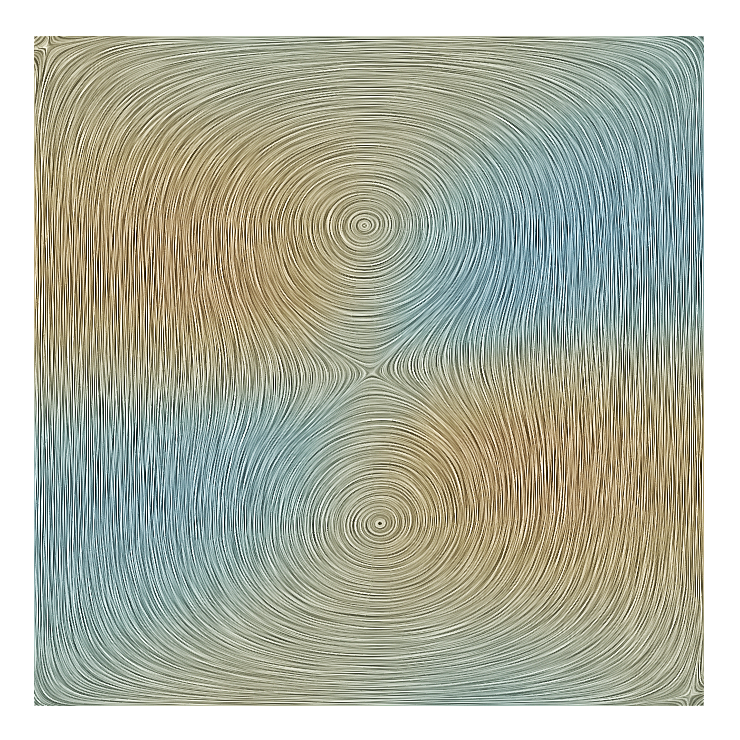}
		\caption{$t=1.50$}
	\end{subfigure}
	\begin{subfigure}[t]{0.07\textwidth}
		\centering
		\includegraphics[width=\textwidth]{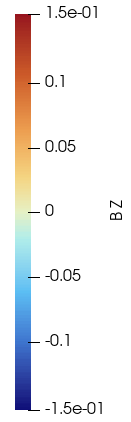}
	\end{subfigure}
	\caption{Snapshots of $\bff B_h^n$ on the plane $z=\frac12$ for one sample path. The colour indicates the $z$-component of $\bff{B}_h^n$, which shows a quadrupolar pattern.}
	\label{fig:sample path magnetic}
\end{figure}

We monitor the discrete total energy
\[
\mathcal E_h^n
:=
\frac12
\left(
\norm{\bff u_h^n}{\bb L^2}^2
+
\norm{\bff B_h^n}{\bb L^2}^2
\right)
\]
and the cellwise magnetic-divergence defect
\[
D_h^n
:=
\max_{K\in\mathcal T_h}
\Big|
\left.\divg\bff B_h^n\right|_K
\Big|.
\]
Figure~\ref{fig:energy divergence ensemble} displays the individual sample paths together with their empirical means. For the chosen noise amplitudes, the dissipative trend remains dominant in the energy evolution. The magnetic-divergence defect remains close to machine precision for every sample path, illustrating the exact preservation of the magnetic Gauss law, up to the solver tolerance.

\begin{figure}[htbp]
	\centering
	\begin{subfigure}[t]{0.49\textwidth}
		\centering
		\includegraphics[width=\textwidth]{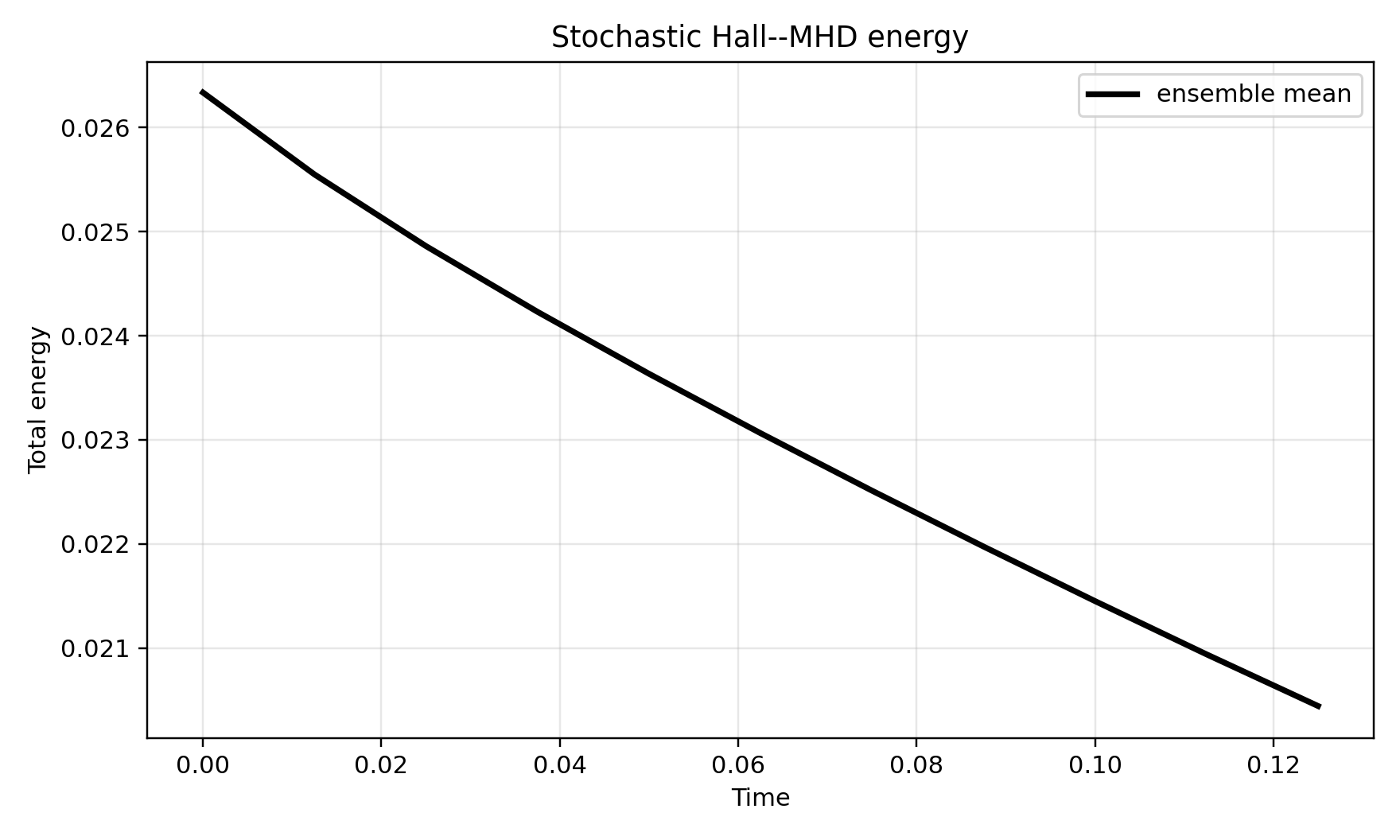}
		\caption{Discrete total energy.}
	\end{subfigure}
	\begin{subfigure}[t]{0.49\textwidth}
		\centering
		\includegraphics[width=\textwidth]{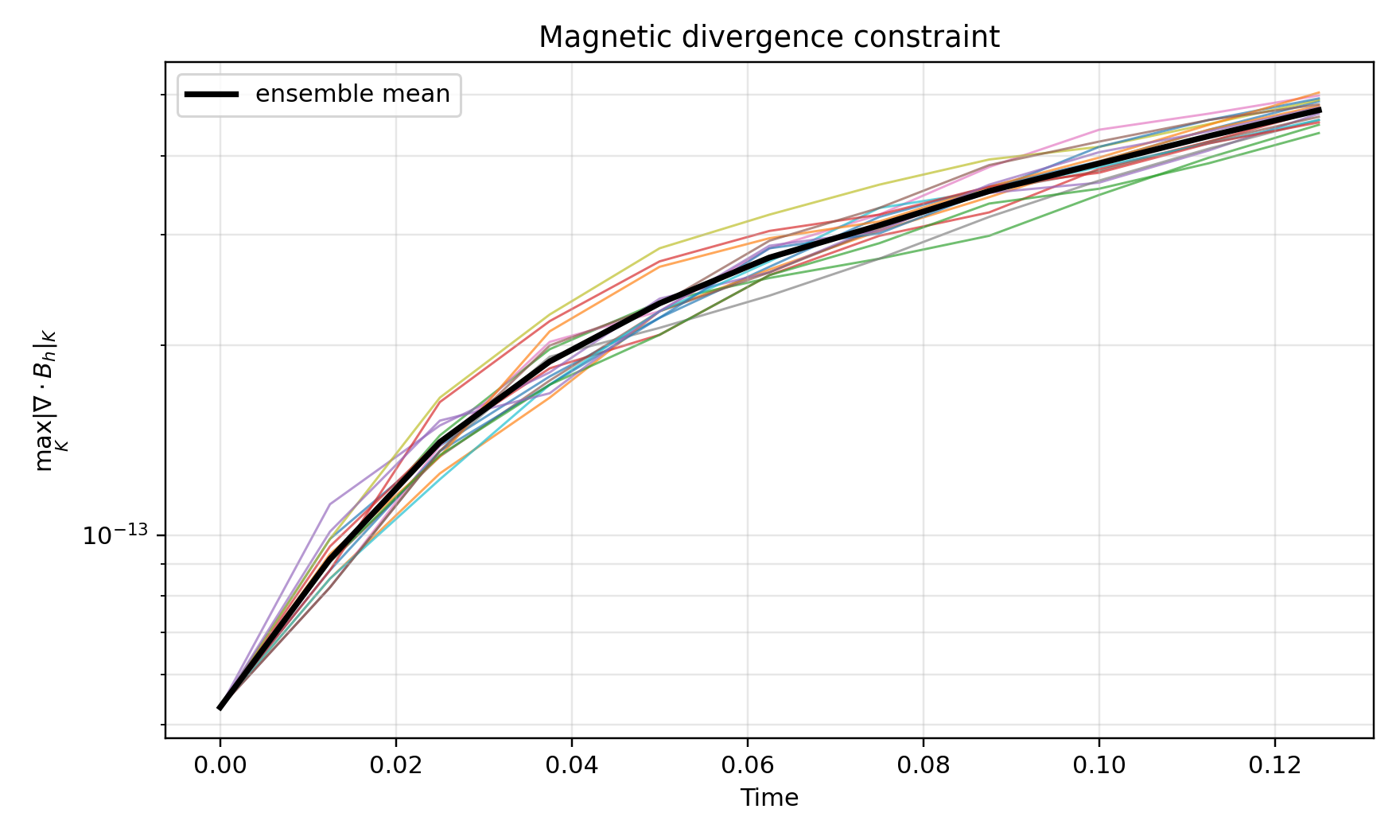}
		\caption{Maximum cellwise magnetic divergence.}
	\end{subfigure}
	\caption{Energy and divergence diagnostics for 20 sample paths. The thicker curve denotes the empirical mean.}
	\label{fig:energy divergence ensemble}
\end{figure}

\subsection{Spatial and temporal convergence}

Since an exact solution is unavailable, we assess convergence by comparison with finer numerical reference solutions. All discretisation levels corresponding to a given sample are driven by the same Brownian paths. For the temporal study, each coarse Brownian increment is obtained by summing the corresponding fine-grid increments.

For $\bff v\in\{\bff u,\bff B,\bff J\}$ and each tested mesh size $h_\ell$,
we compute the root-mean-square final-time error
\[
e_{\bff v,h_\ell}
:=
\left(
\frac1M
\sum_{r=1}^{M}
\norm{
	\bff v_{h_\ell}^{(r)}(T)
	-
	\bff v_{h_{\mathrm{ref}}}^{(r)}(T)
}{\bb L^2}^2
\right)^{1/2},
\]
where $h_{\mathrm{ref}}<h_\ell$ is a fixed reference mesh size. The temporal
errors $e_{\bff v,\tau_\ell}$ are defined analogously using a fixed reference
time step $\tau_{\mathrm{ref}}$ on a fixed spatial mesh. In both studies,
all discretisation levels corresponding to the same sample path are driven
by the same Brownian realisation. We use $M=20$ coupled sample paths and $T=0.125$.

The errors are reported under successive refinement of $h_\ell$ or
$\tau_\ell$. The corresponding observed rates are computed from consecutive
reference errors, e.g.,
\[
\operatorname{rate}_{\bff v,h_\ell}
=
\frac{
	\log(e_{\bff v,h_{\ell-1}}/e_{\bff v,h_\ell})
}{
	\log(h_{\ell-1}/h_\ell)
},
\]
and analogously for the temporal refinement.

For the spatial experiment, we use $N=2,4,6,8$ uniform subdivisions in
each coordinate direction, a finer reference mesh with
$N_{\mathrm{ref}}=16$, and the fixed time step
$\tau=6.25\times10^{-3}$. The resulting root-mean-square errors and
empirical rates are reported in Table~\ref{tab:spatial convergence L2}.

\begin{table}[htbp]
	\centering
	\small
	\caption{Spatial root-mean-square final-time errors and empirical
		convergence rates.}
	\label{tab:spatial convergence L2}
	\begin{tabular}{c c c c c c c c}
		\hline
		$N$
		&
		$h_{\max}$
		&
		$e_{\bff u,h}$
		&
		rate $(\bff u)$
		&
		$e_{\bff B,h}$
		&
		rate $(\bff B)$
		&
		$e_{\bff J,h}$
		&
		rate $(\bff J)$
		\\
		\hline
		$2$
		&
		$8.6603\times10^{-1}$
		&
		$1.4737\times10^{-2}$
		&
		$-$
		&
		$2.0259\times10^{-1}$
		&
		$-$
		&
		$2.1300$
		&
		$-$
		\\
		$4$
		&
		$4.3301\times10^{-1}$
		&
		$1.3728\times10^{-2}$
		&
		$0.10$
		&
		$1.3048\times10^{-1}$
		&
		$0.63$
		&
		$1.7127$
		&
		$0.31$
		\\
		$6$
		&
		$2.8868\times10^{-1}$
		&
		$9.0768\times10^{-3}$
		&
		$1.02$
		&
		$8.9239\times10^{-2}$
		&
		$0.94$
		&
		$1.3859$
		&
		$0.52$
		\\
		$8$
		&
		$2.1651\times10^{-1}$
		&
		$7.9384\times10^{-3}$
		&
		$0.47$
		&
		$6.1541\times10^{-2}$
		&
		$1.29$
		&
		$1.2795$
		&
		$0.28$
		\\
		\hline
	\end{tabular}
\end{table}

For the temporal experiment, the mesh is fixed at $N=8$, and the time
steps
\[
\tau
\in
\left\{
2.5\times10^{-2},
1.25\times10^{-2},
6.25\times10^{-3}
\right\}
\]
are compared with the reference step
$\tau_{\mathrm{ref}}=3.125\times10^{-3}$. The corresponding errors and
empirical rates are shown in Table~\ref{tab:temporal convergence L2}.

\begin{table}[htbp]
	\centering
	\small
	\caption{Temporal root-mean-square final-time errors and empirical
		convergence rates.}
	\label{tab:temporal convergence L2}
	\begin{tabular}{c c c c c c c}
		\hline
		$\tau$
		&
		$e_{\bff u,\tau}$
		&
		rate $(\bff u)$
		&
		$e_{\bff B,\tau}$
		&
		rate $(\bff B)$
		&
		$e_{\bff J,\tau}$
		&
		rate $(\bff J)$
		\\
		\hline
		$2.5\times10^{-2}$
		&
		$1.0483\times10^{-3}$
		&
		$-$
		&
		$7.2402\times10^{-3}$
		&
		$-$
		&
		$3.2303\times10^{-1}$
		&
		$-$
		\\
		$1.25\times10^{-2}$
		&
		$4.8349\times10^{-4}$
		&
		$1.12$
		&
		$3.6803\times10^{-3}$
		&
		$0.98$
		&
		$1.8298\times10^{-1}$
		&
		$0.82$
		\\
		$6.25\times10^{-3}$
		&
		$1.6810\times10^{-4}$
		&
		$1.52$
		&
		$1.3758\times10^{-3}$
		&
		$1.42$
		&
		$7.2850\times10^{-2}$
		&
		$1.33$
		\\
		\hline
	\end{tabular}
\end{table}

The spatial errors decrease overall, although the rates appear less uniform on these coarse meshes. 
The temporal errors also decrease monotonically. These results are intended only as self-convergence indicators, since the reference solutions are numerical and no convergence-rate result is proved here.


\appendix

\section{Auxiliary inequalities}\label{subsec:aux-estimates}

\begin{lemma}[The discrete Jensen inequality]\label{lem:weighted-discrete-holder}
	Let \(q\ge1\) and $k>0$. let \(\{a_j\}_{j=0}^{J-1}\) be a non-negative sequence and set $T=Jk$. Then
	\begin{equation}\label{eq:weighted-discrete-holder}
		\left(
		\sum_{j=0}^{J-1} k a_j
		\right)^q
		\le
		T^{q-1}
		\sum_{j=0}^{J-1} k a_j^q.
	\end{equation}
	The sign in \eqref{eq:weighted-discrete-holder} is reversed if $q\in (0,1)$.
\end{lemma}

\begin{lemma}[The discrete Burkholder--Davis--Gundy (BDG) inequality]\label{lem:discrete-BDG}
	Let $p\ge 1$, let $H$ be a Hilbert space, and let
	$\{\bff Z_j\}_{j=0}^{J-1}$ be an $H$-valued martingale difference sequence
	with respect to a filtration \(\{\mathcal F_j\}_{j=0}^{J}\), that is,
	\(\bff Z_j\) is \(\mathcal F_{j+1}\)-measurable and
	\[
	\mathbb E[\bff Z_j\mid\mathcal F_j]=\bff0,
	\qquad
	j=0,\ldots,J-1.
	\]
	Assume that \(\bff Z_j\in L^p(\Omega;H)\) for each \(j\). Then there exists
	a constant \(C_p>0\), depending only on \(p\), such that
	\begin{equation}\label{eq:discrete-BDG-square-function}
		\bb E\left[
		\max_{0\le r\le J}
		\norm{\sum_{j=0}^{r-1}\bff Z_j}{H}^{p}
		\right]
		\le
		C_p\,
		\bb E\left[
		\left(
		\sum_{j=0}^{J-1}\norm{\bff Z_j}{H}^{2}
		\right)^{p/2}
		\right].
	\end{equation}
\end{lemma}

\begin{lemma}[Asymptotic \(\bb L^2\)-consistency of $\mathcal R_h$]
	\label{lem:Maxwell reconstruction L2 consistency}
	Let \(\mathcal R_h:\bb{RT}_h^0\to\bb W_h^0\) be defined by
	\eqref{equ:Maxwell reconstruction scheme}. Then there exists a function $\varepsilon_{\mathcal R}:(0,1]\to[0,\infty)$ satisfying
	\[
	\varepsilon_{\mathcal R}(h)\to0
	\quad\text{as }h\to0,
	\]
	such that, for all \(\bff C_h\in\bb{RT}_h^0\),
	\begin{align}
		\label{equ:Rh L2 consistency asymptotic}
		\norm{\mathcal R_h\bff C_h-\bff C_h}{\bb L^2}
		\le
		\varepsilon_{\mathcal R}(h)
		\left(
		\norm{\bff C_h}{\bb L^2}
		+
		\norm{\curlh\bff C_h}{\bb L^2}
		+
		\norm{\divg\bff C_h}{L^2}
		\right).
	\end{align}
\end{lemma}

\begin{proof}
	For \(h>0\), define
	\[
	\varepsilon_{\mathcal R}(h)
	:=
	\sup_{\bff C_h\in\bb{RT}_h^0\setminus\{0\}}
	\frac{
		\norm{\mathcal R_h\bff C_h-\bff C_h}{\bb L^2}
	}{
		\norm{\bff C_h}{\bb L^2}
		+
		\norm{\curlh\bff C_h}{\bb L^2}
		+
		\norm{\divg\bff C_h}{L^2}
	}.
	\]
	By the stability estimate \eqref{equ:Maxwell reconstruction stability}, this
	quantity is finite. We prove that $\varepsilon_{\mathcal R}(h)\to 0$ as $h\to 0$.
	
	Suppose not. Then there exist \(\varepsilon_0>0\), a sequence \(h_k\to0\),
	and functions \(\bff C_k\in\bb{RT}_{h_k}^0\) such that
	\begin{align}
		\label{equ:eps contradiction normalization}
		\norm{\bff C_k}{\bb L^2}
		+
		\norm{\curl_{h_k}\bff C_k}{\bb L^2}
		+
		\norm{\divg\bff C_k}{L^2}
		=1,
	\end{align}
	but
	\begin{align}
		\label{equ:eps contradiction lower bound}
		\norm{\mathcal R_{h_k}\bff C_k-\bff C_k}{\bb L^2}
		\ge
		\varepsilon_0.
	\end{align}
	Set $\bff J_k:=\curl_{h_k}\bff C_k$ and $D_k:=\divg\bff C_k$.
	By \eqref{equ:eps contradiction normalization}, the sequences
	\(\{\bff C_k\}\), \(\{\bff J_k\}\), and \(\{D_k\}\) are bounded in
	\(\bb L^2\), \(\bb L^2\), and \(L^2\), respectively. By the discrete Maxwell
	compactness property, there exists a subsequence, not relabelled, and a field
	\(\bff C\) such that
	\[
	\bff C_k\to\bff C
	\qquad\text{strongly in }\bb L^2(\mathscr D).
	\]
	Moreover, up to a further subsequence,
	\[
	\bff J_k\rightharpoonup\bff J
	\quad\text{weakly in }\bb L^2(\mathscr D),
	\qquad
	D_k\rightharpoonup D
	\quad\text{weakly in }L^2(\mathscr D).
	\]
	
	We claim that $\bff J=\curl\bff C$ and $D=\divg\bff C$
	in the sense of distributions. The second identity follows immediately from
	\(D_k=\divg\bff C_k\), the strong convergence
	\(\bff C_k\to\bff C\) in \(\bb L^2\), and the weak convergence
	\(D_k\rightharpoonup D\) in \(L^2\).
	
	For the first identity, let $\bff\omega\in \hzerocurl$. By the approximation property of the space $\bb X_h^0$, there exists
	\(\bff\omega_{h_k}\in\bb X_{h_k}^0\) such that
	\[
	\bff\omega_{h_k}\to\bff\omega
	\qquad\text{strongly in } \hcurl.
	\]
	By the definition of $\curl_{h_k}$, we have
	\[
	\inpro{\bff J_k}{\bff\omega_{h_k}}
	=
	\inpro{\bff C_k}{\curl\bff\omega_{h_k}}.
	\]
	Passing to the limit gives $\inpro{\bff J}{\bff\omega}= \inpro{\bff C}{\curl\bff\omega}$.
	Therefore, $\bff J=\curl\bff C$.
	
	We can now apply Lemma~\ref{lem:conv reconstruction consistency} to the
	sequence \(\{\bff C_k\}\). Its hypotheses are satisfied with
	$\bff K=\curl\bff C$ and  $D=\divg\bff C$.
	Consequently,
	\[
	\mathcal R_{h_k}\bff C_k\to\bff C
	\qquad\text{strongly in }\bb L^2(\mathscr D).
	\]
	Since also \(\bff C_k\to\bff C\) strongly in \(\bb L^2(\mathscr D)\), we then have
	$\norm{\mathcal R_{h_k}\bff C_k-\bff C_k}{\bb L^2}\to0$,
	which contradicts \eqref{equ:eps contradiction lower bound}. Hence
	$\varepsilon_{\mathcal R}(h)\to0$, and thus
	\eqref{equ:Rh L2 consistency asymptotic} follows.
\end{proof}

\begin{lemma}[Negative-norm bound for the magnetic noise coefficient]
	\label{lem:magnetic noise negative coefficient}
	Let \(m\ge3\), and let \(\varepsilon_{\mathcal R}(h)\) be defined by
	Lemma~\ref{lem:Maxwell reconstruction L2 consistency}.  Set
	\begin{equation}\label{equ:vartheta}
		\vartheta_h:=h+\varepsilon_{\mathcal R}(h).
	\end{equation}
	Then we have
	\begin{align}
		\label{equ:magnetic negative coeff bound}
		\sum_{i=1}^{\infty}
		\norm{\curl\Pi_h^{\bb X}
			\bigl(
			\sigma\bff\chi_i
			+
			\eta\bff\chi_i\times\mathcal R_h\bff B_h^{n-1}
			\bigr)}{\widetilde{\bb H}^{-m}}^2
		\le
		C\left(
		1+\norm{\bff B_h^{n-1}}{\bb L^2}^{2}
		\right)
		+
		C\vartheta_h^2
		\norm{\bff J_h^{n-1}}{\bb L^2}^{2}.
	\end{align}
\end{lemma}

\begin{proof}
	Let $\bff\psi\in\bb H^m(\mathscr D)$ with
	$\norm{\bff\psi}{\bb H^m}\le1$. Set $\bff Z_{i,h}^{n-1}
	:=
	\sigma\bff\chi_i
	+
	\eta\bff\chi_i\times\mathcal R_h\bff B_h^{n-1}$ and $\bff H_{i,h}^{n-1}
	:=
	\Pi_h^{\bb X}\bff Z_{i,h}^{n-1}$.
	Since $\bff H_{i,h}^{n-1}\in\bb X_h^0\subset\hzerocurl$, integration by parts
	gives
	\[
	\inpro{\curl\bff H_{i,h}^{n-1}}{\bff\psi}
	=
	\inpro{\bff H_{i,h}^{n-1}}{\curl\bff\psi}.
	\]
	Using the $\bb L^2$-stability of $\Pi_h^{\bb X}$, we obtain
	\[
	\abs{
		\inpro{\curl\bff H_{i,h}^{n-1}}{\bff\psi}
	}
	\le
	\norm{\Pi_h^{\bb X}\bff Z_{i,h}^{n-1}}{\bb L^2}
	\norm{\curl\bff\psi}{\bb L^2}
	\le
	C\norm{\bff Z_{i,h}^{n-1}}{\bb L^2}.
	\]
	Now write
	\[
	\bff Z_{i,h}^{n-1}
	=
	\sigma\bff\chi_i
	+
	\eta\bff\chi_i\times\bff B_h^{n-1}
	+
	\eta\bff\chi_i\times
	\left(
	\mathcal R_h\bff B_h^{n-1}-\bff B_h^{n-1}
	\right).
	\]
	By Lemma~\ref{lem:Maxwell reconstruction L2 consistency}, since
	$\divg\bff B_h^{n-1}=0$ and
	$\Pi_h^{\bb{X}} \bff J_h^{n-1}=\curlh\bff B_h^{n-1}$, we have
	\[
	\norm{\mathcal R_h\bff B_h^{n-1}-\bff B_h^{n-1}}{\bb L^2}
	\le
	\varepsilon_{\mathcal R}(h)
	\left(
	\norm{\bff B_h^{n-1}}{\bb L^2}
	+
	\norm{\bff J_h^{n-1}}{\bb L^2}
	\right).
	\]
	Therefore, for $h$ sufficiently small,
	\[
	\norm{\bff Z_{i,h}^{n-1}}{\bb L^2}
	\le
	C\norm{\bff\chi_i}{\bb W^{1,\infty}}
	\left(
	1+\norm{\bff B_h^{n-1}}{\bb L^2}
	+
	\varepsilon_{\mathcal R}(h)
	\norm{\bff J_h^{n-1}}{\bb L^2}
	\right).
	\]
	Taking the supremum over
	$\bff\psi\in\bb H^m$ with $\norm{\bff\psi}{\bb H^m}\le1$, squaring, and
	summing in $i$, we get
	\[
	\sum_{i=1}^{\infty}
	\norm{\curl\bff H_{i,h}^{n-1}}{\widetilde{\bb H}^{-m}}^2
	\le
	C\left(
	1+\norm{\bff B_h^{n-1}}{\bb L^2}^{2}
	\right)
	+
	C\varepsilon_{\mathcal R}(h)^2
	\norm{\bff J_h^{n-1}}{\bb L^2}^{2},
	\]
	where we also used \eqref{equ:PiX negative curl stability}.
	This implies
	\eqref{equ:magnetic negative coeff bound}, as required.
\end{proof}

\section{Time-increment estimates and tightness of laws}

For a Banach space $E$, a continuous function $v\in C([0,T];E)$, and
$\delta>0$, define the modulus of continuity $\mathfrak m_E(v,\delta)$ by
\[ 
\mathfrak m_E(v,\delta)
:=
\sup_{\substack{s,t\in[0,T]\\ |t-s|\le\delta}}
\norm{v(t)-v(s)}{E}.
\]
In the following lemmas, we derive suitable time-increment estimates which are necessary to prove tightness of laws.

\begin{lemma}[Maximal block increment estimates]
	\label{lem:conv maximal block increments}
	Let $m\ge 3$, and let $\vartheta_h$ be defined by \eqref{equ:vartheta}.
	Assume the hypotheses of Proposition~\ref{prop:stoch stability revised} such that \eqref{equ:stoch stability revised main} holds. Then there exists $C>0$, independent of $h,\tau,a,b$, such that for every
	$0\le a\le b\le N$,
	\begin{align}
		\label{equ:maximal block u}
		\bb E
		\left[
		\max_{a\le n\le \ell\le b}
		\norm{\bff u_h^\ell-\bff u_h^n}{(\bb V^m_\sigma)'}^{4}
		\right]
		&\le
		C\left(((b-a)\tau)^2+\vartheta_h^4\right),
		\\
		\label{equ:maximal block B}
		\bb E
		\left[
		\max_{a\le n\le \ell\le b}
		\norm{\bff B_h^\ell-\bff B_h^n}{\widetilde{\bb H}^{-m}}^{4}
		\right]
		&\le
		C\left(((b-a)\tau)^2+\vartheta_h^4\right).
	\end{align}
\end{lemma}

\begin{proof}
	We first prove the velocity estimate, then indicate the corresponding
	argument for the magnetic field.
	Let $a\le n\le \ell\le b$. Repeating the argument leading to
	\eqref{equ:conv velocity increment bound}, with the sum taken from $r=n+1$ to $r=\ell$, gives
	\begin{align}
		\label{equ:block u pathwise bound}
		\norm{\bff u_h^\ell-\bff u_h^n}{(\bb V^m_\sigma)'}
		&\le
		Ch\left(
		\norm{\bff u_h^\ell}{\bb L^2}
		+
		\norm{\bff u_h^n}{\bb L^2}
		\right)
		+
		C\sum_{r=n+1}^{\ell}\tau\norm{\nabla\bff u_h^r}{\bb L^2}
		\nonumber\\
		&\quad
		+
		C\sum_{r=n+1}^{\ell}\tau
		\norm{\bff u_h^{r-1}}{\bb L^2}
		\norm{\nabla\bff u_h^r}{\bb L^2}
		+
		C\sum_{r=n+1}^{m}\tau
		\norm{\bff J_h^r}{\bb L^2}
		\norm{\bff B_h^{r-1}}{\bb L^2}
		\nonumber\\
		&\quad
		+
		C\sum_{r=n+1}^{\ell}\tau\norm{\bff f^r}{\bb H^{-1}}
		+
		C\sup_{\substack{\bff\phi\in\bb V_\sigma^m\\
				\norm{\bff\phi}{\bb H^m}\le1}}
		\abs{
			\sum_{r=n+1}^{\ell}\sum_{i=1}^{\infty}
			\inpro{\bff g_{i,h}^{r-1}}{\mathcal S_h\bff\phi}
			\Delta_r\beta_i^u
		}.
	\end{align}
	Taking the maximum over $a\le n\le \ell\le b$, the deterministic sums are bounded by the corresponding sums over the full block $r=a+1,\ldots,b$.
	Using Proposition~\ref{prop:stoch stability revised} and an elementary inequality
	\[
	\left(
	\sum_{r=a+1}^{b}\tau A_r
	\right)^4
	\le
	((b-a)\tau)^2
	\left(
	\sum_{r=a+1}^{b}\tau A_r^2
	\right)^2,
	\]
	the deterministic terms
	in \eqref{equ:block u pathwise bound} contribute at most
	\(C((b-a)\tau)^2\) after taking expectations. After taking the fourth power, the projection error gives
	\[
	C h^4
	\bb E
	\max_{0\le r\le N}\norm{\bff u_h^r}{\bb L^2}^{4}
	\le
	Ch^4
	\le
	C\vartheta_h^4.
	\]
	It remains to control the stochastic term. Define the $\bb L^2$-valued
	martingale
	\[
	\bff M_\ell^u
	:=
	\sum_{r=a+1}^{\ell}\sum_{i=1}^{\infty}
	\bff g_{i,h}^{r-1}\Delta_r\beta_i^u,
	\qquad \ell =a,\ldots,b.
	\]
	By \eqref{equ:Stokes bdd}, we have
	\[
	\sup_{\substack{\bff\phi\in\bb V_\sigma^m\\
			\norm{\bff\phi}{\bb H^m}\le1}}
	\abs{
		\sum_{r=n+1}^{\ell}
		\sum_{i=1}^{\infty}
		\inpro{\bff g_{i,h}^{r-1}}{\mathcal S_h\bff\phi}
		\Delta_r\beta_i^u
	}
	\le
	C\norm{\bff M_\ell^u-\bff M_n^u}{\bb L^2}.
	\]
	Moreover,
	\[
	\max_{a\le n\le \ell\le b}
	\norm{\bff M_\ell^u-\bff M_n^u}{\bb L^2}
	\le
	2\max_{a\le \ell\le b}
	\norm{\bff M_\ell^u-\bff M_a^u}{\bb L^2}.
	\]
	By the discrete Burkholder--Davis--Gundy inequality in $\bb{L}^2$ and the conditional Gaussian moment estimate,
	\begin{align}\label{equ:max M bdg}
		\bb E
		\left[
		\max_{a\le n\le \ell \le b}
		\norm{\bff M_\ell^u-\bff M_n^u}{\bb L^2}^{4}
		\right]
		&\le
		C
		\bb E
		\left[
		\sum_{r=a+1}^{b}\tau
		\sum_{i=1}^{\infty}
		\norm{\bff g_{i,h}^{r-1}}{\bb L^2}^{2}
		\right]^2
		\le
		C\bigl((b-a)\tau\bigr)^2.
	\end{align}
	
	For the magnetic field, the estimate leading to \eqref{equ:Bn H minus m sq}
	gives, for $a\le n\le \ell\le b$,
	\begin{align*}
		\norm{\bff B_h^\ell-\bff B_h^n}{\widetilde{\bb H}^{-m}}
		&\le
		C\sum_{r=n+1}^{\ell}\tau\norm{\bff J_h^r}{\bb L^2}
		+
		C\sum_{r=n+1}^{\ell}\tau
		\norm{\bff J_h^r}{\bb L^2}
		\norm{\bff B_h^{r-1}}{\bb L^2}
		\\
		&\quad
		+
		C\sum_{r=n+1}^{\ell}\tau
		\norm{\bff u_h^r}{\bb L^2}
		\norm{\bff B_h^{r-1}}{\bb L^2}
		+
		\norm{
			\sum_{r=n+1}^{\ell}\sum_{i=1}^{\infty}
			\curl\bff H_{i,h}^{r-1}\Delta_r\beta_i^B
		}{\widetilde{\bb H}^{-m}}.
	\end{align*}
	The deterministic terms are handled as above. 
	For the stochastic term, we define the \(\widetilde{\bb H}^{-m}\)-valued martingale
	\[
	\bff M_\ell^B
	:=
	\sum_{r=a+1}^{\ell}\sum_{i=1}^{\infty}
	\curl\bff H_{i,h}^{r-1}\Delta_r\beta_i^B,
	\qquad \ell=a,\ldots,b.
	\]
	Then, as before,
	\[
	\max_{a\le n\le \ell\le b}
	\norm{\bff M_\ell^B-\bff M_n^B}{\widetilde{\bb H}^{-m}}
	\le
	2\max_{a\le m\le b}
	\norm{\bff M_\ell^B-\bff M_a^B}{\widetilde{\bb H}^{-m}}.
	\]
	By the same argument as in \eqref{equ:max M bdg}, we obtain
	\begin{align*}
		\bb E
		\left[
		\max_{a\le n\le \ell\le b}
		\norm{\bff M_\ell^B-\bff M_n^B}{\widetilde{\bb H}^{-m}}^{4}
		\right]
		&\le
		C
		\bb E
		\left[
		\sum_{r=a+1}^{b}\tau
		\sum_{i=1}^{\infty}
		\norm{\curl\bff H_{i,h}^{r-1}}{\widetilde{\bb H}^{-m}}^{2}
		\right]^2
		\\
		&\le
		C
		\bb E
		\left[
		\sum_{r=a+1}^{b}\tau
		\left(
		1+\norm{\bff B_h^{r-1}}{\bb L^2}^2
		\right)
		+
		\vartheta_h^2
		\sum_{r=a+1}^{b}\tau
		\norm{\bff J_h^{r-1}}{\bb L^2}^2
		\right]^2
		\\
		&\le
		C((b-a)\tau)^2
		+
		C \vartheta_h^4
		\bb E
		\left[
		\sum_{r=1}^{N}\tau
		\norm{\bff J_h^{r-1}}{\bb L^2}^2
		\right]^2
		\\
		&\le
		C\left(((b-a)\tau)^2+ \vartheta_h^4\right).
	\end{align*}
	This proves \eqref{equ:maximal block B}, thus completing the proof of the lemma.
\end{proof}

\begin{lemma}[Stochastic equicontinuity in negative norms]
	\label{lem:stochastic equicontinuity negative}
	Let \(m\ge3\). Then, for every \(\eta>0\),
	\begin{align}
		\label{equ:stoch equicont u negative}
		\lim_{\delta\downarrow0}
		\limsup_{h,\tau\to0}\,
		\bb P\left(
		\mathfrak m_{(\bb V^m_\sigma)'}
		(\bff u_{h,\tau},\delta)>\eta
		\right)
		&=0,
		\\
		\label{equ:stoch equicont B negative}
		\lim_{\delta\downarrow0}
		\limsup_{h,\tau\to0}\,
		\bb P\left(
		\mathfrak m_{\widetilde{\bb H}^{-m}}
		(\bff B_{h,\tau},\delta)>\eta
		\right)
		&=0.
	\end{align}
\end{lemma}

\begin{proof}
	We prove \eqref{equ:stoch equicont u negative}. The proof of \eqref{equ:stoch equicont B negative} is analogous.
	
	Fix $0<\delta\le T$, and set $M_\delta:=\left\lceil T/\delta \right\rceil$.
	For $q=0,\ldots,M_\delta-1$, define
	\[
	I_q:=[q\delta,(q+1)\delta]\cap[0,T],
	\]
	so that $\{I_q\}_{q=0}^{M_\delta-1}$ covers $[0,T]$. Moreover, $M_\delta\le \frac{T}{\delta}+1\le \frac{2T}{\delta}$.
	Let
	\[
	I_q^\ast:=[(q-1)\delta,(q+2)\delta]\cap[0,T].
	\]
	If $s,t\in[0,T]$ and $|t-s|\le\delta$, then there exists
	$q\in\{0,\ldots,M_\delta-1\}$ such that $s,t\in I_q^\ast$.
	Now, write $I_q^\ast=[\alpha_q,\beta_q]$ and define
	\[
	a_q:=\max\left\{0,\left\lfloor\frac{\alpha_q}{\tau}\right\rfloor-1\right\},
	\qquad
	b_q:=\min\left\{N,\left\lceil\frac{\beta_q}{\tau}\right\rceil+1\right\}.
	\]
	Then every value $\bff u_{h,\tau}(t)$ with $t\in I_q^\ast$ is a convex
	combination of $\bff u_h^r$ with
	$r\in\{a_q,\ldots,b_q\}$. Moreover,
	\begin{equation}\label{equ:bq aq tau}
		(b_q-a_q)\tau
		\le
		(\beta_q-\alpha_q)+4\tau
		\le
		3\delta+4\tau
		\le
		4(\delta+\tau).
	\end{equation}
	Since $\bff u_{h,\tau}$ is affine on every interval $[t_{r-1},t_r]$, for
	$s,t\in I_q^\ast$ we may write
	\[
	\bff u_{h,\tau}(t)
	=
	\sum_{i=a_q}^{b_q}\alpha_i\bff u_h^i,
	\qquad
	\bff u_{h,\tau}(s)
	=
	\sum_{j=a_q}^{b_q}\beta_j\bff u_h^j,
	\]
	where $\alpha_i,\beta_j\ge0$ and $\sum_{i=a_q}^{b_q}\alpha_i= \sum_{j=a_q}^{b_q}\beta_j= 1$.
	Consequently,
	\[
	\sup_{\substack{s,t\in I_q^\ast\\ |t-s|\le\delta}}
	\norm{\bff u_{h,\tau}(t)-\bff u_{h,\tau}(s)}{(\bb V^m_\sigma)'}
	\le
	\max_{a_q\le n\le \ell\le b_q}
	\norm{\bff u_h^\ell-\bff u_h^n}{(\bb V^m_\sigma)'}.
	\]
	It follows by the union bound that
	\begin{align*}
		\bb P\left(
		\mathfrak m_{(\bb V^m_\sigma)'}(\bff u_{h,\tau},\delta)>\eta
		\right)
		&\le
		\sum_{q=0}^{M_\delta-1}
		\bb P\left(
		\max_{a_q\le n\le \ell\le b_q}
		\norm{\bff u_h^\ell-\bff u_h^n}{(\bb V^m_\sigma)'}
		>\eta
		\right).
	\end{align*}
	By the Chebyshev inequality, \eqref{equ:maximal block u}, and~\eqref{equ:bq aq tau}, we obtain
	\begin{align*}
		\bb P\left(
		\max_{a_q\le n\le \ell\le b_q}
		\norm{\bff u_h^\ell-\bff u_h^n}{(\bb V^m_\sigma)'}
		>\eta
		\right)
		&\le
		\frac{1}{\eta^4}\, \bb E
		\left[
		\max_{a_q\le n\le \ell\le b_q}
		\norm{\bff u_h^\ell-\bff u_h^n}{(\bb V^m_\sigma)'}^{4}
		\right]
		\\
		&\le
		C_\eta
		\left(((b_q-a_q)\tau)^2+\vartheta_h^4\right)
		\\
		&\le
		C_\eta
		\left((\delta+\tau)^2+\vartheta_h^4\right).
	\end{align*}
	Since $M_\delta\le 2T/\delta$, we have
	\[
	\bb P\left(
	\mathfrak m_{(\bb V^m_\sigma)'}(\bff u_{h,\tau},\delta)>\eta
	\right)
	\le
	C_{\eta,T}
	\frac{(\delta+\tau)^2+\vartheta_h^4}{\delta}.
	\]
	Taking first $\limsup_{h,\tau\to0}$, and using $\vartheta_h\to0$, then letting $\delta\downarrow0$ yields the claim \eqref{equ:stoch equicont u negative}.
	
	For the magnetic field, one uses \eqref{equ:maximal block B} and obtains
	\[
	\bb P\left(
	\mathfrak m_{\widetilde{\bb H}^{-m}}
	(\bff B_{h,\tau},\delta)>\eta
	\right)
	\le
	C_{\eta,T}
	\frac{(\delta+\tau)^2+\vartheta_h^4}{\delta}.
	\]
	The same limiting argument then proves
	\eqref{equ:stoch equicont B negative}.
\end{proof}

Tightness of laws can now be shown in the following lemma.

\begin{lemma}\label{lem:tightness negative path spaces}
	Let $m\ge3$ and suppose that \eqref{equ:stoch stability revised main} holds. For every sequence $(h_k,\tau_k)\to(0,0)$, the laws of
	$\{\bff u_{h_k,\tau_k}\}_{k\ge1}$ are tight on $C([0,T];(\bb{V}^m_\sigma)')$, and the laws of $\{\bff B_{h_k,\tau_k}\}_{k\ge1}$ are tight on $C([0,T];\widetilde{\bb H}^{-m})$.
\end{lemma}

\begin{proof}
	We use the compact containment and stochastic equicontinuity criterion for
	tightness in $C([0,T];E)$, where $E$ is a separable Hilbert
	space. In the real-valued case this is~\cite[Theorem~7.2--7.4]{Bil99}, where the proof is based on the Arzel\`a--Ascoli and Prokhorov theorems. It extends verbatim to separable Hilbert-valued paths once the scalar boundedness condition is replaced by compact containment. More precisely, for a family of $E$-valued continuous processes $\{X_{h,\tau}\}$, compact containment
	means that, for every $\varepsilon>0$, there exists a compact set
	$\mathcal K_\varepsilon\Subset E$ such that
	\[
	\liminf_{h,\tau\to0}
	\bb P\left(
	X_{h,\tau}(t)\in\mathcal K_\varepsilon
	\text{ for every }t\in[0,T]
	\right)
	\ge 1-\varepsilon .
	\]
	We verify compact containment for the velocity with
	$E=(\bb{V}^m_\sigma)'$. Since $\bb L^2 \Subset(\bb{V}^m_\sigma)'$,
	for every $R>0$ the set
	\[
	\mathcal K_R^u
	:=
	\overline{
		\left\{
		\bff v\in\bb L^2(\mathscr D):
		\norm{\bff v}{\bb L^2}\le R
		\right\}
	}^{\,(\bb{V}^m_\sigma)'}
	\]
	is compact in $(\bb{V}^m_\sigma)'$. By the stability estimate \eqref{equ:stoch stability revised main} and the Chebyshev inequality,
	\[
	\bb P\left(
	\sup_{t\in[0,T]}
	\norm{\bff u_{h,\tau}(t)}{\bb L^2}>R
	\right)
	\le \frac{C}{R^{2p}}.
	\]
	Choosing $R$ sufficiently large gives
	\[
	\bb P\left(
	\bff u_{h,\tau}(t)\in\mathcal K_R^u
	\text{ for every }t\in[0,T]
	\right)
	\ge 1-\frac{\varepsilon}{2},
	\]
	uniformly in $h$ and $\tau$. Hence, compact containment holds in
	$(\bb{V}^m_\sigma)'$. The stochastic equicontinuity condition in $(\bb{V}^m_\sigma)'$ is precisely~\eqref{equ:stoch equicont u negative}. Therefore the laws of
	$\{\bff u_{h,\tau}\}$ are tight in
	$C([0,T];(\bb{V}^m_\sigma)')$.
	The magnetic field is treated in an analogous manner.
\end{proof}



\end{document}